\documentclass[12pt]{article}
\input epsf.tex

\usepackage{graphicx}
\usepackage{amsmath,amsthm,amsfonts,amscd,amssymb,comment,eucal,latexsym,mathrsfs}
\usepackage{stmaryrd}
\usepackage[all]{xy}

\usepackage{epsfig}

\usepackage[all]{xy}
\xyoption{poly}
\usepackage{fancyhdr}
\usepackage{wrapfig}
\usepackage{epsfig}

\usepackage[pdftex]{hyperref}

\theoremstyle{plain}
\newtheorem{thm}{Theorem}[section]
\newtheorem{prop}[thm]{Proposition}
\newtheorem{lem}[thm]{Lemma}
\newtheorem{cor}[thm]{Corollary}

\theoremstyle{definition}
\newtheorem{defn}{Definition}
\theoremstyle{remark}
\newtheorem{remark}{Remark}

\newtheorem{notation}{Notation}

\advance \topmargin by -\headheight
\advance \topmargin by -\headsep
\evensidemargin \oddsidemargin
  \def\C{{\mathbb{C}}}  \def\E{{\mathbb{E}}}   \def\H{{\mathbb{H}}}     \def\M{{\mathbb{M}}} \def\N{{\mathbb{N}}}    \def\R{{\mathbb{R}}} \def\SS{{\mathbb{S}}}       \def\Z{{\mathbb{Z}}}

 \def\cB{{\mathcal{B}}} \def\cC{{\mathcal{C}}}    \def\cG{{\mathcal{G}}} \def\cH{{\mathcal{H}}}   \def\cK{{\mathcal{K}}} \def\cL{{\mathcal{L}}} \def\cM{{\mathcal{M}}}     \def\cR{{\mathcal{R}}} \def\cS{{\mathcal{S}}}  \def\cU{{\mathcal{U}}} \def\cV{{\mathcal{V}}}    

       \def\tcH{{\tilde{\mathcal{H}}}}

    \def\hE{{\widehat{E}}}                     

       \def\hphi{{\widehat{\phi}}}      \def\hmu{{\widehat{\mu}}}

       \def\tcH{{\tilde{\cH}}}

   \def\tmu{{\widetilde{\mu}}}    \def\tphi{{\widetilde{\phi}}}         

\newcommand{\G}{\Gamma}

\newcommand{\Si}{\Sigma}
\newcommand{\eps}{\epsilon}

\renewcommand\a{\alpha}
\renewcommand\b{\beta}
\renewcommand\d{\delta}

\renewcommand\k{\kappa}
\renewcommand\l{\lambda}

\newcommand\s{\sigma}

\newcommand\z{\zeta}

\newcommand{\Add}{\mathsf{Expand}}

\newcommand{\Bern}{\operatorname{Bern}}

\newcommand{\Cay}{\operatorname{Cay}}
\newcommand\Cocycle{\mathsf{Cocycle}}
\newcommand\cost{{\operatorname{Cost}}}

\newcommand\Cyl{\operatorname{Cyl}}

\newcommand\Expand{{\mathsf{Expand}}}

\newcommand\Lip{\operatorname{Lip}}

\newcommand\Meas{{\operatorname{Meas}}}

\newcommand\Pois{\operatorname{Pois}}
\newcommand\Prob{\operatorname{Prob}}

\newcommand\Radon{\mathsf{Radon}}

\newcommand\Ret{\mathtt{Ret}}

\newcommand\supp{\operatorname{supp}}

\newcommand\Stab{\operatorname{Stab}}

\newcommand\ssum{\operatorname{sum}}

\def\cc{{\curvearrowright}}
\newcommand{\resto}{\upharpoonright}

  \newcommand{\dee}{\textrm{d}}
  \newcommand{\rL}{\textrm{L}}

  \def\bfPi{{\mbox{\boldmath $\Pi$}}}

\begin{document}
\title{Metric criteria for fixed price of countable groups}
\author{Erin Bevilacqua and Lewis Bowen\\ University of Texas at Austin}
\maketitle

\begin{abstract}
We establish general criteria for a countable group $\G$ to have fixed price 1 depending on a choice of left-invariant proper metric on $\G$. We apply this criterion to show that if $\G_1,\G_2$ are two countable groups satisfying a certain growth condition then $\G_1\times \G_2$ has fixed price 1. For example, $\G\times \G$ has fixed price 1 for any countable group $\G$. 
\end{abstract}


\noindent
\tableofcontents

\section{Introduction}

Let $\Gamma$ be a countable (discrete) group. In the field of measured group theory, there is significant interest in studying the probability measure preserving actions of $\Gamma$ on standard probability spaces. An important invariant of these actions is the notion of \textit{cost}, originally due to Levitt \cite{MR1366313} and further refined by Gaboriau \cite{gaboriau-cost}. 

Suppose a countable group $\Gamma$ has an essentially free, probability measure preserving (pmp) Borel action on a standard Borel probability space $(X, \cB, \mu)$. If one looks at the orbits of this action, one gets a countable Borel equivalence relation (CBER) $\cR \subset X \times X$ on $X$:
$$\cR=\{(x,gx):~x\in X, g\in \G\}.$$
A \textit{graphing} $\cG \subset \cR$ of $\cR$ is a Borel subset which is symmetric (so $(x,y) \in \cG \Rightarrow (y,x)\in \cG$) and such that the connected components of the graph with vertex set $X$ and edge set $\cG$ are exactly the equivalence classes of $\cR$. In this case, the graph $\cG$ is said to \textit{generate} $\cR$.

The cost of $\cG$ is $\frac{1}{2}\int \deg_\cG(x)d\mu(x)$ where $\deg_\cG(x)$ is the number of edges $\{x,y\}$ with $(x,y)\in \cG$. The cost of an action is the infimum over the cost of all graphings generating the equivalence relation induced by the action. The cost of a group $\Gamma$ is the infimum of the cost of all essentially free, ergodic actions of $\G$ on standard probability spaces.

$\Gamma$ is said to have \textit{fixed price $c$} if all ergodic, essentially free actions of $\Gamma$ on a standard probability space $X$ have cost $c$. It was conjectured by Gaboriau \cite{gaboriau-cost} that all groups have fixed price. In his work, Gaboriau showed that several large classes of groups have fixed price, including direct products of countably infinite groups where at least one of the factor groups has an infinite subgroup with fixed price 1. However, the question has remained open for direct products of groups in which none of the factor groups has a fixed price 1 subgroup.

Abert-Weiss defined the max-cost of a group $\G$ to be the supremum cost of all essentially free pmp actions of $\G$ \cite{abert-weiss-2013}.  They used Kechris' theory of weak containment to show that Bernoulli actions achieve the max-cost. In this way, one can show that a group has fixed price when cost and max-cost agree. 





Abert and Mellick proved that locally compact, second countable groups of the form $G\times \Z$ where $G$ is compactly generated have fixed price 1 in \cite{MR4555892}. While the first explicit definition of calculating the cost of locally compact groups using cross sections can be found in \cite{MR4538560}, the technique used in Abert-Mellick uses Poisson point processes to find the maximal cost of actions. In this work, they showed that their technique was equivalent to the original definition using cross-sections, and sparked further interest in cost of locally compact groups.


A major breakthrough in proving fixed price 1 for lcsc groups came in the paper of Fr\c aczyk, Mellick, and Wilkens \cite{fraczyk2023poissonvoronoitessellationsfixedprice}, where they used the definition of cost from Abert-Mellick, the ideal Poisson-Voronoi tessellations defined in \cite{dachille2025idealpoissonvoronoitessellationshyperbolic}, and techniques from Lie theory to show that lattices in higher-rank Lie groups have fixed price 1. 
Mellick was able to further refine these techniques to get a more general result not depending on Lie theory in \cite{mellick2023gaboriauscriterionfixedprice}.

This paper refines the techniques of Fr\c aczyk, Mellick, and Wilkens to show that an even larger class of groups has fixed price 1. Moreover, we develop the theory from scratch: the reader need not be familiar with \cite{fraczyk2023poissonvoronoitessellationsfixedprice} to read this paper. 

A sample result is the following:
\begin{cor}\label{C:self}
    If $\G$ is any countable group, then $\G \times \G$ has fixed price 1.
\end{cor}
By contrast, it was not previously known whether $\G\times \G$ has fixed price 1 except in the special case in which $\G$ contains an infinite amenable subgroup.

\begin{remark}
    After this paper was nearly complete, we became aware of \cite{khezeli2025productsinfinitecountablegroups} which proves that general product groups have fixed price 1. It appears that from a big picture point-of-view, our method are similar. However, this paper develops more of the general theory but we do not obtain the full result for direct product groups. 
\end{remark}

We attempted to prove that $\G_1\times \G_2$ has fixed price 1 for any pair of countable groups $\G_1,\G_2$. But we only succeeded in the special case in which $\G_1$ and $\G_2$ have nice metrics with roughly comparable growth rates. To make this statement precise we need the next definitions.

Recall that a quasi-metric on a set $X$ is a function $d:X\times X \to [0,\infty)$ satisfying all of the conditions of a metric except that the triangle inequality is replaced with the quasi-triangle inequality:
$$d(x,z)\le d(x,y)+d(y,z)+C_q$$
where $C_q\ge 0$ is a constant not depending $x,y,z\in X$.

\begin{defn}
Let $d$ be an integer-valued quasi-metric on a countable group $\G$. We say 
\begin{itemize}
    \item $d$  is {\bf proper} if every ball of finite radius is finite;
    \item $d$ is {\bf left-invariant} if $d(gh,gf)=d(h,f)$ for all $f,g,h\in \G$;
    \item $d$ is {\bf $\eps$-approximately sub-additive} if there is an $\eps>0$ such that    if 
    $$SS(\G,n,\eps)=\{x\in \G:~d(x,e) \in [n-\eps,n+\eps]\}$$
is a spherical shell of mean radius $n$ and width $2\eps$ then for every $n,m\ge \eps$,
$$SS(\G,n,\eps)\cdot SS(\G,m,\eps) \supset S(\G,n+m)$$
where $S(\G,n+m)=\{x\in \G:~d(x,e)=n+m\}$ is the sphere of radius $n+m$. 
\end{itemize}
For example, if $d$ is a word metric then it satisfies all three criteria.
\end{defn}

\begin{thm}\label{T:main3-intro}
For $i=1,2$, let $d_i$ be a left-invariant proper integer-valued quasi-metric on a countable group $\G_i$ and $\eps>0$. Assume each $(\G_i,d_i)$ is $\eps$-approximately sub-additive. Let $\G=\G_1\times \G_2$. Let $d$ be the $\ell^1$ quasi-metric on $\G$:
$$d(x,y) = d_1(x_1,y_1)+d_2(x_2,y_2)$$
for $x=(x_1,x_2)\in \G_1\times \G_2$ and $y=(y_1,y_2)\in \G_1\times \G_2$. Assume for $i=1,2$
\begin{align}\label{E:balanced-intro}
    \lim_{n\to\infty}  \frac{\#B(\G_i,n)}{\#B(\G,n)} = 0
\end{align}
where $B(\G,n),B(\G_i,n) $ is the ball of radius $n$ in $\G$, $\G_i$ respectively (centered at the identity say). Then $\G$ has fixed price 1.
\end{thm}

\begin{remark}
    Intuitively, if \eqref{E:balanced-intro} does not hold then one of the groups has significantly faster growth rate than the other so that a horoball in the product group looks like a slightly thickened copy of a horoball in the faster-growth group. This explains why our technique does not work in the case: because our methods do not significantly differentiate between the faster growth group and $\G$ in this case.  
\end{remark}

\begin{remark}
  In Section  \ref{S:Product-groups}, we prove Proposition \ref{P:roughly-comparable} which, roughly speaking, states that if the growth rates of $(\G_1,d_1)$ and $(\G_2,d_2)$ are sufficiently close to each other then \eqref{E:balanced-intro} holds. So if the quasi-metrics are also approximately sub-additive then $\G=\G_1\times \G_2$ has fixed price 1.
\end{remark}




\subsection{Overview}

All of the results above follow from the following general theorem (whose terminology is discussed afterwards)

\begin{thm}\label{T:main-intro}
    If $\G$ is a discrete group with an infinite measure preserving action which is limit-amenable, partially doubly recurrent, and has normalized cost $p$ then $\G$ has max-cost at most $p$. In particular, if $p=1$ then $\G$ has fixed price 1.
\end{thm}

This result partially generalizes \cite{mellick2023gaboriauscriterionfixedprice}, where it is shown that if a locally compact unimodular group $G$ has a unimodular closed amenable subgroup $A$ such that $G$ acting on $(G/A)^2$ is conservative, then $G$ has fixed price 1 (and therefore, all lattices in $G$ have fixed price 1). Our result is only a partial generalization because we only work with countable groups. We have not attempted to extend our results to non-discrete locally compact groups.

To explain the similarities note that, in the setting of Mellick's Theorem, we may assume $G$ is non-amenable since otherwise it is known a folklore theorem (see \cite[Proposition 4.3]{MR3335405})  that $G$ has fixed price 1. So $G$ acting on $G/A$ is an infinite measure preserving action which is amenable and therefore limit-amenable. It is also doubly-recurrent (which is equivalent to $G \cc (G/A)^2$ being infinitely-conservative) and therefore partially doubly-recurrent (the latter condition is introduced in section \ref{S:cost-Poisson}). Because the action of $G$ on $G/A$ is amenable, it has normalized cost 1. 

Our paper begins by developing a notion of weak containment for infinite measure preserving actions. This extends Kechris' weak containment which is by now well-developed for pmp actions \cite{MR4138908}. An imp action is limit-amenable if it is weakly contained in the class of amenable imp actions.

Next we prove that if an imp action $\G \cc (X,\mu)$ is limit-amenable then the Poisson point process on $X$ determines an action of $\G$ which is weakly contained in Bernoulli. This follows from a more general theorem: the Poisson suspension functor preserves weak containment. 

Abert-Weiss proved in \cite{abert-weiss-2013} that if an essentially free action is weakly contained in a Bernoulli action then it realizes the max-cost. It follows that if an imp action $\G \cc (X,\mu)$ is limit-amenable then its Poisson suspension realizes the max-cost.

We are unaware of any definition of the cost of an imp action. So we define the normalized cost of an imp action to be the cost of the orbit-equivalence relation restricted to a complete section of measure 1. Using a result of Gaboriau, this is shown to be independent of the choice of complete section. For example, if the action is amenable then the normalized cost is 1.

Along the way, we prove that if a group $\G$ is exact, then all limit-amenable imp actions are amenable (see Theorem \ref{T:exact}). So in this case, the normalized cost is 1. 

Next we turn our attention to the following setting. Let $\G$ be a countable group with a proper left-invariant quasi-metric $d$. We assume $(\G,d)$ is approximately sub-additive (which occurs, for example, if $d$ is a word metric). The next definition is key to obtaining double recurrence:

\begin{defn}\label{D:onp}
    We say $(\G,d)$ satisfies the {\bf overlapping neighborhoods property} (ONP) if there exists a constant $C>0$ such that for all $m>0$, 
    \begin{align}\label{E:ONP}
\lim_{r\to\infty} \liminf_{n\to\infty}  \frac{\#\{(x,y) \in B(n)^2:~|B(r) \cap B(x,n+C) \cap B(y,n+C)| < m\}}{|B(n)|^2} =0.
\end{align}
where $B(x,n)$ denotes the closed ball of radius $n$ centered at $x\in \G$ and $B(r)=B(e,r)$ where $e\in \G$ is the identity.
\end{defn}

Our next general result is that if $(\G,d)$ has the ONP and is approximately sub-additive then there is a infinite $\G$-invariant measure on the space of horofunctions on $\G$ which is limit-amenable and doubly-recurrent. The precise statement is Theorem \ref{T:metric-group2}. This theorem is used to derive Theorem \ref{T:main3-intro} and thereby Corollary \ref{C:self}.


\subsection{Section guide}

Section \ref{S:weak-containment} generalizes weak containment to imp actions. We start with a definition in terms of partitions and prove an equivalent characterization in terms of vague convergence of infinite invariant measures on locally compact spaces. Section \ref{S:limit-amenability} introduces limit-amenability via vague convergence. By results in \S \ref{S:weak-containment}, an action is limit-amenable if and only if it is weakly contained in the class of amenable actions. This section proves the stronger statement that limit-amenable actions are actually limits of regular actions - which are actions measurably conjugate to the action of $\G$ on itself by left-translations. Section \ref{S:Poisson} proves that the Poisson suspension functor preserves weak containment.  
 
 Section \ref{S:cost} reviews cost, then introduces graph-cost and normalized cost. The main result is that if an imp action $\G \cc (X,\mu)$ is partially doubly recurrent, then the cost of its Poisson suspension is bounded by the normalized cost of $\G \cc (X,\mu)$. Section \ref{S:cost-Poisson} proves that cost of a Poisson suspension is bounded by normalized cost, if the action is partially doubly recurrent.

  Section \ref{S:general} proves Theorem \ref{T:main-intro}. The proof is relatively short because nearly all of the work has been done in previous sections.  Section \ref{S:Metric-groups} proves Theorem \ref{T:metric-group2} which produces actions of $\G$ which are doubly-recurrent and limit-amenable when $\G$ satisfies the overlapping neighborhoods principle and certain other metric properties.   Section \ref{S:Product-groups} proves Theorem \ref{T:main3-intro}. It then applies this theorem to prove several additional results, including Corollary \ref{C:self}.

 The appendix has three sections. Section \ref{S:recurrence} reviews Kaimanovich's generalization of Hopf's Decomposition Theorem. That is: we decompose $X$ into $\mathsf{Con}(X)$ where the action is infinitely conservative and $\mathsf{Dis}(X)$ where it is not.  Section \ref{S:mer} reviews standard notions regarding measured equivalence relations.    Section \ref{S:spaces-of-measures} reviews three notions of convergence of measures: weak*, weak and vague. One highlight is the Portmanteau Theorem for vague convergence of infinite measures on locally compact spaces, which we use frequently.

\subsection{Acknowledgements}
E. B. was partially supported by NSF Grant DMS--1937215; L. B. was partially supported by NSF Grants DMS--2154680 and DMS--2453399. We thank Mikolaj Fra\c czyk, Sam Mellick and Amanda Wilkens for helpful conversations which started this project. This paper will form part of E.B.'s thesis.

\section{Preliminaries}

Throughout this paper $\G$ denotes a countable (discrete) group. We will typically write a measure space as $(X,\mu)$ or $(Y,\nu)$ leaving the sigma-algebra implicit. We always assume our measures are $\s$-finite and standard and maps between measure spaces are assumed to be measurable unless otherwise stated. We use the abbreviations pmp, imp, mp, lcsc to mean probability-measure-preserving, infinite-measure-preserving, measure-preserving, locally compact second countable respectively.








\begin{defn}
For $i=1,2$, let $\G \cc (X_i,\mu_i)$ be imp actions. A {\bf factor map} is a measurable map $\phi:X'_1 \to X_2$ where 
\begin{itemize}
\item $X'_1 \subset X_1$ is a $\G$-equivariant co-null subset;
\item $g\phi(x) = \phi(gx)$ for a.e.~$x\in X_1'$ and every $g\in \G$;
\item $\phi_*(\mu_1 \resto X'_1)$ is equivalent to $\mu_2$.
\end{itemize}
In general, we do not require $\phi_*(\mu_1 \resto X'_1) = \mu_2$. We do not even require that  $\phi_*(\mu_1 \resto X'_1)$ is $\sigma$-finite.

If $\mu_1(\phi^{-1}(A))<\infty$ for every $A \subset X_2$ with $\mu_2(A)<\infty$ then we say $\phi$ is a {\bf finite-measure extension}.

If $\phi_*(\mu_1 \resto X'_1) = \mu_2$ then we say $\phi$ is {\bf measure-preserving}. If $\phi_*\mu_1$ is $\sigma$-finite and is absolutely continuous to $\mu_2$ then we say $\phi$ is {\bf quasi-measure-preserving}. Equivalently, $\phi$ is a factor map and the Radon-Nikodym derivative $\frac{d\phi_*\mu_1}{d\mu_2}$ is finite and positive on a $\mu_2$-conull set. 
\end{defn}

\begin{remark}\label{R:trivial}
As above, suppose $\G \cc (X_i,\mu_i)$ are imp actions and $\phi:X_1\to X_2$ is $\G$-equivariant and measurable. If $\mu_2 = \phi_*\mu_1$ then $\phi$ is automatically a finite measure extension. That is, measure-preserving factors are finite measure extensions.
\end{remark}





\section{Weak containment}\label{S:weak-containment}



The goal of this section is to define a notion of weak containment for imp actions which generalizes Kechris' definition. We generalize some of the standard tools by showing how weak containment relates to convergence of measures.

\begin{notation}
    We let $A$ denote a finite non-empty set of colors and $A_*=A\cup\{*\}$ where $*$ is a special element not in $A$. If $(X,\mu)$ is a measure space and $\phi:X\to A_*$ is measurable then $\phi$ is {\bf $(\mu,A)$-finite} if $\mu(\{x\in X:~ \phi(x)\ne *\})<\infty$. Sometimes we will simply say that $\phi$ is $A$-finite if $\mu$ is understood.
\end{notation}

\begin{defn}
    Let $\cC$ be a class of mp actions. We say that a given mp action $\a=(\G \cc (X,\mu))$ is \textbf{weakly contained} in $\cC$ (denoted $\a\prec\cC$) if for every finite set $A$, $(\mu,A)$-finite measurable map $\phi:X \to A_*$, finite $F \subset \G$ and $\eps>0$ there exist an imp action $\G \cc (Y,\nu)$ in $\cC$, and a $(\nu,A)$-finite measurable map $\psi:Y \to A_*$ such that  
    \begin{align*}
        \sum_{a\in A} \sum_{b \in A_*} \sum_{f\in F}& \Big|\mu\big(\{x\in X:~ \phi(x)=a \textrm{ and } \phi(fx)=b\}\big)\\
        &- \nu\big(\{y\in Y:~ \psi(y)=a \textrm{ and } \psi(fy)=b\}\big)\Big|<\eps.
    \end{align*}
    We will often be concerned with the special case in which $\cC=\{\beta\}$ in which case we say $\a$ is weakly contained in $\beta$, denoted $\a\prec\b$. We say two actions $\a$, $\b$ are \textbf{weakly equivalent} if each one weakly contains the other. This is denoted $\a\sim \b$.
\end{defn}

\begin{remark}
    In the special case of pmp actions, we can assume $\phi(x) \ne *$ for every $x\in X$, in which case the above definition reduces to Kechris' definition of weak containment \cite{kechris-2012}.
\end{remark}

\begin{remark}
   We typically assume that the identity element $e\in F$. The sum then includes the special case in which $a=b$ and $e\in F$ which implies
   \begin{align*}
    \sum_{a\in A} & \Big|\mu\big(\{x\in X:~ \phi(x)=a \}\big)- \nu\big(\{y\in Y:~ \psi(y)=a\}\big)\Big|<\eps.
\end{align*}
\end{remark}

\begin{defn}
    Let $\a=(\G\cc (X,\mu))$ and $\b=(\G\cc (Y,\nu))$ be measure-preserving actions on standard measure spaces. A {\bf measure-preserving factor map} from $\a$ to $\b$ is a $\G$-equivariant map $\Phi:X\to Y$ such that $\Phi_*\mu=\nu$. If such a factor map exists then we say $\b$ is an mp-factor of $\a$ or $\a$ is an mp-extension of $\b$. It is an exercise to check that if $\b$ is an mp-factor of $\a$ then $\b$ is weakly contained in $\a$.
\end{defn}




The notion of weak containment of pmp actions was introduced by Kechris as a counterpart to the classical notion of weak containment of unitary representations \cite{MR2931406}. See \cite{MR4138908} for a recent survey.

\subsection{Multi-correlations}

The definition of weak containment involves pair correlations. That is, we are concerned only with approximating measures of sets of the form $\{x\in X:~\phi(x)=a \textrm{ and } \phi(fx) = b\}$.
We show here that weak containment actually implies it is in fact possible to approximate more general correlations such as sets of the form
$\{x\in X:~\phi(x)=a_0, \phi(f_1x) = a_1 \textrm{ and } \phi(f_2x) = a_2\}$ for some given $a_0, a_1,a_2\in A$ and $f_1,f_2\in F$. We will prove a precise statement to this effect in terms of shift-spaces.

Let $K$ be a compact space and $K^\G$ the space of functions $x:\G \to K $ with the pointwise convergence topology. Let $\G$ act on $K^\G$ by
$$(gx)(f)=x(g^{-1}f) \quad \forall x\in K^\G, g,f \in \G.$$
This is called the {\bf shift-action}. 

Now let $\a=\G \cc (X,\mu)$ be a measure-preserving action (either finite or infinite). Given a measurable function $\phi:X \to K$ and a subset $F \subset \G$, define
$$\phi^F:X \to K^F$$
by
$$\phi^F(x)(f) = \phi(f^{-1}x).$$
In the special case $F=\G$, the map $\phi^\G:X \to K^\G$ is $\G$-equivariant. Therefore, the push-forward measure $\phi^\G_*\mu$ is $\G$-invariant. 

\begin{defn}
    Given a finite set $A$, a finite set $F \subset \G$ with $e\in F$ and a (real-valued) measure $\nu$ on $A_*^F$, define the pseudo-norm
    $$\|\nu\|_F = \sum \{ |\nu(a)|:~ a\in A_*^F \textrm{ and } a(e)\ne *\}.$$
    In other words, $\|\nu\|_F$ is the norm of the restriction of $\nu$ to the subset $\{a\in A_*^F:~a(e)\ne *\}$. 
    Since we often work with infinte measures for which this restriction is fine, it is useful to use the psedo-norm in place of $\|\nu\|$.
\end{defn}


\begin{thm}\label{T:general-correlations}
A given measure-preserving action $\a=\G \cc (X,\mu)$ is weakly contained in a class $\cC$ if and only if for every finite set $A$, $(\mu,A)$-finite measurable map $\phi:X \to A_*$, finite $F \subset \G$ and $\eps>0$ there exist an imp action $\G \cc (Y,\nu)$ in $\cC$, and a $(\nu,A)$-finite measurable map $\psi:Y \to A$ such that 
$$\|\phi^F_*\mu - \psi^F_*\nu\|_F < \eps.$$
\end{thm}

\begin{remark}
    To clarify, $\phi^F_*\mu$ is the measure on $A_*^F$ defined by
    $$\phi^F_*\mu(a) = \mu(\{x\in X:~\phi^F(x)=a\}$$  for $a\in A^F_*$. 
\end{remark}


\begin{proof} Assume $\G \cc (X,\mu)$ is weakly contained in a class $\cC$.    Let $\phi,A,F,\eps$ be as in the statement above. After replacing $F$ with a larger subset if necessary, we assume without loss of generality $e\in F$ and $F=F^{-1}$. 

We apply the definition of weak containment with $\phi^F$ in place of $\phi$ and $B=\{a\in A^F_*:~a(e) \ne *\}$ in place of $A$ to obtain the existence of an action $\G \cc (Y,\nu)$ in $\cC$, and a map $\kappa:Y\to B_*$ such that 
    \begin{align*}
    \sum_{a\in B} \sum_{b\in B_*} \sum_{f\in F}& \Big|\mu\big(\{x\in X:~ \phi^F(x)=a \textrm{ and } \phi^F(fx)=b\}\big)\\
    &- \nu\big(\{y\in Y:~ \kappa(y)=a \textrm{ and } \kappa(fy)=b\}\big)\Big|<\eps.
\end{align*}
Above, $B_*$ is the disjoint union of $B$ and $\{*\}$. The inequality above implies 
\begin{align}\label{E:part1}
\|\phi^F_*(\mu) - \kappa_*(\nu)\|_F < \eps,
\end{align}
where we have abused notation by identifying $\kappa_*(\nu)$ with its restriction to $A_*^F$. 

Define $\psi:Y \to A_*$ by 
\begin{displaymath}
    \psi(y) = \begin{cases}
        \kappa(y)(e), \quad & \kappa(y)\ne * \\
       *,  & \kappa(y)=*.
    \end{cases}
\end{displaymath}
We will see that $\psi^F$ is close to $\kappa$ off of a set of small measure.

Because $\phi^F(fx)(e)=\phi^F(x)(f^{-1})$ (for all $x\in X$ and $f\in F$), if $y\in Y$ is such that $\kappa(fy)(e)\ne \kappa(y)(f^{-1})$, then 
$$\mu\big(\{x\in X:~ \phi^F(x)=\kappa(y) \textrm{ and } \phi^F(fx)=\kappa(fy)\}\big)=0.$$
Therefore,
$$\nu(\{y \in Y:~\kappa(y)\ne * \textrm{ and } \exists f \in F \textrm{ such that } \kappa(fy)(e)\ne \kappa(y)(f^{-1})\})<\eps.$$
On the other hand,  if $y\in Y$ is such that $\kappa(y) \in A_*^F$ and  $\kappa(fy)(e)=\kappa(y)(f^{-1})$ for all $f$ then $\kappa(y)=\psi^F(y)$. Thus
$$\nu(\{y \in Y:~\kappa(y)\ne * \textrm{ and } \kappa(y)=\psi^F(y)\in A_*^F\}) \ge \nu(\{y \in Y:~\kappa(y)\ne * \})-\eps.$$
This implies
\begin{align}\label{E:part2}
 \|\kappa_*(\nu) - \psi^F_*(\nu)\|_F \le \eps. 
\end{align}
Again, we have abused notation by identifying $\kappa_*(\nu)$ with its restriction to $A_*^F$. 

By \eqref{E:part1} and \eqref{E:part2},
\begin{align*}
    \|\phi^F_*(\mu) - \psi^F_*(\nu)\|_F &\le \|\phi^F_*(\mu) - \kappa_*(\nu)\|_F + \|\kappa_*(\nu) - \psi^F_*(\nu)\|_F \\
    &\le 2 \eps.
\end{align*}
Since $\eps$ is arbitrary, this implies the first implication. The converse direction is trivial.
\end{proof}

\subsection{A universal system}

Let $2^{\G\times \N}$ be the space of all functions $x:\G\times \N \to \{0,1\}$ with the topology of pointwise convergence. This is a compact space on which $\G$ acts continuously by
$$(gx)(f,n)=x(g^{-1}f,n)$$
for $f,g\in \G$, $n\in \N$, $x\in 2^{\G\times \N}$. For $g\in \G$, let $x(g)\in 2^\N$ be the function which sends $n$ to $x(g,n)$.

Let $\cU=\{x\in 2^{\G\times \N}:~\forall g\in \G, x(g)\ne 0^\N\}$ where $0^\N\in 2^\N$ is the all zeros sequence. This is an lcsc space (with the subspace topology). It is also $\G$-invariant. Let $\Radon(\cU)$ be the set of Radon measures.


The action of $\G$ on $\cU$ is universal in the sense that it can model any given imp:
\begin{lem}\label{L:finiteness}
    Let $\G \cc (X,\mu)$ be an imp. Then there exists a $\G$-invariant Radon measure $\nu \in \Radon(\cU)$ such that $\G \cc (X,\mu)$ is measurably conjugate to $\G \cc (\cU, \nu)$.
\end{lem}

\begin{proof}
    Because $(X,\mu)$ is a standard $\s$-finite measure space, there exists a sequence $\{B_i\}_{i=1}^\infty$ of Borel subsets $B_i \subset X$ such that each $B_i$ has finite measure $\mu(B_i)<\infty$ and for all $x\in X$ there exist indices $i,j$ with $x \in B_i$ and $x\notin B_j$.

Define $\kappa:X \to 2^\N$ by 
$$\kappa(x)=\{i \in \N:~ x\in B_i\}$$
where we have identified $2^\N$ with the set of all subsets of $\N$. Define $\kappa^\G:X \to 2^{\G\times \N}$ by 
$$\kappa^\G(x)(g,n) = \kappa(g^{-1}x)(n).$$
Then $\kappa^\G$ is $\G$-equivariant, Borel and injective. Let $\nu=\kappa^\G_*\mu$. 

Next we prove that $\nu(K)<\infty$ for all compact $K \subset \cU$. Given a finite $F \subset \G$, $m\in \N$ and a map $\psi:F \times [m] \to \{0,1\}$, let $C(\psi)$ be the corresponding cylinder set:
$$C(\psi) = \{x\in 2^{\G\times \N}:~ x(g,i)=\psi(g,i) \forall g\in F, 1\le i \le m\}.$$
If there exists $(g,i)\in F\times [m]$ with $\psi(g,i)=1$ then $C(\psi)\subset \cU$ is compact. Moreover $\cU$ is a countable union of such compact sets. So it suffices to show $\nu(C(\psi))<\infty$ for all such $\psi$. This is clear because $(\kappa^\G)^{-1}(C(\psi)) \subset K_i$ if there exists $(g,i)$ with $\psi(g,i)=1$. This implies $\nu(C(\psi))\le \mu(K_i)<\infty$.

Because $\kappa^\G$ is Borel, $\nu$ is a Borel locally finite measure on a Polishable space. So it is Radon. 
\end{proof}


\subsection{Shift-space formulation}

\begin{defn}
   Given a measure-preserving action $\a$, let $\mathsf{Factor}(\a,\cU) \subset \Radon(\cU)$ be the set of all Radon measures of the form $\phi_*\mu$ where $\phi:X \to 2^\N$ is measurable  and $\phi^\G_*\mu(2^{\G\times \N}\setminus \cU)=0$. In detail, $\phi^\G:X \to 2^{\G\times \N}$ is the map $\phi^\G(x)(g)=\phi(g^{-1}x)$ where we regard an element of $2^{\G\times \N}$ as a function from $\G$ to $2^\N$. Also, because we require $\phi^\G_*\mu(2^{\G\times \N}\setminus \cU)=0$, we can regard $\phi^\G_*\mu$ as a measure on $\cU$.
   
   More generally, if $\cC$ is a class of measure-preserving actions, we define
$$\mathsf{Factor}(\cC,K)=\bigcup_{\a\in \cC} \mathsf{Factor}(\a,K).$$ 
\end{defn}

\begin{thm}\label{T:Abert-Weiss2}
Let $\a=\G \cc (X,\mu)$ be measure-preserving and let $\cC$ be a class of measure-preserving actions. 
\begin{enumerate}
    \item If $\a$ is weakly contained in $\cC$ then $\mathsf{Factor}(\a,\cU) \subset \overline{\mathsf{Factor}}(\cC,\cU)$ where the overline signifies closure in the vague topology.
    \item Suppose $X$ is a Polish space, $\G \cc X$ is a jointly continuous action, and $\mu$ is a Borel measure on $X$. If $\mu_n$ is a sequence of $\G$-invariant Borel measures on $X$ and either
    \begin{enumerate}
        \item each $\mu_n$ is a probability measure and $\mu_n \to \mu$ weakly as $n\to\infty$, or
        \item $X$ is lcsc and $\mu_n \to \mu$ vaguely as $n\to\infty$,
    \end{enumerate}
    then $\a$ is weakly contained in $\{\G \cc (X,\mu_n)\}_{n=1}^\infty$.
\end{enumerate}
\end{thm}

\begin{remark}
    The Theorem above generalizes \cite[Lemma 8]{abert-weiss-2013} (see also \cite[Prop. 3.6]{T-D12}) which handles the case of pmp actions.
\end{remark}



\begin{proof}[Proof of Theorem \ref{T:Abert-Weiss2}]

We first prove item (1). Making use of Lemma \ref{L:finiteness}, we assume without loss of generality that $\mu\in\Radon(\cU)$. Suppose that $\a=\G\cc (X,\mu)$ is weakly contained in $\cC$. 

Given a measure $\mu$ on $\cU$, a finite set $F \subset \G$ containing the identity and $N \in \N$, define the semi-norm $\|\mu\|_{F,N}$ by
$$\|\mu\|_{F,N} = \sum \left\{ |\mu(\Cyl(a))| :~ a \in 2^{F\times [N]} \textrm{ and } a(e)\ne 0_N\right\}$$
where
$$\Cyl(a)=\{x\in \cU:~x(f,i)=a(f,i) ~\forall (f,i) \in F\times [N]\}.$$

For $N\in\N$, let $\phi:\cU\to 2^N$ be the map $\phi(x)=(x(e,1),\ldots, x(e,N))$. Then $\cU$ is the inverse limit of the spaces $2^{F\times N}$ with respect to the maps $\phi^F$, where $(F,N)$ varies over all finite subsets of $\G$ and $N \in \N$. Therefore, it suffices to prove that for every $N \in \N$, finite $F \subset \G$ and $\eps>0$ there exist an action $\b = \G \cc (Y,\nu) \in \cC$ and a map $\tilde{\psi}:Y \to 2^\N$ such that 
$$\|\mu - \tilde{\psi}^F_*\nu\|_{F,N} < \eps$$
and $\tilde{\psi}^\G_*\nu(2^{\G\times \N}\setminus \cU)=0$.


By Theorem \ref{T:general-correlations}, for every $N \in \N$, finite $F \subset \G$ and $\eps>0$ there exist an action $\b = \G \cc (Y,\nu) \in \cC$, and a map $\psi:Y \to 2^N$ such that
$$\|\mu - \psi^F_*\nu\|_{F,N} < \eps.$$

By Lemma \ref{L:finiteness} we assume without loss of generality that $\nu\in\Radon(\cU)$ and in particular, $Y=\cU$. Thus an element $y\in Y$ is a function $y:\G\times \N\to\{0,1\}$. Define $\tilde{\psi}:Y \to 2^\N$ by
\begin{displaymath}
    \tilde{\psi}(y)(i) = 
    \begin{cases}
        \psi(y)(i), & 1\le i \le N \\
    y(e,i-N),\quad & i>N.
    \end{cases} 
\end{displaymath}
Then $\tilde{\psi}^\G:Y\to \cU$ is $\G$-equivariant and injective. Moreover, $\tilde{\psi}^\G_*\nu$ restricted to $2^{F\times [N]}$ is the same as $\psi^F_*\nu$ restricted to $2^N$. So
$$\|\mu - \tilde{\psi}^\G_*\nu\|_{F,N} < \eps.$$
Because the restriction of $x\in 2^{\G \times \N}$ to $\G \times \{N+1,N+2,\ldots \}$ is the inverse of $\tilde{\psi}^\G$, it follows that $\tilde{\psi}^\G_*\nu(2^{\G\times \N}\setminus \cU)=0$.

Because $F,N$ are arbitrary, this implies $\a\in \overline{\mathsf{Factor}}(\cC,\cU)$. The more general statement $\mathsf{Factor}(\a,\cU) \subset \overline{\mathsf{Factor}}(\cC,\cU)$ follows since we can replace $\a$ with a factor of $\a$ without changing the argument.

Next we prove item (2).


Now suppose $X$ is a Polish space, $\G \cc X$ a jointly continuous action, $\mu$ is a $\G$-invariant Borel measure on $X$ and $(\mu_n)_{n=1}^\infty$ is a sequence of $\G$-invariant Borel measures on $X$ converging to $\mu$ in the weak topology. We will show $\a=(\G \cc (X,\mu))$ is weakly contained in $\{\G \cc (X,\mu_n)\}_{n=1}^\infty$. So let  $\eps>0$, $F \subset \G$ be finite and $\phi:X \to A_*$ be $(\mu,A)$-finite (where $A$ is a finite set). 


Because continuity sets are dense in the measure-algebra, there exists a measurable map $\phi_1:X \to A_*$ such that
\begin{enumerate}
    \item $\mu(\{x\in X:~\phi(x)\ne \phi_1(x)\})<\frac{\eps}{(|A|+1)|F|}$;
    \item $\{x\in X:~\phi_1(x)=a\}$ is a $\mu$-continuity set for every $a\in A$;
    \item if $X$ is lcsc then we also require that $\{x\in X:~\phi_1(x)=a\}$ is pre-compact for every $a\in A$.
\end{enumerate}
Because the $\G$-action preserves continuity sets and continuity sets form an algebra, the last condition implies that
$$\{x\in X:~\phi_1(x)=a    \textrm{ and }   \phi_1(fx)=b\}$$
is a $\mu$-continuity set for every $a\in A$, $b\in A_*$ and $f\in F$. Moreover, it is pre-compact if $X$ is lcsc.

Then
\begin{align*}
\sum_{a\in A} \sum_{b\in A_*} \sum_{f\in F} &\big|\mu(\{x\in X:~\phi_1(x)=a \textrm{ and }   \phi_1(fx)=b\}) \\
&- \mu(\{x\in X:~\phi(x)=a \textrm{ and }   \phi(fx)=b \}) \big| < \eps.   
\end{align*}

We claim: 
\begin{align*}
\lim_{n\to\infty} \mu_n(\{x\in X:~\phi_1(x)=a \textrm{ and }   \phi_1(fx)=b \}) = \mu(\{x\in X:~\phi_1(x)=a \textrm{ and }   \phi_1(fx)=b \})
\end{align*}
for every $a\in A$, $b\in A_*$ and $f\in F$. If each $\mu_n$ is a probability measure and $\mu_n\to \mu$ weakly then this holds by the classical Portmanteau Theorem. If $X$ is locally compact and $\mu_n \to \mu$ vaguely then this holds by the version of the Portmanteau Theorem in \cite{MR2271177} applied to the 1-point compactification of $X$. This second case uses that each $\phi_1^{-1}(a)$ is pre-compact for all $a\in A$ as well as being a $\mu$-continuity set. 


In particular, there exist $N$ such that $n>N$ implies
\begin{align*}
\sum_{a\in A} \sum_{b\in A_*} \sum_{f\in F} &\big|\mu(\{x\in X:~\phi_1(x)=a \textrm{ and }   \phi_1(fx)=b\}) \\
&- \mu_n(\{x\in X:~\phi_1(x)=a \textrm{ and }   \phi_1(fx)=b \}) \big| < \eps.
\end{align*}
So the triangle inequality implies that for $n>N$,
\begin{align*}
\sum_{a\in A} \sum_{b\in A_*} \sum_{f\in F} &\big|\mu(\{x\in X:~\phi(x)=a \textrm{ and }   \phi(fx)=b\}) \\
&- \mu_n(\{x\in X:~\phi_1(x)=a \textrm{ and }   \phi_1(fx)=b \}) \big| < 2\eps.
\end{align*}
Since $\eps$ is arbitrary, this shows $\a$ is weakly contained in $\cC$.
\
\end{proof}

\section{Limit-amenability and regularity}\label{S:limit-amenability}

An imp action is called limit-amenable if it is weakly contained in the class of amenable imp actions. The main goal of this section is to prove that such actions are in fact limit-regular: they are weakly contained in the class of regular actions. This is used in the next section to prove that Poisson suspensions of such actions are weakly contained in Bernoulli shifts. First we need some definitions.


\begin{defn}
    An action $\G \cc (X,\mu)$ is {\bf regular} if it is measurably conjugate to the left-regular action of $\G$ on itself with respect to a Haar measure (which is a scalar multiple of counting measure). 
\end{defn}

\begin{defn}
    An imp action $\G \cc (X,\mu)$ is {\bf amenable} if for a.e.\ $x\in X$, the stabilizer $\Stab_\G(x)=\{g\in \G:~gx=x\}$ is amenable and the orbit equivalence relation $\cR_\G=\{(x,gx):~x\in X, g\in \G\}$ is hyperfinite mod $\mu$ (see \S \ref{S:mer} for definitions).
\end{defn}

\begin{remark}
    This is not the original definition introduced by Zimmer \cite{MR0470692}. However, it is equivalent. For many other equivalent definitions see \cite{zimmer-1984, MR1235474}.  
\end{remark}

\begin{defn}
Let $X$ be an lcsc space, $\G \cc X$ be a continuous action and $\mu$ be a $\G$-invariant Radon measure on $X$. Then we say the action $(\G,X,\mu)$ is {\bf limit-amenable} ({\bf limit-regular}) if there exists a sequence $(\mu_n)_{n=1}^\infty$ of $\G$-invariant Radon measures on $X$ such that 
\begin{enumerate}
    \item $\mu_n$ converges vaguely to $\mu$ as $n\to\infty$;
    \item the action $\G \cc (X,\mu_n)$ is a measure-preserving-factor of an amenable (regular) action for all $n$.
\end{enumerate}
We also say that $\G\cc (X,\mu)$ is limit-amenable (limit-regular) if it is measurably conjugate to a limit-amenable (limit-regular) action. By Theorem \ref{T:Abert-Weiss2}, an action is limit-amenable (limit-regular) if and only if it is weakly contained in the class of amenable (regular) actions.
\end{defn}

\begin{remark}[Permanence properties]\label{R:permanence}
    It is immediate from the definition that limit-amenability and limit-regularity are preserved under measure-preserving factors. It is also preserved under taking further limits in the following sense: suppose $\G \cc X$ is a continuous action on an lcsc space, $\mu$ is a $\G$-invariant measure on $X$ and there is a sequence $(\mu_n)_{n=1}^\infty$ of $\G$-invariant Radon measures on $X$ which converges vaguely to $\mu$ as $n\to\infty$. If $\G \cc (X,\mu_n)$ are all limit-amenable (limit-regular) then $\G \cc (X,\mu)$ is also limit-amenable (limit-regular).
\end{remark}

\begin{thm}\label{T:limit-regular}
An imp $\G \cc (X,\mu)$ is limit-amenable if and only if it is limit-regular. 
\end{thm}

\begin{remark}
  In Theorem \ref{T:exact} we prove: if $\G$ is exact then all limit-amenable actions are amenable. 
\end{remark}

To prove Theorem \ref{T:limit-regular}, we first prove a succession of lemmas culminating in Lemma \ref{L:regular3} which states that any imp action in which almost every ergodic component is essentially transitive with finite stabilizers is limit-regular. Then we prove the theorem in full by reducing it the case in which $\G \cc (X,\mu)$ is an essentially free,  amenable continuous action on an lcsc space and we need to find $\G$-invariant measures $\mu_n$ converging vaguely to $\mu$ which satisfy the hypotheses of Lemma \ref{L:regular3}.

\begin{lem}\label{L:regular1}
    Let $c_\G$ denote counting measure on $\G$ and let $\mathsf{Leb}$ denote Lebesgue measure on $[0,1]$. Then the action $\G \cc (\G\times [0,1], c_\G \times \mathsf{Leb})$ is limit-regular.
\end{lem}

\begin{proof}
First we construct a convenient topological model for $\G \cc (\G\times [0,1],c_\G\times \mathsf{Leb})$. Let $[0,1]_*$ be the disjoint union of $[0,1]$ with $\{*\}$. Consider $[0,1]_*^\G$ with the usual shift-action $(gx)(f)=x(g^{-1}f)$ for $x\in [0,1]_*^\G$ and $g,f \in \G$. Let $*^\G \in [0,1]_*^\G$ be the function which sends every $g\in \G$ to $*$. Let $\cV=[0,1]_*^\G\setminus \{*^\G\}$. This is a $\G$-invariant lcsc space. 

Define $\Phi:\G\times [0,1] \to \cV$ by 
\begin{displaymath}
    \Phi(g,x)(h) =
    \begin{cases}
        x, & h=g \\
        *, & h \ne g.
    \end{cases}
\end{displaymath}
Then $\Phi$ is $\G$-equivariant and injective where $\G$ acts on $\G\times [0,1]$ by $g(h,x)=(gh,x)$. Let $\mu_\infty = \Phi_*(c_\G \times \mathsf{Leb})$. Then $\Phi$ is a measure-conjugacy from $\G \cc (\G\times [0,1],c_\G\times \mathsf{Leb})$ to $\G \cc (\cV,\mu_\infty)$. So it suffices to show that $\G \cc (\cV,\mu_\infty)$ is limit regular.

Next, we construct a sequence of $\G$-invariant measures $\mu_n$ which converge to $\mu_\infty$ vaguely. Let $\{F_n\}_{n=1}^\infty, \{\G_n\}_{n=1}^\infty$ be sequences of finite subsets of $\G$ satisfying
\begin{enumerate}
    \item $e\in F_1\subset F_2\subset \cdots \subset \G$, and $\G  = \cup_n F_n$;
    \item $F_n=F_n^{-1}$;
    \item $|\G_n|=n$ and $\G_n$ is $F_n^2$-separated. This means: if $g,h \in \G_n$ are distinct then there does not exist an element $f\in F_n^2$ with $gf=h$. 
\end{enumerate}
We arbitrarily enumerate $\G_n$ as $\G_n=\{g_{n,1},\ldots, g_{n,n}\}$. Define $\phi_n \in \cV$ by
\begin{displaymath}
    \phi_n(h) =
    \begin{cases}
        k/n & h=g_{n,k} \\
        * & h \notin \G_n.
    \end{cases}
\end{displaymath}
Finally, define a measure $\mu_n$ on $\cV$ by
$$\mu_n = (1/n)\sum_{g\in \G} \d_{g\phi_n}.$$
That is, $\mu_n$ is the sum of the Dirac masses on the orbit of $\phi_n$. Therefore $\mu_n$ is $\G$-invariant and is a factor of the regular action $\G \cc (\G, c_\G/n)$. So it suffices now to prove that $\mu_n$ converges vaguely to $\mu_\infty$.

Let $\cK_n$ be the set of all $x\in \cV$ such that there exists $f \in F_n$ with $x(f)\ne *$. Then $\cK_n$ is compact. Moreover, it is a continuity set for all measures because its boundary is empty.

The inverse image $\Phi^{-1}(\cK_n)$ is $F_n \times [0,1]$. So the restriction of $\mu_\infty$ to $\cK_n$ is pushforward $\Phi_*(c_{F_n} \times \mathsf{Leb})$ where $c_{F_n}$ denotes counting measure on $F_n$.

Let $\pi_{F_n}: \cK_n \to [0,1]_*^{F_n}$ be the projection map. It suffices to prove that $\pi_{F_n*}(\mu_m\resto \cK_n)$ converges to $\pi_{F_n*}(\mu_\infty\resto \cK_n)$ as $m\to\infty$ and $n$ is held fixed. 




Let $m\ge n$. We claim that $\pi_{F_n*}(\mu_m\resto \cK_n) = \pi_{F_n*}(\Phi_*(c_{F_n} \times u_m)\resto \cK_n)$ where $u_m$ is the uniform probability measure on $\{1/m,2/m,\ldots, 1\} \subset [0,1]$. Because $u_m$ converges to Lebesgue measure on $[0,1]$ as $m\to\infty$, this claim implies the lemma.

First observe that for any $g\in \G$, $g\phi_m$ has the property that there is at most one element $f\in F_n$ with $g\phi_m(f)\ne *$. Indeed, suppose $f_1,f_2 \in F_n$ are distinct and $g\phi_m(f_i)\ne *$ for $i=1,2$. Since $g\phi_m(f_i)=\phi_m(g^{-1}f_i)$ this implies $g^{-1}f_i \in \G_m$. Because $\G_m$ is $F_m^2$-separated (and therefore, is $F_n^2$-separated), this implies $f_1^{-1}f_2 \notin F_n^2$. Therefore, it is not possible for both $f_1$ and $f_2$ to be in $F_n$.

Next, fix $f\in F_n$ and consider the set $\cL_f$ consisting of all $x\in \cV$ with $x(f)\ne *$. Then $\cK_n = \cup_{f\in F_n} \cL_f$. The previous paragraph show that if $f_1\ne f_2$ and $m\ge n$ then $\mu_m(\cL_{f_1} \cap \cL_{f_2})=0$. So it suffices to show that if we restrict $\mu_m$ to $\cL_f$ and then project to $[0,1]$ via the map which sends $x\in \cL_f$ to $x(f)$, the resulting measure is $u_m$.  
Note that $g\phi_m \in \cL_f$ if and only if $g\phi_m(f)\ne *$ which occurs if and only if $g^{-1}f \in \G_m$, if and only if $g \in f\G_m^{-1} = f\{g_{m,1}^{-1},\ldots, g_{m,m}^{-1}\}$. The measure $\mu_m$ restricted to $\cL_f$ and projected to $[0,1]$ is $1/m$ times counting measure on $\{1/m,2/m,\ldots, 1\}$. This is $u_m$. This finishes the proof.
\end{proof}

\begin{lem}\label{L:regular2}
    Let $c_\G$ denote counting measure on $\G$ and let $(Z,\zeta)$ be a standard $\s$-finite measure space on which $\G$ acts trivially.  Then the diagonal action $\G \cc (\G\times Z, c_\G \times \zeta)$ is limit-regular.
\end{lem}

\begin{proof}
The limit-regular property is invariant under scalars in the sense that if $\G \cc (X,\mu)$ is limit regular and $0<t<\infty$ then $\G \cc (X,t\mu)$ is also limit regular. This is because if $\mu_n$ converges vaguely to $\mu$ then $t\mu_n$ converges vaguely to $t\mu$. So Lemma \ref{L:regular1} implies that $\G \cc (\G\times [0,n], c_\G \times \mathsf{Leb}_{[0,n]})$ is limit regular for every $n>0$. 

 Because the measures $c_\G\times \mathsf{Leb}_{[0,n]}$ converge vaguely to $c_\G\times \mathsf{Leb}_{[0,\infty)}$, it follows that $\G \cc (\G\times [0,\infty), c_\G\times \mathsf{Leb}_{[0,\infty)})$ is limit-regular. 
 
Now let $(Z,\zeta)$ be an arbitrary standard $\sigma$-finite measure space. Then there exists $n\in [0,\infty]$ and a Borel map $\phi:[0,n) \to Z$ such that $\zeta = \phi_*\mathsf{Leb}_{[0,n)}$. Let $\Phi:\G \times [0,n] \to \G \times Z$ be the factor map $\Phi(g,x)=(g,\phi(x))$. Then $$\Phi_*(c_\G\times \mathsf{Leb}_{[0,n)})=c_\G\times\zeta.$$
Because limit-regularity is closed under measure-preserving factors, this implies the lemma. 
\end{proof}

\begin{defn}
    An action $\G \cc (X,\mu)$ is {\bf essentially transitive with finite stabilizers} if there is an $x\in X$ such that $\mu(X \setminus \G x)=0$ and $\Stab_\G(x)$ is finite.
\end{defn}

\begin{lem}\label{L:regular3}
    Let $\a=\G \cc (X,\mu)$ be an imp with the property that a.e.\ ergodic component is essentially transitive with finite stabilizers. Then $\a$ is limit-regular.
\end{lem}

\begin{proof}
Let $\cC$ be a collection of finite subgroups of $\G$ representing the conjugacy classes of finite subgroups of $\G$. That is, if $H \le \G$ is finite then there is a unique $H_0 \in \cC$ such that $H$ is conjugate to $H_0$.

For $H\in \cC$, let $X_H$ be the set of all $x\in X$ with stabilizer conjugate to $H$. Ignoring a measure zero set we have that $X$ is the disjoint union of $X_H$ over all $H\in \cC$. 

Because the action $\a$ is such that a.e.\ ergodic component is essentially transitive with finite stabilizers, there are standard $\s$-finite measure spaces $(Z_H,\zeta_H)$  for $H\in \cC$ such that $\a$ is measurably conjugate to the disjoint union of $\G \cc (\G/H\times  Z_H, c_{\G/H} \times \zeta_H)$ over $H\in \cC$. So without loss of generality we may assume $(X,\mu)$ is the direct sum of the measure spaces  $(\G/H \times Z_H,c_{\G/H}\times \zeta_H)$ over $H \in \cC$.

Let $(Z,\zeta)$ be the direct sum of the measure spaces $(Z_H,\zeta_H)$. By Lemma \ref{L:regular2}, the action $\G \cc (\G\times Z,c_\G\times \zeta)$ is limit-regular. 

Define the factor map $\Phi:\G \times Z \to X$ by $\Phi(g,z)=(gH,z)$ if $z\in Z_H$. Then $\Phi_*(c_\G\times \zeta)=\mu$. Because limit-regularity is closed under measure-preserving factors, this implies the lemma.
\end{proof}



\begin{proof}[Proof of Theorem \ref{T:limit-regular}]

Since every regular action is amenable, all limit-regular actions are limit-amenable. So it suffices to prove the converse. Since every limit-amenable action is a limit of amenable actions, we may assume $\a=G\cc (X,\mu)$ is amenable and prove that it is limit-regular. 

Let $\G \cc (Z,\zeta)$ be an essentially free pmp action. Then the product action $\G \cc (X\times Z,\mu\times \zeta)$ is essentially free. It is also amenable because extensions of amenable actions are amenable. If the product action is limit regular then the original action is limit regular. This is because limit-regularity is preserved under measure-preserving factor maps. So without loss of generality, we may assume that our original action $\G \cc (X,\mu)$ is amenable and essentially free.

Let $A$ be a finite set and $\phi:X \to A_*=A\cup\{*\}$ be a measurable map with $0<\mu(\phi^{-1}(A))<\infty$. Define $\phi^\G:X \to A_*^\G$ by
$$\phi^\G(x)(g)=\phi(g^{-1}x).$$
Then $\phi^\G$ is $\G$-equivariant. Moreover, $\phi^\G_*\mu$ is a Radon measure on $A_*^\G \setminus \{*^\G\}$ where $*^\G \in A_*^\G$ is the constant map $*^\G(g)=*$ for all $g\in \G$.

By Lemma \ref{L:regular3} and Remark \ref{R:permanence} it suffices to show there are $\G$-invariant Radon measures $\mu_n$ on $A_*^\G \setminus \{*^\G\}$ satisfying for all $n$:
\begin{enumerate}
    \item almost every ergodic component of $\G \cc (A_*^\G,\mu_n)$ is essentially transitive with finite stabilizers;
        \item $\mu_n(\{x\in A_*^\G:~ x(e) \ne *\})<\infty$;
    \item $\mu_n$ converges vaguely to $\phi^\G_*\mu$  as $n\to\infty$.
\end{enumerate}

Let $\cR=\{(x,gx):~x\in X, g\in \G\}$ be the orbit-equivalence relation. Because the action $\G \cc (X,\mu)$ is amenable, $\cR$ is $\mu$-hyperfinite. So there exist measurable equivalence relations $\cR_1 \subset \cR_2 \subset \cdots \subset \cR$ such that $\cR=\cup_n \cR_n$ and for a.e.\ $x\in X$, the $\cR_n$-class of $x$ has size $2^n$. 

Define $\Phi_n: X \to A_*^G$ by 
$$\Phi_{n}(y)(h) = \phi(h^{-1}y)=\phi^\G(y)(h)$$
if $(y,h^{-1}y)\in \cR_n$. Let $\Phi_{n}(y)(h)=\ast$ otherwise. Because each $\cR_n$ class has size $2^n$,  $\Phi_{n}(y)(h)\in A$ for at most $2^n$ values of $h$. 

The map $\Phi_n$ is not $\G$-equivariant. However, it is  $\cR_n$-invariant in the sense that if $h\in \G$ and $(y,hy) \in \cR_n$ then $\Phi_{n}(hy)=h\cdot \Phi_{n}(y)$.


Let $X_n$ be the set of all $x\in X$ such that there exist $y$ with $(x,y)\in \cR_n$ and $\phi(y)\in A$. Since $\cR_n$ increases to $\cR$, $X_n$ increases to $X$. 

Let 
$$\mu_n = 2^{-n} \sum_{g\in \G} g_*\Phi_{n*}(\mu\resto X_n)$$
where $\mu \resto X_n$ is the restriction of $\mu$ to $X_n$. Let us record some basic facts about $\mu_n$:

{\bf Fact \#1}. We claim $\mu_n(\{*^\G\})=0$. By definition, for $\mu_n$-a.e.\ $y$, there exists $g\in \G$ and $x\in X_n$ such that $y=g\Phi_n(x)$. Because $x\in X_n$, there exists $h\in \G$ such that $(x,h^{-1}x)\in \cR_n$ and $\phi(h^{-1}x)\in A$. Therefore, 
\begin{align*}
    y(gh) = g\Phi_n(x)(gh) = \Phi_n(x)(h) = \phi(h^{-1}x) \in A.
\end{align*}
This proves $y \ne *^\G$. So $\mu_n(\{*^\G\})=0$ as claimed.

{\bf Fact \#2}.  Observe that $\mu_n$ is $\G$-invariant. Indeed if $\nu$ is any measure on $A_*^\G$ then $\sum_{g\in \G} g_*\nu$ is $\G$-invariant.

{\bf Fact \#3}. We claim $\mu_n$ is a Radon measure on $A_*^\G \setminus \{*^\G\}$. We will also show that  $\mu_n(\{x\in A_*^\G:~ x(e) \ne *\})<\infty$.

For $k\in \G$, let $Z_k = \{z\in A_*^\G:~z(k)\ne *\}$. Then $A_*^\G=\cup_{k\in \G} Z_k$. So it suffices to prove $\mu_n(Z_k)<\infty$ for every $k$. Observe that $Z_k=kZ_e$. Since $\mu_n$ is $\G$-invariant, $\mu_n(Z_k)=\mu_n(Z_e)$. So
\begin{align*}
   \mu_n(Z_k)&=\mu_n(Z_e)= \mu_n(\{z\in A_*^\G:~z(e)\ne *\})\\
   &= 2^{-n} \sum_{g\in \G} \mu(\{x\in X_n:~(g\Phi_n(x))(e)\ne *\}) \\
   &=2^{-n} \int_{X_n} \#\{g\in \G:~(g\Phi_n(x))(e)\ne *\}~\dee\mu(x).
\end{align*}
Next we observe that if $x\in X_n$ and $g\in \G$ then $(g\Phi_n(x))(e)=\Phi_n(x)(g^{-1})$ is not equal to $*$ if and only if $(x,gx) \in \cR_n$ and $\phi(gx)\ne *$. Since the $\cR_n$-class of $x$ has cardinality $2^n$, 
$$\#\{g\in \G:~(g\Phi_n(x))(e)\ne *\}\le 2^n.$$
Therefore, $\mu_n(Z_k) \le \mu(X_n)$. It suffices now to prove $\mu(X_n)<\infty$. 

For this, let $F:\cR_n \to \R$ be defined by $F(x,y)=1$ if $\phi(x)\in A$ and $F(x,y)=0$ otherwise. By the Mass Transport Principle,
\begin{align*}
    \int \sum_y F(x,y)~\dee\mu(x) &= \int \sum_y F(y,x)~\dee\mu(x). 
\end{align*}
The left hand side is $2^n \mu(X_0)$ (where $X_0=\{x\in X:~\phi(x)\in A\}$). Indeed if $x\in X_0$ then $\sum_y F(x,y)=2^n$ and if $x\notin X_0$ then $\sum_y F(x,y)=0$.

The right hand side is at least $\mu(X_n)$. Indeed, if $x\in X_n$ then  $\sum_y F(y,x) \ge 1$. So we obtain $2^n \mu(X_0) \ge \mu(X_n)$, which in particular, implies $\mu(X_n)<\infty$ as claimed.

{\bf Fact \#4}. We claim that almost every ergodic component of $\G \cc (A_*^\G,\mu_n)$ is essentially transitive with finite stabilizers. To see this, for $x\in X_n$, let $\nu_x$ be the measure
$$\nu_x = 2^{-n} \sum_{g\in \G} g_*\d_{\Phi_n(x)}.$$
Then the action $\G \cc (A_*^\G,\nu_x)$ is conjugate to the action $\G \cc \G/\Stab_\G(\Phi_n(x))$. The stabilizer $\Stab_\G(\Phi_n(x))$ is finite because the set $\{g\in \G:~x(g)\ne *\}$ is finite and non-empty. Finally, observe that
$$\mu_n = \int \nu_x ~d(\mu \resto X_n)$$
is the ergodic decomposition of $\G \cc (A_*^\G,\mu_n)$. 


{\bf Fact \#5}. Lastly, we claim that $\mu_n$ converges vaguely to $\phi^\G_*\mu$ as $n\to\infty$. First we prove convergence on cylinder sets. So let $W \subset \G$ be a finite set containing the identity such that $W=W^{-1}$. Fix a coloring $\chi:W \to A_*$ with $\chi(e)\in A$. Let $C \subset A_*^\G$ be the cylinder set
$$C=\{x\in A_*^\G:~x(w)=\chi(w)~\forall w \in W\}.$$
We claim that $\lim_{n\to\infty} \mu_n(C) = \phi^\G_*\mu(C)$.

Define $F:\cR_n\to[0,1]$ by $F(x,y)=2^{-n}$ if $(x,y)\in \cR_n$, $y\in X_n$ and $\Phi_n(y)\in C$ and $F(x,y)=0$ otherwise. By the Mass Transport Principle,
\begin{align*}
    \int \sum_{y\cR_n x} F(x,y)~\dee\mu(x) &= \int \sum_{y\cR_n x} F(y,x)~\dee\mu(x). 
\end{align*}
The right-hand side equals $\mu(\{x\in X_n:~ \Phi_n(x)\in C\})$. Indeed if $x\in X_n$ and $\Phi_n(x)\in C$ then $\sum_y F(y,x)=1$. Otherwise, $\sum_y F(y,x)=0$.

So
\begin{align*}
 \mu(\{x\in X_n:~ \Phi_n(x)\in C\}) &=   2^{-n}\int \#\{g\in\G:~ gx\cR_n x \textrm{ and } g\Phi_n(x) \in C\}~\dee(\mu\resto X_n)(x)\\
 &= \mu_n(C).
\end{align*}
It follows that
\begin{align*}
    |\mu_n(C)  - \phi^\G_*\mu(C)|  &\le \mu(\{x\in X_n:~ \Phi_n(x)\in C, \phi^\G(x)\notin C\}) + \mu(\{x\in X_n:~ \Phi_n(x)\notin C, \phi^\G(x)\in C\}).
\end{align*}
Let
\begin{align*}
   Y_n&=\{x\in X:~wx \cR_n x ~\forall w \in W\} \\
   K&=\{x\in X:~\exists w\in W \textrm{ such that }\phi(wx)\ne *\}.
\end{align*}
Observe that if $x\notin K$ then $\Phi_n(x)\notin C$ and $\phi^\G(x)\notin C$. On the other hard, if $x\in Y_n$ then $\Phi_n(x)(w)=\phi^\G(x)(w)$ for all $w\in W$ and therefore, $\Phi_n(x)\in C$ if and only if $\phi^\G(x)\in C$. These observations imply
\begin{align*}
    |\mu_n(C)  - \phi^\G_*\mu(C)|  &\le \mu(\{x\in (X_n \cap K) \setminus Y_n\}).
\end{align*}
However, $Y_n \cap K$ increases to $K$ as $n\to\infty$ because $\cR_n$ increases to $\cR$. Since $K$ has finite $\mu$-measure, it follows that
\begin{align*}
 \lim_{n\to\infty}   |\mu_n(C)  - \phi^\G_*\mu(C)| =0.
\end{align*}
The sigma-algebra of $A_*^\G$ is generated by the action of $\G$ and cylinder sets of the same form as $C$. So this proves that $\mu_n$ converges vaguely to $\phi^\G_*\mu$, which completes the proof of the theorem.
\end{proof}



    

\subsection{Limit amenability and ergodic decomposition}

The main result of this subsection is:
\begin{thm}\label{C:limit-amen-erg-decomp2}
    An imp action $\G \cc (X,\mu)$ is limit-amenable if and only if a.e.\ ergodic component of the action is limit-amenable.
\end{thm}

We begin by proving the backwards direction. 
\begin{lem}\label{L:limit-ergodic1}
   If a.e.\ ergodic component of an imp action $\G \cc (X,\mu)$ is limit-amenable then the action  is limit-amenable.
\end{lem}

\begin{proof}

    Let $\G\cc(X,\mu)$ and without loss of generality we assume $X$ is locally compact. Suppose that a.e.~ergodic component is limit-amenable.
    By the Ergodic Decomposition theorem, there exists a Borel probability measure $\mu$ on $\Radon(\G, X, X_0)$ such that $\z$-a.e.\ $\nu$ is ergodic and $\mu = \int\nu~d\z(\nu)$.

    Since $\z$ is a Radon-measure and $X$ is Polish, we can approximate $\z$ via a sequence $\z_n$ of measures on $Z$ such that each $\z_n$ has finite support and $\z_n\to \z$ vaguely. Since the barycenter map is continuous in the vague topology, setting $\mu_n = \int_Z \nu~d\z_n(\nu)$ gives us a sequence of Radon measures such that $\mu_n \to \mu$ vaguely. Additionally, each $\mu_n$ has finitely many ergodic components. Since limit-amenability is preserved under taking further limits, it is sufficient to show that each $\mu_n$ is limit-amenable.

    Since $\mu_n$ is supported on finitely many ergodic components, we can write $\mu_n = \sum_{i=1}^{N_n}c_{n,i}\nu_{n,i}$ where $N_n$ is some finite number depending on $n$. Each $\nu_{n,i}$ is limit-amenable (and hence limit-regular) by assumption. Thus there exists a sequence $\big(\nu^{(m)}_{n,i}\big)_{m=1}^\infty$ of measures converging to $\nu_{n,i}$ which are each factors of regular actions. Thus the finite sum $\mu_n^{(m)} = \sum_{i=1}^{N_n}c_{n,i}\nu^{(m)}_{n,i}$ is a factor of an regular action, and $\mu_n^{(m)} \to \mu_n$ vaguely. 
    Hence each $\mu_n$ is limit-amenable, as desired. 
\end{proof}



To prove the converse direction,  it will be helpful to have a general result which shows that we can require the sets $\phi^{-1}(a)$ to be continuity sets.

\begin{defn}
    Let $\G \cc (X,\mu)$ be an imp action, $A$ a finite set, $\phi:X \to A_*$ a $(A,\mu)$-finite map. Define the correlation function $C_{\mu,\phi}:A \times A_* \times \G \to [0,1]$ by
    $$C_{\mu,\phi}(a,b,f) = \mu(\{x\in X:~ \phi(x)=a \textrm{ and } \phi(fx) = b\}).$$
\end{defn}

\begin{defn}
    A map $\phi:X \to A_*$ is a {\bf $\mu$-continuity observable} if it is $(A,\mu)$-finite and $\phi^{-1}(a)$ is a pre-compact $\mu$-continuity set for all $a\in A$.
\end{defn}

\begin{prop}\label{P:continuity-observables}
    $\G \cc (X,\mu)$ is weakly contained in a class $\cC$ of actions if and only if: for every $\mu$-continuity observable $\phi: X \to A_*$, finite $F \subset \G$ and $\eps>0$, there exists an action $\G \cc (Y,\nu)$ in $\cC$ and $(A,\nu)$-finite map $\psi:Y \to A_*$ satisfying
 $$|C_{\phi,\mu}(a,b,f) - C_{\psi,\nu}(a,b,f)|<\eps$$
for all $(a,b,f) \in A \times A_*\times F$.    
\end{prop}

\begin{proof}
$\G \cc (X,\mu)$ is weakly contained in $\cC$ then the conclusion is immediate. So we prove the converse. 

Let $\phi:X \to A_*$ be $(A,\mu)$-finite, $F \subset \G$ be finite, $\eps>0$. Let $\delta>0$ (to be chosen later). 

For every $a\in A$, there exists a pre-compact $\mu$-continuity set $K_a\subset X$ such that $\mu(\phi^{-1}(a)\vartriangle K_a)<\delta$. After replacing $K_a$ with $K_a \setminus \cup_{b \ne a} K_a$ if necessary, we may assume that the sets $\{K_a\}_{a\in A}$ are pairwise disjoint. 

Define $\hphi:X \to A_*$ by 
\begin{displaymath}
    \hphi(x) = 
    \begin{cases}
        a & x\in K_a \\
        * &  x \notin \cup_{a\in A} K_a.
    \end{cases}
\end{displaymath}
Then $\hphi$ is a $\mu$-continuity observable. By hypothesis, there exists an action $\G \cc (Y,\nu)$ in $\cC$ and $(A,\nu)$-finite map $\psi:Y \to A_*$ satisfying
 $$|C_{\hphi,\mu}(a,b,f) - C_{\psi,\nu}(a,b,f)|<\delta$$
for all $(a,b,f) \in A \times A_*\times F$.  

Observe that
\begin{align*}
    |C_{\hphi,\mu}(a,b,f) - C_{\phi,\mu}(a,b,f)| &= |\mu(\hphi^{-1}(a) \cap f^{-1}\hphi^{-1}(b)) - \mu(\phi^{-1}(a) \cap f^{-1}\phi^{-1}(b))| \\
    &= |\mu(K_a \cap f^{-1}K_b) - \mu(\phi^{-1}(a) \cap f^{-1}\phi^{-1}(b))| \le 2\delta.
\end{align*}
By the triangle inequality,
$$|C_{\phi,\mu}(a,b,f) - C_{\psi,\nu}(a,b,f)|<3\delta$$
for all $(a,b,f) \in A \times A_*\times F$. Choose $\delta < \eps/3$. Because $\phi, F, \eps$ are arbitrary, this shows $\G \cc (X,\mu)$ is weakly contained in the class $\cC$.
\end{proof}

The next lemma shows that limit-amenability passes to restrictions. This will be used in the general case. 

\begin{lem}\label{L:abs-continuity}
    Let $\G \cc (X,\mu)$ be a limit-amenable imp action. Suppose there is a $\G$-invariant measurable set $Y \subset X$ with positive measure and $\nu$ is defined by $\nu(E)=\mu(E\cap Y)$ for all $E \subset X$. Then $\G \cc (X,\nu)$ is limit-amenable.
\end{lem}

\begin{proof}
    Let $A$ be a finite set, $\phi:X \to A_*$ be a $(\nu,A)$-finite map, $F \subset \G$ be finite and $\eps>0$.
    
    Define $\hphi:X \to A_*$ by
\begin{displaymath}
    \hphi(x) =
    \begin{cases}
        \phi(x) & x\in Y \\
        * &  x \notin Y
    \end{cases}
\end{displaymath}
    Then 
\begin{align}\label{E:l.a.-abs-cont}
    \nu\big(\{x\in X:~ \phi(x)=a \textrm{ and } \phi(fx)=b\} = \mu\big(\{x\in X:~ \hphi(x)=a \textrm{ and } \hphi(fx)=b\}
\end{align}
for all $a\in A$,$b\in A_*$ and $f\in \G$ because $Y$ is $\G$-invariant and $\nu$ is the restriction of $\mu$ to $Y$.

Because $\mu$ is limit-amenable, there exist an amenable imp action $\G \cc (Z,\zeta)$ and a $(\zeta,A)$-finite measurable map $\psi:Z \to A_*$ such that  
\begin{align*}
    \sum_{a\in A} \sum_{b \in A_*} \sum_{f\in F}& \Big|\mu\big(\{x\in X:~ \hphi(x)=a \textrm{ and } \hphi(fx)=b\}\big)\\
    &- \zeta\big(\{z\in Z:~ \psi(z)=a \textrm{ and } \psi(fz)=b\}\big)\Big|<\eps.
\end{align*} 
By \eqref{E:l.a.-abs-cont},
\begin{align*}
    \sum_{a\in A} \sum_{b \in A_*} \sum_{f\in F}& \Big|\nu\big(\{x\in X:~ \phi(x)=a \textrm{ and } \phi(fx)=b\}\big)\\
    &- \zeta\big(\{z\in Z:~ \psi(z)=a \textrm{ and } \psi(fz)=b\}\big)\Big|<\eps.
\end{align*} 
This implies $\G \cc (X,\nu)$ is limit-amenable as claimed.
\end{proof}

\begin{defn}
    We will say a correlation $C_{\mu,\phi}:A \times A^* \times \G \to [0,1]$ is {\bf $(F,\eps)$-separated from limit-amenable} (where $F \subset \G$ is finite and $\eps>0$) if for every amenable imp action $\G \cc (K,\kappa)$ there does not exist a $(A,\kappa)$-finite map $\psi:K \to A_*$ with 
$$|C_{\phi,\mu}(a,b,f) - C_{\psi,\kappa}(a,b,f)|<\eps$$
for all $(a,b,f) \in A \times A_*\times F$. 
\end{defn}

The next proposition finishes the proof of Theorem \ref{C:limit-amen-erg-decomp2}. 

\begin{prop}\label{P:limit-amenability-to-ergodic-comp}
If $\G \cc (X,\mu)$ is limit-amenable then a.e.\ ergodic component of $\G \cc (X,\mu)$ is limit-amenable.
\end{prop}

\begin{proof}
Without loss of generality, we may assume $X$ is an lcsc space. Let $X_0 \subset X$ be a complete section with $\mu(X_0)<\infty$. For simplicity, we will assume $\mu(X_0)=1$. Let $\mathsf{Radon}(\G,X,X_0)$ be the space of $\G$-invariant Radon measures $\nu$ on $X$ with $\nu(X_0)=1$. We consider $\mathsf{Radon}(\G,X,X_0)$ with the vague topology.

The Ergodic Decomposition Theorem implies the existence of a Borel probability measure $\zeta$ on $\mathsf{Radon}(\G,X,X_0)$ such that $\zeta$-a.e.\ $\nu$ is ergodic and $\mu = \int \nu ~d\zeta(\nu)$ (so $\mu$ is the barycenter of $\zeta)$.

Let $\mu_0$ be in the support of $\zeta$. To obtain a contradiction, suppose that $\mu_0$ is not limit-amenable. By the previous proposition, there exist a finite set $A$, a  $\mu_0$-continuity observable $\phi:X \to A_*$, $\eps>0$ and a finite $F \subset \G$ such that the correlation function $C_{\phi,\mu_0}$ is $(F,\eps)$-separated from limit amenable.


Let $\delta>0$. Let $O$ be the set of all measures $\nu \in \mathsf{Radon}(\G,X,X_0)$ such that $|\nu(\phi^{-1}(a) \cap f^{-1}\phi^{-1}(b) )-\mu_0(\phi^{-1}(a)\cap f^{-1}\phi^{-1}(b) )|<\eps/2$ 
for all $a,b \in A$ and $f \in F$. 

 Because $\phi$ is a $\mu_0$-continuity observable, the set $O$ is open. Since $\mu_0$ is in the support of $\zeta$ this implies $\zeta(O)>0$. Since $\mu_0$ is $(F,\eps)$-separated from limit-amenable, every $\nu \in O$ is $(F,\eps/2)$-separated from limit-amenable. 
 
By definition, $O$ is convex. Let $\nu = \int_{\lambda \in O} \lambda d\zeta(\lambda)$ be the barycenter of the restriction of $\zeta$ to $O$. By the Ergodic Decomposition Theorem, $\nu$ satisfies the hypotheses of Lemma \ref{L:abs-continuity}. Therefore it is limit-amenable. However, $\nu \in O$ and so is  is $(F,\eps/2)$-separated from limit-amenable. This contradiction proves the proposition. 
\end{proof}

\subsection{Direct products}

Here we show that products of limit-amenable action with pmp actions are limit-amenable. 

\begin{thm}\label{T:extensions}
Let $\G \cc (X,\mu)$ be a limit-amenable imp action and $\G \cc (Y,\nu)$ be a pmp action. Then the direct product $\G \cc (X\times Y, \mu \times \nu)$ is limit-amenable.
\end{thm}

\begin{remark}
    Arbitrary extensions of amenable actions are amenable, but this remains open for limit-amenable actions.
\end{remark}

\begin{proof}
Without loss of generality, we may assume $X$ is an lcsc space, $Y$ is a compact metrizable space and the actions on $X$ and $Y$ are by homeomorphisms. Moreover, because $\G \cc (X,\mu)$ is limit-amenable (and thus limit-regular by Theorem \ref{T:limit-regular}), there exist coefficients $c_{n}>0$, elements $x_n \in X$ ($n\in \N$) such that if 
$$\mu_n = c_n\sum_{g\in \G} \d_{gx_n}$$
then $\mu_n$ converges vaguely to $\mu$ as $n\to\infty$. Therefore $\mu_n \times \nu$ converges vaguely to $\mu\times \nu$ as $n\to\infty$. So it suffices to show that $\mu_n \times \nu$ is limit-amenable. Each ergodic component of $\mu_n\times \nu$ has the form 
$$c_n\sum_{g\in \G} \d_{gx_n} \times \d_{gy}$$
for some $g\in \G$. In fact the ergodic decomposition of $\mu_n \times \nu$ is 
$$\mu_n \times \nu = \int c_n\sum_{g\in \G} \d_{gx_n} \times \d_{gy}~d\nu(y).$$
So its ergodic components are regular. By Theorem \ref{C:limit-amen-erg-decomp2} this implies $\mu_n \times \nu$ is limit-amenable.
\end{proof} 

\subsection{Finite measure-preserving actions}

\begin{lem}\label{C:non-amen100}
    If $\G \cc (X,\mu)$ is a limit-amenable measure-preserving action and $\mu(X)<\infty$ then $\G$ is amenable.
\end{lem}

\begin{proof}
By Theorem \ref{T:limit-regular}, we may assume $X$ is an lcsc space and there exist elements $x_n \in X$ and constants $c_n>0$ such that if 
$$\mu_n = c_n \sum_{g\in \G} \d_{gx_n}$$
then $\mu_n$ converges to $\mu$ vaguely. 

Let $\eps>0$. Let $K$ be a pre-compact $\mu$-continuity set with $\mu(K)>\mu(X)-\eps$. Let $F_n=\Ret(K,x_n)=\{g\in \G:~gx_n \in O\}$ be the return time set. We will show that $\{F_n\}$ is almost a F\o lner sequence for $\G$. 

By definition of $\mu_n$, for any subset $Y \subset X$,
\begin{align*}
\mu_n(Y) &=c_n |\Ret(Y,x_n)|.
\end{align*}
Therefore,
\begin{align*}
\frac{|F_n \cap h F_n|}{|F_n|} &= \frac{|\Ret(K,x_n) \cap h \Ret(K,x_n)|}{|\Ret(K,x_n)|} \\
&= \frac{|\Ret(K,x_n) \cap \Ret(h^{-1}K,x_n)|}{|\Ret(K,x_n|}\\
&= \frac{|\Ret(K \cap h^{-1}K,x_n)|}{|\Ret(K,x_n|}\\
&= \frac{\mu_n(K\cap h^{-1}K)}{\mu_n(K)}
\end{align*}
where the third-to-last equality holds because $h \Ret(K,x_n)= \Ret(h^{-1}K,x_n)$.

Because $K$ is a continuity set, it follows that
\begin{align*}
\lim_{n\to\infty} \frac{|F_n \cap h F_n|}{|F_n|} 
&= \frac{\mu(K\cap h^{-1}K)}{\mu(K)} \ge 1-2\eps.
\end{align*}
Because $\eps$ is arbitrary, it follows that for all finite subsets $H \subset \G$ there exists a finite set $F \subset \G$ such that $\frac{|F \cap hF|}{|F|} > 1-\eps$. This implies $\G$ is amenable.
\end{proof}

\subsection{Exactness}
Exactness of groups is surveyed in \cite{anantharaman-survey}.

\begin{thm}\label{T:exact}
If $\G$ is a countable exact group then every limit-amenable action is amenable. 
\end{thm}

\begin{remark}
One can use the main result of \cite{jardónsánchez2025exactnesstopologyspaceinvariant} to prove that if $\G$ is non-exact then there exists a limit-amenable action which is not amenable. We will not need this fact so we leave it to the interested reader.
\end{remark}

First we will reduce the problem to one on the space $2^\G_*$ of non-empty subsets of $\G$ (using a return-times trick). Then we prove the result in the special case that $\G$ is finitely generated before handling the general case.

Let $2^\G$ be the set of all subsets of $\G$ with the product topology. This is a compact metrizable space. Let $2^\G_* \subset 2^\G$ be the set of all non-empty subsets. This is a locally compact second countable space on which $\G$ acts by left-multiplication. Let $\Meas_e(2^\G_*)$ be the space of all $\G$-invariant Radon measures $\nu$ on $2^\G_*$ with the property that $\nu(\{D\subset \G:~e \in D\})<\infty$. 

\begin{prop}\label{P:transfer-to-subsets}
    Let $\G \cc (X,\mu)$ be an imp action. Given a complete measurable section $Y \subset X$ with finite positive measure, define the inverse return time map $\Phi_Y:X \to 2^\G_*$ by
    $$\Phi_Y(x)=\{g \in \G:~ g^{-1}x \in Y\}.$$
    Then $\Phi_Y$ is $\G$-equivariant and $\Phi_{Y*}\mu \in \Meas_e(2^\G_*)$. If $\G \cc (X,\mu)$ is limit amenable then there exists a sequence $(F_n)_{n=1}^\infty$ of finite subsets $F_n \subset \G$ and scalars $t_n>0$ such that if $\zeta_n = t_n \sum_{g\in \G} \d_{gF_n} \in \Meas_e(2^\G_*)$, then $\zeta_n$ converges to $\mu$ in the vague topology as $n\to\infty$.
\end{prop}

\begin{proof}
 By Theorem \ref{T:limit-regular}, $\G \cc (X,\mu)$ is limit-regular. By Theorem \ref{T:Abert-Weiss2}, the action is weakly contained in the class of regular actions. 

Let $Y \subset X$ be a complete measurable section with finite positive measure. Let $\phi:X \to \{*,1\}$ be the characteristic function of $Y$. So $\phi(x)=1$ if and only if $x\in Y$. This means that $\mu$ is $(\mu,A)$-finite where $A=\{1\}$ is a singleton.

By \ref{T:general-correlations}, for every finite $F\subset \G$ and $\eps>0$ there exist a regular action $\G \cc (Y,\nu)$ and a $(\nu,A)$-finite measurable map $\psi:Y \to A$ such that 
$$\|\phi_*^F\mu - \psi^F_*\nu\|_F < \eps.$$

Because the action $\G \cc (Y,\nu)$ is regular, there exists a scalar $t>0$ such that $\G \cc (Y,\nu)$ is measurably conjugate to the left-regular action $\G \cc \G$ with the measure $t\cdot c_\G$ where $c_\G$ is counting measure on $\G$. If we let $\psi(e)=x \in A_*^\G$, we see that 
$$\psi^\G_*\nu=t\cdot \sum_{g\in \G} \d_{gx}.$$
Because $\psi$ is $(\nu,A)$-finite, 
$$\nu(\{y:~ \phi(y)=1\})<\infty.$$
But $\nu(\{y:~ \phi(y)=1\})$ equals $t$ times 
$$\#\{g\in \G:~x(g)=1\}.$$
So the latter set is finite.

Because $\eps$ and $F$ in the paragraph above are arbitrary, there exist a sequence $(x_n)_{n=1}^\infty$ of elements of $A_*^\G$ and a sequence $(t_n)_{n=1}^\infty$ of scalars such that if 
$$\zeta_n = t_n \cdot \sum_{g\in \G} \d_{gx_n}$$
then $\zeta_n$ converges to $\phi^\G_*\mu$ in the vague topology as $n\to\infty$. Moreover, each $x_n$ satisfies
$$\#\{g\in \G:~x_n(g)=1\}<\infty.$$

Let $2^\G_*$ be the set of all non-empty subsets of $\G$ with the usual $\G$-action by left-multiplication. Then $A_*^\G$ and $2^\G$ are topologically conjugate by the map $\Psi:A_*^\G \to 2^\G$ which sends $x$ to $x^{-1}(1)$. So the result holds by applying $\Psi$.
\end{proof}

\begin{prop}\label{P:exact-fg}
If $\G$ is a finitely generated exact group. Let $(F_n)_{n=1}^\infty$ be a sequence of finite subsets $F_n \subset \G$ and $t_n>0$ be scalars. Let 
$$\zeta_n = t_n \sum_{g\in \G} \d_{gF_n} \in \Meas_e(2^\G_*).$$
Suppose the sequence $(\zeta_n)_{n=1}^\infty$ converges in the vague topology to a Radon measure $\mu$ as $n\to\infty$. Then the action $\G \cc (2^\G_*,\mu)$ is amenable.
\end{prop}

\begin{proof}

Let $S \subset \G$ be a finite symmetric generating set and $\Cay(\G,S)$ denote the corresponding Cayley graph. We will say that a subset $D \subset \G$ is connected if the subgraph it induces in $\Cay(\G,S)$ is connected. Connected components of $D$ are defined similarly.

Let $2_0^\G \subset 2^\G_*$ be the set of subsets that contain the identity. More generally, let $2^\G_n \subset 2^\G_*$ be the set of subsets which have nontrivial intersection with the ball of radius $n$ centered at the identity in the Cayley graph $\Cay(\G,S)$.

The canonical sub-graphing $\cG_n$ is the set of all pairs $(D,gD) \in 2^\G_n \times 2^\G_n$ such that the identity is contained in the radius-$n$ neighborhood of $D$ and in the radius-$n$ neighborhood of $gD$ (where $g\in \G$). This is most intuitive in the special case in which $n=0$ in which case $(D,gD)\in \cG_0$ if and only if $g^{-1}\in D$.

It is well-known that if $\G$ is exact then every Cayley graph of $\G$ (with respect to a finite generating set) has Yu's Property A \cite{anantharaman-survey, MR1728880}. This result is attributed to Ozawa \cite{MR1763912} (although that paper does not explicitly mention Property A, it is well-known Property A is equivalent to the other properties mentioned in that paper- see \cite[Proposition 3.13]{anantharaman-survey}). 

By Theorem 1 of \cite{elek2023comprehensive}, $\Cay(\G,S)$ is locally hyperfinite which means for every $\eps>0$ there exists $k>0$ satisfying the following condition: for any finite subset $L \subset \G$ there exists a subset $L' \subset L$, $|L'| < \eps |L|$ such that if one deletes $L'$ and all adjacent edges in $L$ then the sizes of the remaining components are at most $k$.

It follows that the connected components of the radius $n$ neighborhoods of $F_k$ form a hyperfinite family in the language of \cite{MR2967301}. By Theorem 1 of \cite{MR2967301}, implies that $\cG_n$ is hyperfinite mod $\mu$. Elek's Theorem may be thought of as a generalization of Schramm's earlier result \cite{MR2372897} which is formulated in terms of unimodular random rooted graphs. 

There is a minor technical issue with the possibility of nontrivial stabilizers. In order to handle this, let $\tilde{\cG}=\G\times 2_*^\G$ be the groupoid of the action where multiplication is defined by 
$$(h,gD)(g,D)=(hg,D)$$
for any $g,h\in \G$ and $D \subset \G$. Let $\tilde{\cG_n}$ be the sub-groupoid consisting of pairs $(g,D)$ such that $(D,gD) \in \cG_n$. Then Elek's Theorem \cite{MR2967301} and Schramm's Theorem from \cite{MR2372897} imply that $\tilde{\cG}_n$ is amenable (with respect to the restriction of $\mu$ to $2^\G_n$). Since $\tilde{\cG}$ is the increasing union of $\tilde{\cG}_n$, it follows that $\tilde{\cG}$ is amenable (with respect to $\mu$), i.e. the action $\G \cc (2^\G_*,\mu)$ is amenable. 
\end{proof}

\begin{lem}\label{L:limits-subgroup}
    Let $\G \cc (X,\mu)$ be a limit-regular action and let $H \le \G$ be a subgroup. Then the action $H \cc (X,\mu)$ is also limit-regular.
\end{lem}

\begin{proof}
   By taking limits, it suffices to consider the special case in which the action $\G \cc (X,\mu)$ is regular. That is, we may assume $X=\G$, $\mu$ is counting measure and the action is by left-translations. So the restricted action of $H$ on $\G$ consists of $|\G/H|$ copies of the left-regular action of $H$ on itself. The lemma now follows from Lemma \ref{L:regular2}.
\end{proof}

\begin{proof}[Proof of Theorem \ref{T:exact}]
Propositions \ref{P:transfer-to-subsets} and \ref{P:exact-fg} imply Theorem \ref{T:exact} in the special case in which $\G$ is finitely generated. This uses the fact that extension of amenable actions are amenable. So if the action $\G \cc (2^\G_*, \Phi_{Y*}\mu)$ is amenable then the action $\G \cc (X,\mu)$ is amenable.

Let $\G$ be a countable exact group and let $\G \cc (X,\mu)$ be limit-amenable. Let $H_1,H_2,\ldots$ be an increasing sequence of finitely generated subgroups of $\G$ with $\cup_i H_i = \G$. 

Because closed subgroups of exact groups are exact \cite[Theorem 4.1]{MR1725812}, each subgroup $H_i$ is exact. So Lemma \ref{L:limits-subgroup} and Proposition \ref{P:exact-fg} imply that the restricted actions $H_i \cc (X,\mu)$ are amenable. Since the action groupoid of $\G \cc (X,\mu)$ is the increasing union of the action groupoids $H_i \cc (X,\mu)$, it follows that $\G \cc (X,\mu)$ is amenable. 
\end{proof}

\section{Poisson suspensions and limit-amenability}\label{S:Poisson}




The main result of this section is that the Poisson suspension functor preserves weak containment. We begin by defining the Poisson suspension functor. For this, let $X$ be a locally compact second countable space.

\begin{defn}
 A measure $\Pi$ on $X$ is called a {\bf point measure} if it can be expressed as a sum of Dirac masses $\Pi = \sum_{x\in S} c(x)\delta_x$ for some locally finite countable subset $S\subset X$ and non-negative integers $c(x)$ ($x\in S$). The local finiteness condition means that for any compact set $K \subset X$, $K \cap S$ is finite. Therefore $\Pi(K)<\infty$. The support of $\Pi$ is $\mathsf{Support}(\Pi)=\{x\in X:~\Pi(x)>0\}$. 
\end{defn}

\begin{defn}
The set of all point measures on $X$ is denoted $\M(X)$ and is called the \textbf{configuration space of $X$}. We consider $\M(X)$ as embedded in the space of Radon measures on $X$ which may be identified with a subset of the Banach dual $C_0(X)^\ast$ via the Riesz Representation Theorem. We will consider $C_0(X)^\ast$ with the weak$^*$ topology. Then $\M(X)$ is a closed subset of $C_0(X)^\ast$ and therefore $\M(X)$ is Polish in the sense that there exists a complete separable metric inducing its topology.  
\end{defn}

\begin{defn}
Let  $\mu$ be a Radon measure on $X$. A {\bf Poisson point process on $X$ with intensity measure $\mu$} is a random variable $\bfPi$ taking values in $\M(X)$ satisfying
\begin{enumerate}
    \item for any measurable $E \subset X$ with $\mu(E)<\infty$,  $\bfPi(E)$ is a Poisson random variable with mean $\mu(E)$;
    \item if $E_1,E_2,\ldots$ are pairwise disjoint measurable subsets of $X$ then the restrictions of $\bfPi$ to $E_1,E_2,\ldots$ are jointly independent random variables.
\end{enumerate}
All such processes have the same law, which we denote by $\Pois(\mu) \in \Prob(\M(X))$.
\end{defn}

\begin{defn}
    If $\a=(\G \cc (X,\mu)$) is an imp action then the induced action $\Pois(\a)=(\G \cc (\M(X),\Pois(\mu)))$ is a pmp action. It is called the {\bf Poisson suspension} of the action $\G \cc (X,\mu)$. If $\cC$ is a class of imp actions, then $\Pois(\cC)=\bigcup_{\a\in\cC} \Pois(\a)$.
    See \cite{MR2486789} for an introduction to Poisson suspensions.
\end{defn}


The main result of this section is:
\begin{thm}\label{T:wc2}
If an imp action $\G \cc (X,\mu)$ is weakly contained in a class $\cC$ of actions then its Poisson suspension $\G \cc (\M(X),\Pois(\mu))$ is weakly contained in $\Pois(\cC)$.
\end{thm}





To prove this, we will combine the next result with Theorem \ref{T:Abert-Weiss2}.

\begin{thm}\label{T:Pois-continuity}
    Let $X$ be an lcsc space. 
    Let $\Radon(X)$ be the space of $\G$-invariant Radon measures on $X$ with the vague topology. Let
    $$\Pois:\Radon(X)  \to \Prob(\M(X))$$
    be the Poisson suspension functor. Then $\Pois$ is continuous with respect to the vague topology on $\Radon(X)$ and the weak topology on $\Prob(\M(X))$.
\end{thm}
This theorem will be proven after a few lemmas. For this,  we fix the following hypotheses. For $n\in \N \cup \{\infty\}$, let $\mu_n \in \Radon(X)$ and suppose $\mu_n \to \mu_\infty$ in the vague topology as $n\to\infty$. For $\Lambda \subset X$ and $I \subset \N \cup \{0\}$ let 
$$\M(X,\Lambda,I)=\{\Pi \in \M(X):~ \Pi(\Lambda)\in I\}.$$
In the special case in which $I=\{t\}$, we write $\M(X,\Lambda,t)=\M(X,\Lambda,I)$ for simplicity. For $n\in\N\cup\{\infty\}$, let $\Pois_n=\Pois(\mu_n)$.

   Recall that a measurable subset $\Lambda \subset X$ is a {\bf continuity set} for $\mu_\infty$ if $\mu_\infty(\partial \Lambda)=0$ where $\partial \Lambda = \overline{\Lambda} \cap \overline{(X \setminus \Lambda)}$ is the topological boundary of $\Lambda$. It is well-known the collection of continuity sets is closed under complementation, finite unions and intersections. Moreover, if $K \subset X$ is any compact set and $\eps>0$ then there exist a continuity sets $K',K''$ with $K' \subset K \subset K''$ and $\mu_\infty(K''\setminus K')<\eps$. Indeed, we can take $K''$ to be a radius $\delta$ neighborhood of $K$ for some sufficiently small $\delta$ with respect to a continuous proper metric and similarly let $K'$ be the complement of a radius $\delta$ neighborhood $X \setminus K$ for some sufficiently small $\delta$. This is because there are uncountably many $\delta$ to choose from but the set of all $\delta$ such that the radius $\delta$ neighborhood is not a continuity set is at most countable. 
   
\begin{lem}\label{L:continuityset}
 Suppose that $\Lambda_1,\ldots, \Lambda_k \subset X$ are pairwise disjoint pre-compact continuity sets for $\mu_\infty$ and $t_i \in \N \cup \{0\}$ for $1\le i \le k$. Then
 $$\lim_{n\to\infty} \Pois_n(\cap_{i=1}^k \M(X,\Lambda_i, t_i)) =  \Pois_\infty(\cap_{i=1}^k \M(X,\Lambda_i, t_i)).$$
\end{lem}

\begin{proof}
In fact this is straightforward because
$$\Pois_n(\cap_{i=1}^k \M(X,\Lambda_i, t_i))= \prod_{i=1}^k \exp(-\mu_n(\Lambda_i)) \frac{\mu_n(\Lambda_i)^{t_i}}{t_i!}$$
is a continuous function of $(\mu_n(\Lambda_i))_{i=1}^k$ and $\lim_n \mu_n(\Lambda_i)=\mu_\infty(\Lambda_i)$ for all $i$ by the Locally Compact Portmanteau Theorem \ref{T:portmanteau}.
\end{proof}

\begin{lem}\label{L:tight}
 The sequence $\{\Pois(\mu_n)\}_{n \in \N}$ is tight. In particular, there exists a subsequential limit which is a Borel probability measure.
\end{lem}

\begin{proof}
It suffices to prove that for every $\eps>0$ there exists a compact set $K \subset \M(X)$ such that
$$\liminf_{n\to\infty} \Pois(\mu_n)(K) \ge 1-\eps.$$
Fix a basepoint $x_0\in X$. Let $d$ be a continuous metric on $X$ so that closed balls of finite radius in $d$ are compact. Let $0<r_1<r_2<\cdots$ be an increasing sequence of radii with $\lim_i r_i=\infty$ such that for each $i$, the open ball $B(r_i,x_0)$ of radius $r_i$ centered at $x_0$ is a continuity set for $\mu_\infty$. Then
$$\lim_{n\to\infty} \mu_n(B(r_i,x_0)) = \mu_\infty(B(r_i,x_0))$$
for all $i$ by the Portmanteau Theorem \ref{T:portmanteau}. It follows that if
$$V_i = \sup_{1\le n < \infty} \mu_n(B(r_i,x_0))$$
then $V_i$ is finite for all $i$. Therefore, there exist numbers $N_i\in \N$ such that 
$$\Pois_n(\{\Pi:~ \Pi(B(r_i,x_0)) \le N_i\}) \ge 1-\eps/2^n$$
for all finite $n \in \N$ and all $i$. 

Let
$$K = \{\Pi:~ \Pi(B(r_i,x_0)) \le N_i~\forall i \ge 1\}.$$
Then $K$ is compact because $B(r_i,x_0)$ is compact and $\mu_n(K) \ge 1-\eps$ for all $n$. 
\end{proof}

\begin{proof}[Proof of Theorem \ref{T:Pois-continuity}]
    
It suffices to prove $\Pois_n \to \Pois_\infty$ in the weak topology as $n\to\infty$. By Lemma \ref{L:tight} and Prokhorov's Theorem, there exists a subsequence $(n_i)_{i=1}^\infty$ such that $\Pois_{n_i}$ converges weakly to a Borel probability measure $\kappa$ on $\M(X)$. By Lemma \ref{L:continuityset}, 
\begin{align}\label{E:Poisson1}
    \kappa\left(\cap_{i=1}^k \M(X,\Lambda_i, t_i)\right)=\Pois_\infty\left(\cap_{i=1}^k \M(X,\Lambda_i, t_i)\right)
\end{align}
for any pairwise disjoint measurable sets $\Lambda_1,\ldots, \Lambda_k$ and any $t_1,\ldots, t_k \in \N \cup \{0\}$ as long as the $\Lambda_i$ are pre-compact $\mu_\infty$-continuity sets. 

Let $\Lambda_1,\Lambda_2,\ldots, \Lambda_k\subset X$ be a sequence of pairwise disjoint pre-compact Borel sets. It suffices to show
\begin{align}\label{E:Poisson2}
    \kappa\left(\cap_{i=1}^k \M(X,\Lambda_i, t_i)\right)=\Pois_\infty\left(\cap_{i=1}^k \M(X,\Lambda_i, t_i)\right).
\end{align}
This is because, by an application of the Caratheodory-Hahn Extension Theorem, if this holds for all $k$ then it holds for $k=\infty$. Moreover, because $X$ is lcsc, it is $\s$-compact; so the assumption that the sets $\Lambda_i$ are pre-compact does not cause difficulties.

Let $I_\kappa$ be the intensity measure of $\kappa$. This is the measure on $X$ defined by
$$I_\kappa(\Lambda) = \int \Pi(\Lambda) \, \dee \kappa(\Pi)$$
for $\Lambda \subset X$. We will prove that $I_\kappa=\mu_\infty$.

Note $I_\kappa(\Lambda) = \mu_\infty(\Lambda)$ whenever $\Lambda$ is a measurable pre-compact $\mu_\infty$-continuity set by \eqref{E:Poisson1}. 

Let $\Lambda \subset X$ be any compact set and $\eps>0$. Then there exists a pre-compact measurable $\mu_\infty$-continuity set $\Lambda'$ such that $\Lambda \subset \Lambda'$ and $\mu_\infty(\Lambda' \setminus \Lambda)<\eps$. 
Then
$$I_\kappa(\Lambda) \le I_\kappa(\Lambda') = \mu_\infty(\Lambda') \le \mu_\infty(\Lambda)+\eps.$$
Since this is true for all $\eps$ and all compact sets, we have $I_\kappa \le \mu_\infty$. 

On the other hand if $\Lambda\subset X$ is any pre-compact open set then there exists a pre-compact measurable $\mu_\infty$-continuity set $\Lambda'$ with $\Lambda' \subset \Lambda$ and $\mu_\infty(\Lambda \setminus \Lambda')<\eps$. 
Therefore
$$I_\kappa(\Lambda) \ge I_\kappa(\Lambda') = \mu_\infty(\Lambda') \ge \mu_\infty(\Lambda)-\eps.$$
Since this is true for all $\eps$ and all pre-compact open sets and since $X$ is locally compact, we have $I_\kappa \ge \mu_\infty$. Therefore $I_\kappa=\mu_\infty$.



Let $\Lambda_1,\ldots, \Lambda_k$ be a sequence of pairwise disjoint pre-compact Borel sets. For each $i$, there  exist continuity sets $K_{1,j}\subset K_{2,j} \subset \cdots \subset \Lambda_i$ with $\mu_\infty(\Lambda_i \setminus \cup_j K_{i,j})=0$. Since $I_\k=\mu_\infty$, it follows that
\begin{align*}
    \k(\cap_{i=1}^k \M(X,\Lambda_i, t_i)) &= \lim_{n\to\infty}  \k(\cap_{i=1}^k \M(X,K_{i,n}, t_i)) \\
    &= \lim_{n\to\infty}  \Pois_\infty(\cap_{i=1}^k \M(X,K_{i,n}, t_i)) \\
    &= \Pois_\infty(\cap_{i=1}^k \M(X,\Lambda_i, t_i)).
\end{align*}
This proves \eqref{E:Poisson2}.
\end{proof}

\begin{proof}[Proof of Theorem \ref{T:wc2}]
Let $\cC$ be a class of imp actions and suppose  $\G \cc (X,\mu)$ is weakly contained in $\cC$. By Lemma \ref{L:finiteness}, we may assume without loss of generality, that $X=\cU$ and $\mu$ is Radon.

By Theorem \ref{T:Abert-Weiss2}, $\mathsf{Factor}(\a,\cU) \subset \overline{\mathsf{Factor}}(\cC,\cU)$. In particular, $\mu \in \overline{\mathsf{Factor}}(\cC,\cU)$. By Theorem \ref{T:Pois-continuity}, $\Pois(\mu)$ is contained in the weak closure of $\{\Pois(\nu):~ \nu \in \mathsf{Factor}(\cC,\cU) \}$. It now follows from Theorem \ref{T:Abert-Weiss2}.

\end{proof}

\subsection{Weakly contained in Bernoulli}

\begin{defn}
    Let $(K,\k)$ be a standard Borel probability space and let $(K,\k)^\G=(K^\G,\k^\G)$ be the direct product of $\G$ copies of $(K,\k)$. An element $x$ of $K^\G$ is a function $x:\G \to K$. The group $\G$ acts on $K^\G$ by translations: $(gx)(f)=x(g^{-1}f)$ for $f,g \in \G$ and $x\in K^\G$. This action preserves the product measure $\k^\G$. The action $\G \cc (K,\k)^\G$ is called the {\bf Bernoulli shift} over $\G$ with base space $(K,\k)$.
\end{defn}

\begin{defn}
    A pmp action $\G \cc (X,\mu)$ is \textbf{weakly contained in Bernoulli} if there is a Bernoulli shift action $\G \cc (K,\k)^\G$ which weakly contains it. It is known that all Bernoulli shift actions of $\G$ are weakly equivalent \cite{abert-weiss-2013}. Therefore, if a pmp action is weakly contained in Bernoulli then it is weakly contained in every Bernoulli shift.
\end{defn}

\begin{cor}\label{weakbern}
    If $\a$ is a limit-amenable imp then $\Pois(\a)$ is weakly contained in Bernoulli.
\end{cor}

\begin{proof}
Let $\a$ be a limit-amenable imp. By Theorem \ref{T:limit-regular}, $\a$ is limit-regular. By Theorem \ref{T:wc2}, $\Pois(\a)$ is weakly contained in the $\Pois(\l)$ where $\l$ is the left-translation action of $\G$ on itself. It is an exercise to check that $\Pois(\l)$ is a Bernoulli shift. So this implies the corollary.
\end{proof}

\section{Cost}\label{S:cost}



In this section we discuss cost, starting with a review of the standard theory in \S \ref{S:cost-eq}. 
In \S \ref{S:cost-normal} we introduce the idea of normalized cost for general measure preserving actions.
We then introduce an equivalent notion to cost in \S \ref{S:cost-graph} which is used in our proof of Theorem \ref{T:ppp-cost} in \S \ref{S:cost-Poisson}.
This notion is called the graph-cost and may be thought of as an alternative to groupoid-cost in the special case of action groupoids. 

\subsection{Cost of discrete groups}\label{S:cost-eq}
We will freely use standard terminology from the theory of measured equivalence relations, as reviewed in \S \ref{S:mer}. 


\begin{defn}
Fix a discrete pmp equivalence relation $(X,\mu,\cR)$. A \textbf{sub-graphing} of $\cR$ is a symmetric Borel subset $\cG \subset \cR$. Symmetric means $(x,y)\in \cG$ implies $(y,x)\in \cG$. We think of $\cG$ as the edges of a graph with vertex set $X$. Since $\cG \subset \cR$, each connected component of this graph is necessarily contained in the complete graph of an $\cR$-class. A sub-graphing is a \textbf{graphing} if each connected component of this graph spans an $\cR$-class (ignoring a set of measure zero). This means: for a.e.\ $(x,y)\in \cR$ there exist $(x_1,x_2),(x_2,x_3),\ldots, (x_{n-1},x_n) \in \cG$ such that $x_1=x$ and $x_n=y$. In this case, we say that $\cG$ \textbf{generates} $\cR$. 
\end{defn}

\begin{defn}
The \textbf{$\mu$-cost} of $\cR$ is $$\cost_\mu(\cR)=\cost(X,\mu,\cR)=(1/2)\hmu(\cG)=(1/2)\inf \int \#\{y:~ (x,y)\in \cG\}~\dee\mu(x)$$ 
where the infimum is over all graphings $\cG$ generating $\cR$ and $\hmu=\mu_L=\mu_R$ is the measure on $\cR$ induced by $\mu$ as in Theorem \ref{T:mp}.   
\end{defn}

\begin{defn}
    The \textbf{cost} of a group $\G$ is the infimum of $\cost(X,\mu,\cR_\G)$ where the infimum is over all essentially free pmp action $\G \cc (X,\mu)$ and $\cR_\G=\{(x,gx):~x\in X, g\in \G\}$ is the orbit-equivalence relation. 
\end{defn}

\begin{defn}
The \textbf{max-cost} of $\G$ is the supremum of $\cost(X,\mu,\cR_\G)$  over all essentially free pmp actions $\G \cc (X,\mu)$ and $\cR_\G=\{(x,gx):~x\in X, g\in G\}$. By \cite{abert-weiss-2013}, the max-cost is achieved by Bernoulli shifts over $\G$. Moreover, if $\G \cc (X,\mu)$ is any essentially free pmp action which is weakly contained in Bernoulli, then the cost of its orbit-equivalence relation is the max-cost of $\G$.
\end{defn}

\subsection{Normalized cost}\label{S:cost-normal}

\begin{defn}
Let $X$ be a standard Borel space and let $\cR$ be a discrete Borel equivalence relation. A subset $S \subset X$ is a {\bf complete section} for $\cR$ if it meets every $\cR$-class. In other words, if for every $x\in X$ there exists $y\in S$ with $(x,y)\in \cR$. Let $\mu$ be an $\cR$-invariant measure. Then a subset $S$ is a {\bf complete section mod $\mu$} if it is a complete section for the restriction of $\cR$ to some $\mu$-conull subset of $X$.
\end{defn}
We will use the following theorem of Gaboriau, which also appears as \cite[Theorem 21.1]{MR2095154}.

\begin{thm}\label{section}\cite{gaboriau-cost}
    Let $(X,\mu,\cR)$ be a countable finite-measure-preserving Borel equivalence relation on $X$, $S\subseteq X$ a Borel complete section mod $\mu$. 
    Then \[\cost_\mu(\cR) = \cost_{\mu|S}(\cR|S)+\mu(X\setminus S).\]
    Here $\mu|S$ is the measure defined by $\mu(B)=\mu(S\cap B)$ and $\cR|S = \cR \cap (S\times S)$ is the restriction of $\cR$ to $S$. 
\end{thm}

\begin{cor}\label{C:cost-section}
 Let $(X,\mu,\cR)$ be a countable measure-preserving Borel equivalence relation on $X$, $S,T\subseteq X$ Borel complete sections mod $\mu$. 
Then
$$\cost_{\mu|S}(\cR|{S})+1-\mu( S) = \cost_{\mu|T}(\cR|{T})+1-\mu(T).$$
\end{cor}

\begin{proof}
By Theorem \ref{section} of Gaboriau, we have
    \begin{align*}
        \cost_{\mu | S\cup T }(\cR|{S \cup T}) &= \cost_{\mu|S}(\cR|S)+ \mu((S \cup T) \setminus S) \\
        &= \cost_{\mu|T}(\cR|T)+ \mu((S\cup T) \setminus T).
    \end{align*}
Thus we have $\cost_{\mu|S}(\cR|S)-\mu( S) = \cost_{\mu|T}(\cR|T)-\mu(T)$.
\end{proof}

\begin{defn}
Let $(X,\mu,\cR)$ be a discrete mp equivalence relation. The\textbf{ normalized cost of $(X,\mu,\cR)$} is defined by 
$$\mathsf{ncost}(X,\mu,\cR) = \cost_{\mu|\Lambda}(\cR|\Lambda)+1-\mu(\Lambda)$$
where $\Lambda \subset X$ is any complete section with finite positive measure and $\mu|\Lambda$ is the restriction of $\mu$ to $\Lambda$. By Corollary \ref{C:cost-section}, this does not depend on the choice of section $\Lambda$. Furthermore, in the case where $\mu(\Lambda) = 1$, the normalized cost of $\cR$ is just the $\mu$-cost of the equivalence relation restricted to $\Lambda$. 
\end{defn}



\subsection{Graph-cost}\label{S:cost-graph}

Let $\G$ be a countable group. In order to bound the cost of the Poisson suspension of an action of $\G$, we give an alternative definition of the cost of an essentially free probability measure preserving $\G$-action. We will show that this definition agrees with $\mu$-cost. 

Let $2^{\G\times \G}$ be the set of all subsets of  $\G\times \G$. Endowed with the product topology, $2^{\G\times \G}$ is a compact metrizable space on which $\G$ acts continuously by
$$h\cdot G=\{(h g_1,h g_2 ):~(g_1,g_2)\in G\}$$ for $h\in \G$, $G \in 2^{\G\times \G}$.
Let $\mathsf{Graph}(\G) \subseteq 2^{\G\times \G}$ be the set of all symmetric subsets $G$ of $\G\times \G$, i.e.\ those such that $(g_1,g_2) \in G$ if and only if $(g_2,g_1)\in G$.

The {\bf support} $\mathsf{supp}(G)$ of $G \in \mathsf{Graph}(\G)$ is the set of all $g\in \G$ such that there exists some $g' \in \G$ with $(g, g')\in G$. If $G \in \mathsf{Graph}(\G)$, then $G$ can be thought of as an undirected graph with vertex set $V(G) = \mathsf{supp}(G) \subseteq \Gamma$ and edge set $E(G) = G \subseteq \Gamma \times \Gamma$.
We say $G$ is {\bf fully supported} if its support is all of $\Gamma$.

An element $G \in \mathsf{Graph}(\G) $ is {\bf connected} if the graph $G = (V(G), E(G))$ it generates is connected. 
In other words, $G$ is connected if for all $g,h\in \mathsf{supp}(G)$ there exist some $n\in \N$ and $g=g_1, g_2, \dots, g_n=h \in \mathsf{supp}(G)$ such that
$(g_i,g_{i+1})\in G$ for all $1\leq i\leq n-1$.


The degree of $G \in \mathsf{Graph}(\G)$ at the identity is 
$$\deg_e(G) = \#\{g\in \G \setminus \{e\}:~ (e,g)\in G\}.$$

Suppose $\G \cc (X,\mu)$ is a pmp action and  $\pi: X \to \mathsf{Graph}(\G)$ is a $\G$-equivariant map such that $\pi_* \mu$-almost every $G$ is connected and fully supported. Then we say that \textbf{$\pi$ is a $\G$-graph}, and we say that the \textbf{graph-cost of $\pi$} is
\[ \operatorname{grcost}_\mu(\pi) = \frac{1}{2} \int \deg_e(G)~\dee \pi_*\mu(G) = \frac{1}{2}\int \deg_e(\pi(x))~\dee \mu(x). \]


The \textbf{graph-cost of a pmp action $\G \cc (X,\mu)$} is 
$$ \operatorname{grcost}(\G,X,\mu) = \inf\{\operatorname{grcost}_\mu(\pi) :~ \pi \text{ is a $\G$-graph}\}$$
where the infimum is taken over all $\G$-equivariant maps $\pi:X \to \mathsf{Graph}(\G)$ such that $\pi_*\mu$-a.e.\ $G$ is connected and fully supported.

\begin{remark}
    It can be shown that the graph-cost of a pmp action is the same as the groupoid cost of the induced groupoid. We will not need this so we do not prove it here.
\end{remark}


Let $\pi: X \to \mathsf{Graph}(\G)$ be a $\G$-graph. Define $\cG_\pi \subset X\times X$ by $(x,g x) \in \cG_\pi$ if and only if $(e,g^{-1})\in \pi(x)$. We will prove that $\cG_\pi$ is a graphing associated to the action of $\G$ on $(X,\mu)$, and that $\operatorname{grcost}_\mu(\pi)=\cost_\mu(\cG_\pi)$. That is, $\cG_\pi$ is a graph such that the connected components are exactly the equivalence classes of $\cR_\G$.



\begin{lem}
  Let $\pi: X \to \mathsf{Graph}(X)$ be $\G$-invariant. If $\pi(x)$ is fully supported and connected for a.e.\ $x$ then  $\cG_\pi$ is a graphing generating the orbit equivalence relation $\cR_\G$. 
\end{lem}

\begin{proof}
    Let $g_1, g_2 \in \G$. By definition, $(g_1 x,g_2 x) \in \cG_\pi$ if and only if $(g_1^{-1},g_2^{-1})\in \pi(x)$. More precisely,
    $(hx,gx) \in \cG_\pi$ if and only if $(e,hg^{-1})=(e,(gh^{-1})^{-1}) \in \pi(hx)=h\pi(x)$ since $\pi$ is $\G$-invariant. This occurs exactly when $(h^{-1},g^{-1}) \in \pi(x)$.
    
    Furthermore, if $(e,g^{-1}) \in \pi(x)$ then $(g,e) \in \pi(g x)$ which implies $(e,g) \in \pi(g x)$ and hence $(g x,x)\in \cG$.

    If $x,y\in X$ are in the same orbit then there exists $g \in \G$ with $y=g x$. Because $\pi(x)$ is connected and has full support almost surely, there exists a path from $x$ to $g x$ in $\pi(x)$ a.s.
    Thus there a.s.~exists $g_1,g_2,\ldots, g_n \in \G$ with $(e,g_1), (g_1,g_2),\ldots, (g_{n-1},g_n) \in \pi(x)$ and hence $(x, g_1^{-1}x), (g_1^{-1}x, g_2^{-1}x), 
    \ldots, (g_{n-1}^{-1}x, g_{n}^{-1}x) \in \cG_\pi$. This means that all points in the orbit of $x$ under the action of $\G$ are in the same connected component of $\cG_\pi$ almost surely.
\end{proof}

\begin{cor}
    Whenever $\pi$ is a $\G$-graph, the graph-cost of $\pi$ is at least the cost of the induced graphing $\cG_\pi$. These are equal if the action $\G \cc (X,\mu)$ is essentially free.
\end{cor}

\begin{proof}
    Let $\cG_\pi$ be the graphing induced by $\pi$.  As in the proof of the previous lemma, for $x\in X$, there is a surjective map from the edges of $\pi(x)$ adjacent to $e$ to the edges of $\cG_\pi$ adjacent to a vertex $x$ given by $(e,g)\mapsto (x,g^{-1}x)$. This implies  $\deg_{\cG_\pi}(x) \le \deg_e(\pi(x))$ a.e.. Moreover, this map is a bijection if the action $\G \cc (X,\mu)$ is essentially free. So  
    \[\cost_\mu(\cG_\pi) = \frac{1}{2} \int \deg_{\cG_\pi}(x) d \mu(x) \le \frac{1}{2} \int \deg_e(\pi(x)) d\mu(x) = \operatorname{grcost}_\mu(\pi) \]
 with equality throughout in the case where $\G\cc(X,\mu)$ is essentially free.  
\end{proof}

\begin{lem}
    If $\cG$ is a graphing generating the equivalence relation $\cR_\G$ induced by an essentially free action of $\G$, then there exists some $\G$-graph $\pi$ such that $\cG = \cG_\pi$. 
\end{lem}

\begin{proof}
    Suppose $\cG \subset X \times X$ is a graphing that generates the equivalence relation $\cR_\G$. Define $\pi:X \to \mathsf{Graph}(\G)$ by $\pi(x)$ is the set of all $(g'^{-1}, g^{-1})$ such that $(g x, g'x) \in \cG$.
    


    Clearly $\pi(x)$ is symmetric. Since $\cG$ generates $\cR_\G$, the connected components of $\cG$ are the equivalence classes in $\cR_\G$. That means that for a.e.~$x \in X$, if $y \in X$ is such that there exists $g \in \G$ with $x=g y$, there exists a path from $x$ to $y$ in $\cG$. Thus there exists a path from $e$ to $g$ in $\pi(x)$, and $\pi(x)$ is connected almost surely. Furthermore, $\pi(x)$ has full support because the action is essentially free.

    Lastly, $\pi$ is $\G$-equivariant. Let $g \in \G$. We can see that $\pi(g x) = g \pi(x)$ since if $(x, g x) \in \cG$ meaning $(e, g^{-1}) \in \pi(x)$, then $(g x, x) \in \cG$ and hence $(e, g) \in \pi(g(x))$.
\end{proof}

\begin{cor}\label{cost}
    The graph-cost of an essentially free probability measure preserving action $\Gamma \cc (X, \mu)$ is equal to the cost of $\Gamma \cc (X, \mu)$. 
\end{cor} 

\subsubsection{Graphs which are not fully supported}
The main purpose of this subsection is to prove Lemma \ref{weight-cost} below which provides a formula for graph-cost in terms of $\G$-maps $\pi: X \to \textsf{Graph}(\G)$ such that $\pi_*\mu$ almost-every $G$ is connected and non-empty, but does not necessarily have full support. We will use this in the proof of Theorem \ref{T:ppp-cost}.

Suppose $\G \cc (X, \mu)$ is pmp and ergodic. Given a $\G$-equivariant map $\pi: X \to \textsf{Graph}(\G)$, let $\operatorname{Weight}(\pi)$ be the probability that the identity $e$ is in the support of the random graph $G$ where $G$ has law $\pi_*\mu$. For example, if $\pi$ is a $\G$-graph, then the weight of $\pi$ is 1.  





    

\begin{lem}\label{weight-cost}
    Let $\G \cc (X,\mu)$ be an essentially free probability measure preserving action. Then the $\mu$-cost of this action is
    \begin{align}\label{weight}
   \cost_\mu(\G,X) &=  \inf_\pi\left(1-\operatorname{Weight}(\pi)+ \frac{1}{2}\int \deg_e(G)~\dee \pi_*\mu(G) \right)
    \end{align}
    where the infimum is     over all $\G$-equivariant maps $\pi:X \to \mathsf{Graph}(\G)$ such that $\pi_*\mu$-almost every $G$ is connected and non-empty.
\end{lem}

\begin{proof}
    In the case that $\pi_*\mu$-almost every $G$ has full support, we have that $\text{Weight}(\pi) = 1$, and thus equation \ref{weight} is equal to the original graph-cost formula. Hence if we take the infimum over only these $\pi$, we get the cost by Corollary \ref{cost}. Hence we obtain the inequality
    $$    \cost_\mu(\G,X) \ge \inf_\pi \left(1- \text{Weight}(\pi)+ \frac{1}{2}\int \deg_e(G)~\dee \pi_*\mu(G)\right).$$


    Now suppose that $\pi_*\mu$-almost every $G$ is connected and non-empty but does not necessarily have full support. We need to show that $1- \text{Weight}(\pi)+ \frac{1}{2}\int \deg_e(G)~\dee \pi_*\mu(G) \geq \operatorname{grcost}_\mu(\cR_\G)$. For this let
    \begin{align*}
        S &= \{x\in X:~ \supp(\pi(x))\ni e\}\\
        \cG_\pi&=\{(x,gx) \in S \times S:~(e,g^{-1})\in \pi(x)\}.
    \end{align*}
We claim:
\begin{enumerate}
    \item $S$ is a complete section;
    \item $\cG_\pi$ is a graphing of $\cR_\G$ restricted to $S$;
    \item $\mu(S)=\text{Weight}(\pi)$.
\end{enumerate}
To see item (1), note that for a.e.\ $x\in X$, $\pi(x)$ is non-empty. Hence there exists $g\in \supp(\pi(x))$ which implies $e \in \supp(\pi(g^{-1}x))$ (by equivariance). Therefore $g^{-1}x \in S$. This shows $S$ is a complete section.

To see item (2), let $(x,gx)\in \cR_\G \cap (S\times S)$. Since $\pi(x)$ is connected and $\supp(\pi(x))$ contains both $e$ and $g^{-1}$ (since $\supp(\pi(gx))=g\supp(\pi(x))$), there is a path in $\pi(x)$ from $e$ to $g^{-1}$. That is, there are elements $e=g_1,\ldots, g_n=g$ such that $(g_i^{-1},g_{i+1}^{-1})\in \pi(x)$ for all $i<n$. Therefore, $(g_1x,g_2x), (g_2x,g_3x),\ldots, (g_{n-1}x,g_nx)$ is a path in $\cG_\pi$ from $x$ to $gx$. 

Item (3) is true by definition of $S$ and $\text{Weight}(\pi)$. It now follows from Gaboriau's Theorem \ref{section} that
   $$    \cost_\mu(\G,X) \le     1- \mu(S)+ \frac{1}{2}\int \deg_e(G)~\dee \pi_*\mu(G).$$
This proves the opposite inequality.
\end{proof}

\section{Cost of Poisson suspensions}\label{S:cost-Poisson}

The goal of this section is to prove Theorem \ref{T:ppp-cost}, which gives an upper bound for the cost of a Poisson suspension of an imp action under certain circumstances. First, we introduce the last necessary definition for the theorem.
Additional information on double recurrence is reviewed in Appendix \ref{S:recurrence}.

\begin{defn}
    Let $\G \cc (X,\mu)$ be an ergodic imp action. By Theorem \ref{T:Kaim-Hopf}, $X^2$ is the disjoint union of $\G$-invariant measurable sets $\mathsf{Con}(X^2)$ and $\mathsf{Dis}(X^2)$, and the restriction of $\G$ to $\mathsf{Con}(X^2)$ is infinitely conservative. We will say the action $\G \cc (X,\mu)$ is {\bf partially doubly recurrent} (PDR) if for a.e.\ $x,y \in X$ there exist $x=x_1,x_2,\ldots, x_n=y$ with $(x_i,x_{i+1})\in \mathsf{Con}(X^2)$ for all $i$. In other words, the equivalence relation generated by $\mathsf{Cont}(X^2)$ is all of $X$ (up to a set of measure zero). 
    
\end{defn}



\begin{thm}\label{T:ppp-cost} Suppose $\G$ is a countable group.
    Let $\G \cc (X,\mu)$ be a $\G$-invariant action such that a.e.\  ergodic component is infinite, non-atomic, essentially free and partially doubly recurrent. Then 
    $$\cost_{\Pois(\mu)}( \cR_\Pi)\le \operatorname{ncost}(\cR_\G)$$ where $\cR_\G$ is the equivalence relation induced by $\G \cc (X,\mu)$ and $\cR_\Pi$ is the equivalence relation induced by the pmp action $\G \cc(\M(X),\Pois(\mu))$.
\end{thm}

We will prove this after the next two lemmas.

\begin{lem}\label{L:cost-ess-free}
    Suppose $\G \cc (X,\mu)$ is an essentially free ergodic imp action. Then the action of $\G \cc (\M(X),\Pois(\mu))$ is essentially free.
\end{lem}

\begin{proof}
    Let $g\in \G$ be nontrivial. It suffices to show that $\Pois(\mu)$-a.e.\ $\Pi$, $g\Pi \ne \Pi$. The special case in which $g$ has infinite order is implied by \cite[Theorem 4.8]{MR2308588}. So suppose $g$ has finite order $n$. Because the action is essentially free, the Ergodic Decomposition Theorem implies there exists a measurable set $Y \subset X$ such that $X$ is the disjoint union of $Y,gY,\ldots, g^{n-1}Y$ (mod $\mu$).

Since $\mu(X)=\infty$, $\mu(Y)=+\infty$ too. Let $A \subset Y$ be a set with finite positive measure. If $\Pi = g^{-1}\Pi$ then $\Pi(A) = \Pi(gA)$. So   it suffices to show that 
    \begin{align}\label{E:Poisson-thing}
        \lim_{\mu(A)\to +\infty} \Pois(\mu)(\{\Pi:~ \Pi(A) = \Pi(gA)\}) = 0.
    \end{align}
where the limit is over all positive finite measure sets $A \subset Y$. 

By definition of $\Pois(\mu)$ (and since $A$ and $gA$ are disjoint),
\begin{align*}
    \Pois(\mu)(\{\Pi:~ \Pi(A) = \Pi(gA)\}) &=\sum_{k=0}^\infty  \left(\frac{\mu(A)^k}{e^{\mu(A)}k!}\right)^2 \\
    &\le \max_{k \in \{0\} \cup \N} \frac{\mu(A)^k}{e^{\mu(A)}k!}.
\end{align*}
The value of $k \in \{0\} \cup \N$ which maximizes the last expression is called the mode of a Poisson random variable with mean $\mu(A)$. It is known to equal $\lfloor \mu(A) \rfloor$. By Stirling's Approximation, this implies 
\begin{align*}
     \Pois(\mu)(\{\Pi:~ \Pi(A) = \Pi(gA)\})   &=O(1/\sqrt{\mu(A)}). 
\end{align*}
This implies \eqref{E:Poisson-thing}.

\end{proof}

\begin{proof}[Proof of Theorem \ref{T:ppp-cost}]
By the Ergodic Decomposition Theorem, it suffices to handle the special case in which the action is ergodic.

Let $\G \cc (X, \mu)$ be an ergodic, essentially free imp action where $\G$ is a finitely generated group. By Lemma \ref{L:cost-ess-free}, the action $\G \cc (\M(X), \Pois(\mu))$ is essentially free.


Let $S \subset X$ be a set with finite positive measure. Because the action is ergodic, this set is a  complete section mod $\mu$. Let $\cR_\G=\{(x,gx):~x \in X, g\in \G\}\subset X\times X$ be the orbit equivalence relation for $\G\cc(X,\mu)$. Also let $\cR_\G|S=\cR_\G \cap (S\times S)$ be the induced equivalence relation on $S$. Let $\eps>0$ and $\cG \subset \cR_\G|S$ be a graphing such that $\cost_{\mu|S}(\cG) < \cost(\cR_\G|S) + \epsilon/2$. 




Define $E: X \to \mathsf{Graph}(\G)$ by $E(x) = \{ (f, g) : (f^{-1}x, g^{-1}x) \in \cG\}$. Then $E$ is $\G$-equivariant and $\mathsf{supp}(E(x))=\Ret(x, S)^{-1} = \{g \in \G : g^{-1}x \in S\}$. Because $\cG$ is a graphing, $E(x)$ is connected for a.e.\ $x$. 

Define $\hE: \M(X) \to \mathsf{Graph}(\G)$  by 
 \[ \hE(\Pi) =\bigcup_{x \in \Pi} E(x)\]
where we have abused notation by writing $x\in \Pi$ as shorthand for $\Pi(x)>0$. Then $\hE$ is $\G$-equivariant (because $E$ is $\G$-equivariant) and $\mathsf{supp}(\hE(\Pi))= \cup_{x\in \Pi} \mathsf{supp}(E(x))$. Although $E(x)$ is connected for a.e.\ $x$, $\hE(\Pi)$ is not connected in general. 

In order to obtain a connected graph, we will take the union of $\hE(\Pi)$ with a Bernoulli edge-percolation. To make this precise, let $p:\G \to (0,1]$ be a positive function. Later on, we will be especially interested in the case in which $\sum_{g\in \G} p(g)$ is very small.

Let $\Bern$ be a random variable taking values in $\mathsf{Graph}(\G)$ defined by: for each $g,h\in \G$, the edge $\{g,gh\}$ is present in $\Bern$ with probability $p(h)$. Moreover, these events are jointly independent. This is a Bernoulli edge percolation. Its distribution is $\G$-invariant.


Let $\nu_p \in \Prob(\mathsf{Graph}(\G))$ be the law of $\Bern$. So $\nu_p$ is a product measure. In fact, we identify $\mathsf{Graph}(\G)$ with the product space $\{0,1\}^{{\G\choose 2}}$ (where ${\G\choose 2}$ is the set of all two-element subsets of $\G$) and then $\nu_p = \prod_{\{g,gh\}} \lambda_{\{g,gh\}}$ where $\lambda_{\{g,gh\}} = (1-p(h))\d_0 + p(h)\d_1$ and $\d_0$ is the Dirac mass concentrated on $\{0\}$, for example. This measure is $\G$-invariant.


\noindent {\bf Claim 1}. For 
$\mu\times \mu \times \nu_p$-a.e.\ $(x,y,G)$, if $(x,y)\in \mathsf{Cont}(X^2)$ then there is a connected component of $E(x)\cup E(y)\cup G$ containing $E(x)\cup E(y)$. 

\begin{proof}[Proof of Claim 1]

Fix $f,g \in \G$  be such that 
$$\mu^2(\mathsf{Cont}(X^2) \cap  fS\times gS)>0.$$

Then for a.e.\ $(x,y) \in \mathsf{Cont}(X^2) \cap fS\times gS$,
$$|\Ret(fS\times gS, (x,y))|=\infty.$$
That is, there exist infinitely many $h\in \G$ with $(hx,hy) \in fS\times gS$. We will show that almost surely there is an edge in $G$ from $E(x)$ to $E(y)$.

Note that $(hx,hy)\in fS \times gS$ if and only if  $h^{-1}f \in \mathsf{supp}(E(x))$ and $h^{-1}g \in \mathsf{supp}(E(y))$. 

Now let $G\sim \nu_p$ be a sample of the Bernoulli percolation. The events $\{h^{-1}f, h^{-1}g\} \in G$ are jointly independent as $h$ varies. Moreover, each of these events has probability $p(f^{-1}g)>0$ of occurring. So by the infinite monkey theorem, for $\nu_p$-a.e.\ $G$ there exists at least one $h\in \Ret(fS\times gS, (x,y))$ with $\{h^{-1}f, h^{-1}g\} \in G$. So there exists an edge in the graph $E(x)\cup E(y) \cup G$ from $E(x)$ to $E(y)$. 

So we have shown that a.e.\ $(x,y) \in fS\times gS$, the claim is true. This implies the claim in general because $\cup_{f,g} fS\times gS=X\times X$ (mod $\mu$) since $S$ is a complete section.
\end{proof}

Claim 1 implies:

\noindent {\bf Claim 2}. 
For $\Pois(\mu)\times \nu_p$-a.e.\ $(\Pi,G)$, there is a connected component of $\hE(\Pi)\cup G$ containing $\hE(\Pi)$.

Let $E^0(\Pi,G) \in \mathsf{Graph}(\G)$ be the component of $\hE(\Pi) \cup G$ which contains $\hE(\Pi \resto Y)$. Then $E^0$ is $\G$-equivariant (because $\hE(\Pi)$ is $\G$-equivariant).

We will now use the cost of $E^0$ to bound the cost of the Poisson suspension. 

\noindent {\bf Claim 3}. The cost of the action $\Gamma\cc  (\M(X),\Pois(\mu))$ is the same as the cost of the action $\Gamma\cc  (\M(X)\times \mathsf{Graph}(\G),\Pois(\mu)\times \nu_p)$.

\begin{proof}[Proof of Claim 3]
Because $\G \cc (\mathsf{Graph}(\G),\nu_p)$ is Bernoulli, Theorem 1.6 of \cite{MR3294302} implies the action $\Gamma\cc  (\M(X)\times \mathsf{Graph}(\G),\Pois(\mu)\times \nu_p)$ is weakly equivalent to $\Gamma\cc  (\M(X),\Pois(\mu))$. By Theorem 10.14 of \cite{Kechris-global-aspects}, the two actions have the same cost. This uses the fact that the action $\Gamma\cc  (\M(X),\Pois(\mu))$ is essentially free by Lemma \ref{L:cost-ess-free}. 
\end{proof}

By Lemma \ref{weight-cost}, the cost of $(\Gamma, \M(X)\times \mathsf{Graph}(\G),\Pois(\mu)\times \nu_p)$ is at most equal to
 $$1- \text{Weight}(E^0)+ \frac{1}{2}\int \deg_e(E^0(\Pi,G))~\dee \Pois(\mu)\times \nu_p(\Pi,G).$$

Since $E^0$ includes $\hE$, $\text{Weight}(E^0)\ge \text{Weight}(\hE)$. The weight of $\hE$ is the probability that $\Pi$ has points in $S$ (where $\Pi \sim \Pois(\mu)$). By definition of the Poisson point process, this probability is $1-\exp(-\mu(S))$. So 
\begin{align}\label{E:cost9}
    \text{Weight}(E^0) \ge 1-\exp(-\mu(S)).
\end{align}

Let $|p|=\sum_{g\in \G} p(g)$. Observe that
\begin{align}\label{E:cost8}
   \frac{1}{2}\int \deg_e(G)~\dee \nu_p(G) \le |p|.
\end{align}

Since 
$$\deg_e(E^0(\Pi,G))\le \deg_e(\hE(\Pi))+\deg_e(G),$$
this implies
\begin{align*}
    &\frac{1}{2}\int \deg_e(E^0(\Pi,G))~\dee \Pois(\mu)\times \nu_p(\Pi,G)\\
    &\le \frac{1}{2}\int \left(\deg_e(\hE(\Pi))+\deg_e(G)\right)~\dee\Pois(\mu)\times \nu_p(\Pi,G) \\
     &\le |p|+\frac{1}{2}\int \deg_e(\hE(\Pi))~\dee \Pois(\mu)(\Pi).
\end{align*}

\noindent {\bf Claim 4}. 
$$\E[\deg_e(\hE(\Pi))| |\Pi \cap S|=n] \le  2n \cdot \cost(\cG)/\mu(S).$$
That is, the expected value of the degree of $\hE(\Pi)$, conditioned on $|\Pi \cap S|=n$, is at most $2n \cdot \cost(\cG)/\mu(S).$

\begin{proof}[Proof of Claim 4]
Suppose $\Pi \cap S=\{x_1,\ldots, x_n\}$. Then $\hE(\Pi)=\cup_{i=1}^n E(x_i)$. So
$$\deg_e(\hE(\Pi)) \le \sum_{i=1}^n \deg_e(E(x_i)).$$
By linearity of expectation, it follows that
$$\E[\deg_e(\hE(\Pi))| |\Pi \cap S|=n] \le n\cdot \E[\deg_e(\hE(\Pi))| |\Pi \cap S|=1].$$
So suppose $\Pi\cap S=\{x\}$. Then $\hE(\Pi)=E(x)$ and $\deg_e(\hE(\Pi))=\deg_e(E(x))$. Moreover, if $\Pi$ is random with law $\Pois(\mu)$ conditioned on $|\Pi \cap S|=1$ then the random point $x \in \Pi\cap S$ has law $\frac{\mu|S}{\mu(S)}$. Therefore,
\begin{align*}
    \E[\deg_e(\hE(\Pi))| |\Pi \cap S|=1] &= \mu(S)^{-1} \int_S \deg_e(E(x))~\dee\mu(x) = 2 \cost(\cG)/\mu(S)
\end{align*}
where the last equality holds by definition of $\cost(\cG)$. 
\end{proof}

From Claim 4, it follows that
\begin{align}
   \int \deg_e(\hE(\Pi))~\dee \Pois(\mu)(\Pi)&=\sum_{n=0}^\infty  \E[\deg_e(\hE(\Pi))| |\Pi \cap S|=n] \Pois(\mu)(\{\Pi:~ |\Pi \cap S|=n\}) \nonumber\\
   &\le \sum_{n=0}^\infty (2n \cdot \cost(\cG)/\mu(S))\cdot \Pois(\mu)(\{\Pi:~ |\Pi \cap S|=n\})\nonumber \\ &= 2\cost(\cG) \label{E:cost-10}
\end{align}
where the last equality occurs because $\sum_{n=0}^\infty n \cdot \Pois(\mu)(\{\Pi:~ |\Pi \cap S|=n\})$ is the expected value of $|\Pi \cap S|$ which is $\mu(S)$.


By Lemma \ref{weight-cost}, \eqref{E:cost9}, \eqref{E:cost8} and \eqref{E:cost-10},
\begin{align*}
    \cost(\Gamma, \M(X),\Pois(\mu)) &\le 1- \textrm{Weight}(E^0)+ \frac{1}{2}\int \deg_e(E^0(\Pi))~\dee \Pois(\mu)(\Pi)\\
    &\le 1- (1-\exp(-\mu(S))) + |p| + \cost(\cG) \\
    &\le \exp(-\mu(S)) + |p| + \cost(\cR|S) + \eps.
\end{align*}
Since $p$ and $\eps$ are arbitrary, we get
\begin{align*}
    \cost(\Gamma, \M(X),\Pois(\mu)) 
    &\le \exp(-\mu(S))  + \cost(\cR|S).
\end{align*}
By definition of normalized cost,
$$\mathsf{ncost}(\G,X,\mu) = \mathsf{Cost}(\cR|S)+1-\mu(S).$$
So
\begin{align*}
    \cost(\Gamma, \M(X),\Pois(\mu)) 
    &\le \mathsf{ncost}(\G,X,\mu) + \exp(-\mu(S))   - 1 + \mu(S).
\end{align*}
However, $S$ is an arbitrary finite positive measure set. Since $\mu$ is non-atomic, we can choose $\mu(S)$ to be as small as we wish. Since $e^{-x}-1+x$ tends to zero as $x \searrow 0$, this implies the theorem.
\end{proof}


\section{A general criterion for fixed price 1}\label{S:general}


\begin{thm}\label{T:general}
    If $\G$ has an imp action   which is limit-amenable, partially doubly recurrent, and has normalized cost $p$ then $\G$ has max-cost at most $p$. In particular, if $p=1$ then $\G$ has fixed price 1.
\end{thm}

\begin{proof}

 If $\G$ is amenable then $\G$ has fixed price 1 and the theorem is trivial. So we may assume $\G$ is non-amenable. By Conjecture \ref{C:non-amen100} all limit-amenable mp actions of $\G$ are infinite. 
 
Suppose $\G \cc (X,\mu)$ is an imp action which is limit-amenable, partially doubly recurrent, and has normalized cost $p$. Then a.e.\ ergodic component of the action is also limit-amenable (by Conjecture \ref{C:limit-amen-erg-decomp2}), partially doubly recurrent (by Lemma \ref{L:conserv-erg-dec}) and there must be some ergodic component with normalized cost at most $p$ (since normalized cost behaves linearly with respect to ergodic decomposition).

So without loss of generality we may assume the action is ergodic. After taking the direct product with a Bernoulli shift if necessary, we may also assume the action is non-atomic and essentially free. This is because direct products with pmp actions preserve limit-amenability (by Theorem \ref{T:extensions}) and partial double recurrence (by Lemma \ref{L:finitemeasureextension}).

Theorem \ref{T:ppp-cost} implies the cost of the Poisson suspension  $\G \cc (\M(X), \Pois(\mu))$ has cost bounded above by the normalized cost of $\G \cc (X, \mu)$. Because the action is limit-amenable, by Corollary \ref{weakbern} the Poisson suspension is weakly contained in Bernoulli. Thus by the Abert-Weiss Theorem \cite{abert-weiss-2013}, $\G$ has maximal cost less than or equal to $\text{ncost}(\cR)$.   
\end{proof}

\section{Metric groups}\label{S:Metric-groups}

The main result of this section is that if a countable group $\G$ is equipped with a (quasi-) metric $d$ satisfying certain properties then there is an infinite $\G$-invariant measure $\mu$ on the space $\cH$ of horofunctions such that $\G \cc (\cH,\mu)$ is limit-amenable and doubly-recurrent. This is Theorem \ref{T:metric-group} and is stated in a more detailed manner in Theorem \ref{T:metric-group2}.

To begin, we define the properties of quasi-metrics, then the space of horofunctions, then we explore certain classes of $\G$-invariant measures on $\cH$. The last section \S \ref{S:metric-DR} proves the main result.

\begin{defn}\label{D:quasi-metric}
A {\bf quasi-metric} is essentially the same as a metric except that the triangle inequality holds only up to an additive constant. To be precise: a function $d:X \times X \to [0,\infty)$ is a {\bf quasi-metric} if there exists a constant $C_q\ge 0$ such that for all $x,y,z \in X$:
\begin{enumerate}
    \item $d(x,y)=0$ if and only if $x=y$;
    \item $d(x,y)=d(y,x)$;
    \item (quasi-triangle inequality) $d(x,z)\le d(x,y)+d(y,z)+C_q$.
\end{enumerate}
As usual, the closed ball of radius $r$ centered as $x\in X$ is $B(x,r)=\{y\in X:~d(x,y)\le r\}$. We often denote the closed ball of radius $r$ about $e\in\G$ as $B(r)$. Sometimes, we write this as $B(\G,x,r)$ or $B(\G,r)$ to emphasize the role of the group $\G$.
\end{defn}

\begin{remark}
    If $d:X \times X \to [0,\infty)$ is a metric then there is an integer-valued quasi-metric $\overline{d}:X\times X \to \Z$ defined by $\overline{d}(x,y) = \lceil d(x,y) \rceil$. This construction is the main reason for introducing quasi-metrics; because certain arguments are easier when the metric takes on only integer values.
\end{remark}

\begin{defn}\label{D:qm-properties}
Let $d$ be a quasi-metric on a countable group $\G$. We say 
\begin{itemize}
    \item $d$  is {\bf proper} if every ball of finite radius is finite;
    \item $d$ is {\bf left-invariant} if $d(gh,gf)=d(h,f)$ for all $f,g,h\in \G$;
    \item $d$ is {\bf $\eps$-approximately sub-additive} if there is an $\eps>0$ such that    if 
    $$SS(\G,n,\eps)=\{x\in \G:~d(x,e) \in [n-\eps,n+\eps]\}$$
is the spherical shell of width $2\eps$ then for every $n,m\ge \eps$,
$$SS(\G,n,\eps)\cdot SS(\G,m,\eps) \supset S(\G,n+m)$$
where $S(\G,n+m)=\{x\in \G:~d(x,e)=n+m\}$ is the sphere of radius $n+m$. The reason for the terminology is this property implies the sequence $\{\log(|B(\G,n)|)+C\}_{n=1}^\infty$ is sub-additive for some constant $C>0$. This is proven in Lemma \ref{L:growth0}.

\end{itemize}
Additionaly, the \textbf{exponential growth rate} of $(\G,d)$ is defined by
    $$\mathsf{growth}(\G,d) = \lim_{n\to\infty} \frac{\log |B(n)|}{n}>0$$
    assuming the limit exists. 
\end{defn}

\begin{remark}
    Word metrics are proper, left-invariant and approximately sub-additive with $\eps=0$. It may help to just assume $d$ is a word metric on first reading. 
\end{remark}

\begin{lem}\label{L:growth0}
    If $d$ is proper and $(\G,d)$ is approximately sub-additive then the growth rate $\mathsf{growth}(\G,d)$ exists. 
\end{lem}

\begin{proof}
Because $(\G,d)$ is approximately sub-additive,  $B(n+m) \subset B(n+\eps)B(m+\eps)$ for all $n,m \ge 0$ (since $B(n+\eps)$ contains the spherical shell $SS(\G,n,\eps)$ for example). Moreover, $B(n+\eps) \subset B(n)B(3\eps)$. So 
    $$B(n+m) \subset B(n)B(3\eps)B(m)B(3\eps).$$
    Thus the sequence $a_n = \log |B(n)|$ is approximately sub-additive in the sense that there is a constant $C=2\log|B(3\eps)|$ such that $a_{n+m}\le a_n + a_m + C$. But this implies that the sequence $\{a_n+C\}$ is sub-additive. Fekete's Lemma implies
    $$\mathsf{growth}(\G,d)=\lim_{n\to\infty} \frac{a_n+C}{n} =\lim_{n\to\infty} \frac{a_n}{n}$$
    exists.
\end{proof}

\begin{cor}\label{C:generating}
    If $(\G,d)$ is proper and $\eps$-approximately sub-additive where $\eps>0$, then $B(3\eps)$ is a finite generating set for $\G$. 
\end{cor}

\begin{proof}
As in the proof of the previous lemma, $B(n+\eps) \subset B(n)B(3\eps)$ for any $n\ge \eps$. So, by induction, $B(3\eps)^n \supset B((n+2)\eps)$. Since $n$ is arbitrary and $d$ is proper this implies $B(3\eps)$ is a finite generating set.
\end{proof}

\begin{notation}
   From now on we fix an integer-valued, left-invariant, proper, approximately sub-additive quasi-metric on a group $\G$. We also assume $(\G,d)$ is non-amenable. This implies it has positive exponential growth. We let $B(g,r)\subset \G$ denote the ball of radius $r$ centered at $g$. We also write $B(r)=B(e,r)$ and $|x|=d(e,x)$ for $x\in \G$. Since $d$ is integer-valued, we may as well assume $\eps \ge 1$ and $C_q\ge 0$ (the quasi-metric constant) are integer-valued as well.    
\end{notation}

\begin{defn}\label{D:onp}
    We say $(\G,d)$ satisfies the {\bf overlapping neighborhoods property} (ONP) if there exists a constant $C>0$ such that for all $m>0$, 
    \begin{align}\label{E:ONP}
\lim_{r\to\infty} \liminf_{n\to\infty}  \frac{\#\{(x,y) \in B(n)^2:~|B(r) \cap B(x,n+C) \cap B(y,n+C)| < m\}}{|B(n)|^2} =0.
\end{align}
Intuitively, this means that if two radius $n$ balls intersect non-trivially then their radius-$C$ neighborhoods are likely to have large overlap, where $C$ is a constant and the size of the overlap depends on $n$. 
\end{defn}

The main result of this section is:

\begin{thm}\label{T:metric-group}
   If $d$ is a left-invariant, proper, approximately sub-additive integer-valued quasi-metric on a non-amenable group $\G$ and $(\G,d)$ has the overlapping neighborhoods property (ONP) then there exists a limit-amenable doubly-recurrent imp action $\G \cc (\cH,\mu)$.
\end{thm}

\begin{remark}
It follows from Theorem \ref{T:general} that the max-cost of $\G$ is at most the normalized cost of the action in Theorem \ref{T:metric-group}. We will later apply this to certain product groups where we can prove that the normalized cost of the action in this theorem is 1.
\end{remark}

\begin{cor}
    If $(\G,d)$ has the overlapping neighborhoods property and $\G$ is exact then $\G$ has fixed price 1.
\end{cor}

\begin{proof}
Because $\G$ is exact, all limit-amenable actions are amenable (Theorem \ref{T:exact}) and therefore have normalized cost 1. So this follows from Theorem \ref{T:metric-group} and Theorem \ref{T:general}.    
\end{proof}

\subsection{Horofunctions}

Let $\Lip(\G)$ be the space of all $1$-Lipschitz functions $h:\G \to \Z$ with the topology of pointwise convergence on finite subsets (where $h$ is $1$-Lipschitz if $|h(g)-h(f)|\le d(g,f)$ for all $g,f \in \G)$. Also let $\Lip_0(\G)=\{h\in \Lip(\G):~h(e)=0\}$. By the Ascoli-Arzela Theorem, $\Lip_0(\G)$ is compact.

For each $x\in \G$, let $d_x \in \Lip(\G)$ be the distance-to-$x$ function. That is:
$$d_x(y)=d(x,y).$$
Also define $h_x \in \Lip_0(\G)$ by $h_x=d_x-|x|$. This is the horofunction associated with $x$ normalized so that $h_x(e)=0$.  

Let $\cH^{\circ}_0=\{h_x:~x\in \G\}$ and let $\cH_0$ be the closure of  $\cH^\circ_0$ in $\Lip_0(\G)$. Because $\Lip_0(\G)$ is compact, $\cH_0$ is also compact. It is called the {\bf horofunction compactification of $\G$}. 

\begin{defn}\label{D:horospace}
    Let $\cH=\cH(\G)=\cH_0+\Z \subset \Lip(\G)$ which is the set of all functions of the form $h=h_0+r$ with $h_0 \in \cH_0 $ and where $r\in \Z$ is a constant. Because $\cH_0$ is compact, $\cH$ is locally compact. We call $\cH$ the space of {\bf horofunctions on $\G$}. Sometimes we emphasize the role of the group by writing $\cH(\Gamma)$ to mean $\cH$ (and similarly with $\cH_0$, etc). 
\end{defn}
Observe that $\G$ acts continuously on $\cH$ by 
$$(g\cdot h)(x)=h(g^{-1}x)$$
for $g,x\in \G$ and $h \in \cH$. Also let $$\partial \cH_0 = \cH_0 \setminus \cH^{\circ}_0, \quad \partial \cH = \partial\cH_0 + \Z.$$
We also let $\cH_{\le n} = \{h \in \cH:~h(e)\le n\}$ and 
$\partial \cH_{\le n} = \cH_{\le n} \cap \partial \cH$.

\begin{lem}\label{L:section}
    The space $\partial \cH_{\le 0}(\G)$ is a complete section for the action of $\G$ on $\partial \cH(\G)$.
\end{lem}

\begin{proof}
    Let $h \in \partial \cH(\G)$. It suffices to show there exists $g\in \G$ with $h(g)\le 0$. If $h(e)\le 0$ then we are done. So we will assume $h(e)>0$. 

    Let $h_0 \in \partial \cH_0(\G)$ be the function $h_0(x)=h(x)-h(e)$. By definition of $\partial \cH_0(\G)$, there exists a sequence $\{x_n\}_{n=1}^\infty \subset \G$ diverging to infinity with $h_0 = \lim_{n\to\infty} d_{x_n} - |x_n|$. 

    Because $(\G,d)$ is $\eps$-approximately sub-additive, $x_n$ is contained in
    $$SS(h(e)+\eps,\eps) \cdot SS(|x_n|-\eps-h(e),\eps)$$
    for all sufficiently large $n$.    Thus there exists $y_n \in SS(h(e)+\eps,\eps)$ and $z_n \in SS(|x_n|-\eps-h(e),\eps) $ with $x_n=y_nz_n.$ After passing to a sub-sequence if necessary, we may assume $y_n=y$ is constant. Note $d(x_n,y_n) = |y_n^{-1}x_n|=|z_n|$. So
    \begin{align*}
h_0(y) &= \lim_{n\to\infty} d(x_n,y) - |x_n| =   \lim_{n\to\infty} |z_n| - |x_n| 
\in [-h(e)-2\eps, -h(e)].
    \end{align*}
Thus $h(y) = h(e)+h_0(y)\in [-2\eps,0]$.  
\end{proof}

\begin{defn}
For $t\in \Z$, define $\Add:\Lip(\G)  \to \Lip(\G)$ by setting $\Add(\xi)=\xi-1$.
\end{defn}

\begin{remark}
The canonical horoball associated with a horofunction $h$ is the set $B(h)=h^{-1}((-\infty,0])$. Note that $B(\Expand(h))=h^{-1}((-\infty,1])$ contains $B(h)$. This is why we call the map `expand' because it is expands canonical horoballs. We won't actually need these canonical horoballs, but they are useful to keep in mind for intuition building. 
 \end{remark}
The next lemma follows immediately from the definitions.
\begin{lem}\label{L:expand-properties}
    The map $\Add$ is $\G$-equivariant, continuous and $\Add(\cH)=\cH$.
\end{lem}

\subsection{Measures on the space of horofunctions}\label{S:metric-measures}

For $n \ge 0$, let $\mu_{n}$ be the measure on $\cH(\G)$ defined by
$$\mu_{n} = \frac{1}{|B(n)|} \sum_{x\in \G} \d_{d_x-n}$$
where $\d_{d_x-n}$ is the Dirac delta mass on the function $d_x-n \in \cH(\G)$. 

Equivalently, if 
$$\mu_{0} = \sum_{x\in \G} \d_{d_x}$$
then 
\begin{align}\label{E:mu_n}
    \mu_{n} = \frac{\Add^n_{*}\mu_{0}}{|B(n)|}.
\end{align}

\begin{lem}\label{L:growth1}
For $t\in \Z$, let $\cH_{\leq t} = \{h\in \cH:~ h(e)\leq t\}$. Then for each $n$, $\mu_n$ is $\G$-invariant and $\mu_n(\cH_{\leq t}) = \frac{|B(n+t)|}{|B(n)|}$. If $t\ge 0$, then $\mu_n(\cH_{\leq t}) \le |B(3\eps)|^{t}$. 
\end{lem}

\begin{proof}
Because $\mu_{n} = \frac{\Add^n_{*}\mu_{0}}{|B(n)|}$ and $\Add^n$ commutes with the action of $\G$, to prove that $\mu_n$ is $\G$-invariant, it suffices to prove that $\mu_0$ is $\G$-invariant. 
This is true because for any $f,x\in \G$, $f \d_{d_x}=\d_{f\cdot d_x}$ and $f\cdot d_x = d_{fx}$ since $d$ is a left-invariant metric. 
So 
$$f \mu_0 = \sum_{x\in \G} \d_{d_{fx}} = \mu_0.$$
This proves $\mu_n$ is $\G$-invariant.

By definition,
\begin{align*}
   \mu_n(\cH_{\leq t}) &= \frac{\#\{x\in \G:~(d_x-n)(e)\leq t\}}{|B(n)|}=\frac{|B(n+t)|}{|B(n)|}. 
\end{align*}

If $t\ge 0$, then as in the proof of Corollary \ref{C:generating} $B(n+t) \subset B(n)B(3\eps)^{t/\eps}$. This implies the last statement (since $\eps\ge 1$). 
\end{proof}

\begin{cor}\label{C:non-amenable}
There is a constant $C_{na}>0$ such that  for all $m \ge 3\eps+C_q$ and $n \in \N$, $\mu_n(\cH_{\le -m}(\G))\le e^{-C_{na}\cdot m}$.   
\end{cor}

\begin{remark}
    We call $C_{na}$ the non-amenability constant.
\end{remark}

\begin{proof}
   By Lemma \ref{L:growth1}, it suffices to prove that 
   $$\liminf_{n\to\infty} \frac{|B(n-m)|}{|B(n)|}=\exp(-C_{na}\cdot m)$$
   for some constant $C_{na}>0$ and all $m \in \N$. 

   Let $k=3\eps+C_q \ge 1$.    It suffices to prove the special case in which $m=k$ since
   $$\liminf_{n\to\infty} \frac{|B(n-m)|}{|B(n)|} \ge \liminf_{n\to\infty}\left(\frac{|B(n-k)|}{|B(n)|}\right)^{\lfloor m/k \rfloor}.$$
   So it suffices to assume 
    $$\liminf_{n\to\infty} \frac{|B(n-k)|}{|B(n)|}=1$$
    and obtain a contradiction.
   
   We claim $\{B(kn)\}_{n=1}^\infty$ is a F\o lner sequence for $\G$. This is because $B(kn)B(3\eps) \subset B(kn+3\eps + C_q) \subset B(k(n+1))$. So 
$$\frac{|B(kn)|}{|B(kn)B(3\eps)|} \ge \frac{|B(kn)|}{|B(k(n+1))|}\to 1$$
as $n\to\infty$. On the other hand, by Corollary \ref{C:generating}, $B(3\eps)$ is a generating set. So this proves $\{B(kn)\}_{n=1}^\infty$ is F\o lner, contradicting the assumption that $\G$ is non-amenable.
\end{proof}



\subsection{Spaces of measures}

\begin{defn}
  Let $\Radon(\cH)$ be the set of Radon measures on $\cH$. We will give this space a topology that lies between the vague topology and the weak topology. Let us say that a function $f:\cH \to \R$ has {\bf upper-bounded support} if there is some number $n$ such that $f$ is supported on $\cH_{\le n}$. Let $C_{ub}(\cH)$ be the space of all bounded continuous functions with upper-bounded support. We say that a sequence $(\nu_n)_{n=1}^\infty \subset \Radon(\cH)$ converges {\bf almost weakly} to a measure $\nu_\infty$ if $\nu_n(f)$ converges to $\nu_\infty(f)$ for all non-negative $f \in C_{ub}(\cH)$.  This defines a topology on $\Radon(\cH)$ which we call the {\bf almost-weak topology}.
\end{defn}

\begin{lem}\label{L:almost-weak}
    Let $(\nu_n)_{n=1}^\infty \subset \Radon(\cH)$ and $\nu_\infty \in \Radon(\cH)$. Then the following are equivalent:
    \begin{enumerate}
        \item $\nu_n$ converges almost weakly to $\nu_\infty$ as $n\to\infty$;
        \item for every $t\in \Z$, the restriction of $\nu_n$ to $\cH_{\le t}$ converges weakly to the restriction of $\nu_\infty$ to $\cH_{\le t}$.
    \end{enumerate}
\end{lem}

\begin{proof}
    This holds because $\cH_{\le t}$ is clopen in $\cH$. So any continuous function on $\cH_{\le t}$ can be continuously extended to all of $\cH$ by defining it to be zero on the complement of $\cH_{\le t}$.
\end{proof}

\begin{defn}\label{D:limit-measures}
Let $\Meas(\cH)=\{\mu_n\}_{n=1}^\infty$ be the sequence of measures defined in the previous section. Let $\overline{\Meas(\cH)}$ denote the almost-weak closure of $\Meas(\cH)$ in $\Radon(\cH)$. Also let $\partial \Meas(\cH) = \overline{\Meas(\cH)} \setminus \Meas(\cH)$.
\end{defn}

\begin{lem}\label{L:compact}
    $\overline{\Meas(\cH)}$ is compact in the almost-weak topology.
\end{lem}

\begin{proof}
    By Lemma \ref{L:almost-weak} and Prokhorov's Theorem, it suffices to prove for every $t\in \Z$
    \begin{enumerate}
        \item the sequence $\{\mu_n\resto \cH_{\le t}\}_{n=1}^\infty$ is tight in the sense that for every $\delta>0$ there exists a compact set $K \subset \cH_{\le t}$ such that $\mu_n(\cH_{\le t} \setminus K)<\delta$ for all $n$;
        \item the sequence of real numbers $\{\mu_n(\cH_{\le t})\}_{n=1}^\infty$ is bounded.
    \end{enumerate}

To prove item (1), let $\delta>0$ and fix $t\in \Z$. Let $N \in \N$ be large enough so that $e^{-C_{na}\cdot N} < \delta$ and $t \ge -N$. Note that $K = \cH_{[-N,t]}$ is compact. By Corollary \ref{C:non-amenable}, $\mu_n(\cH_{\le t} \setminus K) \le e^{-C_{na}\cdot N} < \delta$. This proves item (1). 

By Lemma \ref{L:growth1}, if $t \ge 0$ then $\mu_n(\cH_{\le t}) \le |B(3\eps)|^t$. This proves item (2). 
\end{proof}

    \begin{remark}
    This Lemma is the main reason why we work with integer-valued quasi-metrics instead of metrics. Because $d$ is integer-valued,  $\cH_{\le t}$ is closed and open in $\cH$ which is useful in defining the almost-weak  topology on $\overline{\Meas}(\cH)$.
\end{remark}

It should be noted that every measure $\mu \in \overline{\Meas}(\cH)$ is $\G$-invariant since each $\mu_n$ is $\G$-invariant and almost-weak limits preserve $\G$-invariance (since almost-weak convergence implies vague convergence).

\begin{lem}\label{L:support}
  Let $(n_i)_{i=1}^\infty$ be a sequence tending to infinity with $\mu = \lim_{i\to\infty} \mu_{n_i}$ (almost weakly). Then $\mu \in \partial \Meas(\cH)$ and $\mu$ is supported on $\partial \cH(\G)$. Conversely, every measure in $\partial \Meas(\cH)$ is of this form.
  In particular, by Lemma \ref{L:compact}, this implies $\partial \Meas(\cH)$ is non-empty.
\end{lem}

\begin{proof}
    Suppose $h\in \cH$ is in the support of $\mu$.     Since $\mu_{n_i}$ converges to $\mu$ in the vague topology, for each relatively compact open neighborhood $U \subset \cH$ of $h$ there exists infinitely many $i$ such that $\mu_{n_i}(U)>0$. Thus for each $n_i$ large enough, there exists $x_i \in \G$ such that $d_{x_i} - n_i \in U$. Because $U$ is arbitrary, there must exist group elements $x_i \in \G$ such that $(d_{x_i} - n_i) \longrightarrow h$ pointwise on $\G$ as $i\to\infty$.

  Observe that
$$h(e) = \lim_{i\to\infty}|x_i| - n_i$$
is finite. Since $n_i \to \infty$ as $i\to\infty$, this implies $|x_i|\to \infty$ as $i\to\infty$.   

       To obtain a contradiction, suppose $h=d_y-n$ for some $y\in \G$ and $n\in \Z$. Then $h(z) \ge -n$ for all $z \in \G$.

Let $m>n+\eps$. Because $(\G,d)$ is approximately sub-additive, 
$$SS(\G,m,\eps)\cdot SS(\G,|x_i|-m,\eps)\supset S(\G, |x_i|). $$
So there exist elements $a_i \in SS(\G,m,\eps)$, $b_i \in SS(\G,|x_i|-m,\eps)$ with $a_ib_i = x_i$. Since $|a_i|\le m+\eps$ is bounded, after passing to a sub-sequence if necessary, we may assume there exists $a\in \G$ with $a= a_i$ for all $i$. 

Note
\begin{align*}
    h(a) &= \lim_{i\to\infty} d(x_i,a) - n_i =  \lim_{i\to\infty} |a^{-1}x_i| - n_i  = \lim_{i\to\infty} |b_i| - n_i.
\end{align*}
Since $|b_i| = |x_i|-m$ up to an error of $\pm \eps$ and $h(e)=\lim_{i\to\infty} |x_i| - n_i$, it follows that
$$h(a) \in [-m-\eps, -m+\eps].$$
Since $-m+\eps < -n$, this contradiction shows that $h\ne d_y -n$ for any $y,n$. Since $h$ is arbitrary,  $\mu$ is supported on $\partial \cH$. Because $\mu$ is supported on $\partial \cH$ it cannot equal $\mu_n$ for any $n$. So $\mu \in \partial \Meas(\cH)$.

Conversely, if $\mu \in \partial \Meas(\cH)$ then, by definition, there exists a sequence $(n_i)_{i=1}^\infty$ such that $\mu$ is the almost-weak limit of $\mu_{n_i}$ as $i\to\infty$. Because $\mu$ is not in $\Meas(\cH)$, it follows that $(n_i)_{i=1}^\infty$ must diverge to infinity. 
\end{proof} 

\begin{lem}\label{L:metric-amenable}
    For every $\mu \in \partial\Meas(\cH)$, the action $\G \cc (\cH,\mu)$ is limit-amenable.
\end{lem}

\begin{proof}
The action $\G \cc (\cH,\mu_0)$ is measurably conjugate to the left-regular action of $\G$ on itself. So it regular. By Theorem
\ref{T:limit-regular}, an action is limit-amenable if and only if it is limit-regular.
By definition there exists a sequence $\mu_{n_i}$ converging vaguely to $\mu$, so it is sufficient to show that there exists a measure-preserving factor map from $(\cH, \mu_0) \to (\cH, \mu_n)$ for each $n$.

We can see that $\Expand^n$ is in fact a measure-conjugacy-up-to-scalars between $(\cH, \mu_0)$ and $(\cH, \mu_n)$: By Lemma \ref{L:expand-properties} $\Expand^n$ is $\G$-equivariant and bijective for each $n$. Furthermore, it is clear that the map $\Expand^{-n}$ is a measurable inverse of $\Expand^n$. Finally, if $A \subseteq \cH$ is a measurable set, then $\mu_n(A) = \frac{1}{|B(n)|}\cdot \mu_0(\Expand^{-n}(A))$. Thus normalizing $\Expand^n$ gives a measure-preserving factor map, proving the lemma.
\end{proof}

\subsubsection{Expansion-invariant measures}

The goal of this section is to prove the existence of an $\Add_*$-quasi-invariant measure which is a convex integral of measures in $\partial\Meas(\cH)$. Recall that $\Add:\Lip(\G)\to \Lip(\G)$ is defined by $\Add(\xi)=\xi-1$.

\begin{prop}\label{P:regular}
There exists a $\G$-invariant measure $\mu$ on $\cH$ such that $\mu$ is equivalent to $\Expand_{*}\mu$ and $\mu = \int \nu~d\zeta(\nu)$ for some Borel probability measure $\zeta$ on $\partial\Meas(\cH)$. In particular, the action $\G \cc (\cH,\mu)$ is limit-amenable.
\end{prop}

\begin{proof}

Define $T: \partial\Meas(\cH) \to \partial\Meas(\cH)$ by
$$T(\mu) = \frac{\Add_{*}\mu}{\mu(\cH_{\le 1})}.$$
Then $T$ is continuous in the almost-weak topology because $\Add$ is continuous and the map which sends $\mu$ to $\mu(\cH_{\le 1})$ is continuous. This is by the Portmanteau Theorem using the fact that $\cH_{\le 1}$ is both closed and open. This is one of the reasons why we work with the almost-weak topology instead of the vague topology.

By the way, the reader might wonder why $T$ maps $\partial\Meas(\cH)$ into $\partial\Meas(\cH)$. This is because $\frac{|B(n)|}{|B(n+1)|}\Add_*\mu_n = \mu_{n+1}$ by \eqref{E:mu_n}. So if $\mu \in \partial\Meas(\cH)$ is is the limit of a sequence of measures $\{\mu_{n_i}\}_{i=1}^\infty$ then $T(\mu)$ the limit of the sequence of measures $\{\mu_{n_i+1}\}_{i=1}^\infty$.

By the Kyrylov-Bogolyubov fixed point Theorem, there is a $T$-invariant Borel probability measure $\zeta$ on $\partial\Meas(\cH)$. This means $T_*\zeta=\zeta$. This uses the compactness Lemma \ref{L:compact}.

Define a measure $\mu$ on $\cH$ by
$$\int f~d\mu = \iint f~d\nu ~d\zeta(\nu)$$
for all compactly supported continuous functions $f$ on $\cH$. By the Riesz-Markov Theorem, this defines a measure $\mu$ on $\cH$.  

Because each $\nu$ in the support of $\zeta$ is limit-amenable, it follows from Theorem \ref{C:limit-amen-erg-decomp2}, that the action $\G \cc (\cH,\mu)$ is limit-amenable. Indeed, the bounds from Lemma \ref{L:growth1} and Corollary \ref{C:non-amenable} imply that the measures in $\partial \Meas(\cH)$ are uniformly bounded on compacts.
\end{proof}

\subsection{Double recurrence}\label{S:metric-DR}

The main result of this subsection is the following refinement of Theorem \ref{T:metric-group}:

\begin{thm}\label{T:metric-group2}
As above, we assume $d$ is an integer-valued, left-invariant, proper, approximately sub-additive quasi-metric on $\G$. Let $\mu$ be a measure on the space of horofunctions $\cH=\cH(\Gamma)$ satisfying the conclusions of Proposition \ref{P:regular}.  If $(\G,d)$ has the overlapping neighborhoods property then $\G \cc (\cH(\G),\mu)$ is doubly-recurrent.
\end{thm}

Throughout this section we will assume that $\G$ satisfies the overlapping neighborhoods property. Let $C>0$ be the constant in the definition of that property. Let $C_q$ be the constant in the definition of quasi-metrics, and $\epsilon>0$ be such that $d$ is $\epsilon$-approximately sub-additive. 


We use the following notation: $\cH^2=\cH\times \cH$, $\mu^2=\mu\times \mu$ and if $\cS \subset \cH^2$ is any subset and $h=(h_1,h_2)\in \cH^2$ then $\Ret(\cS, h)$ is the set of all group elements $g\in \G$ such that $gh \in \cS$ (where $\G$ acts on $\cH^2$ diagonally). We need to show that for all finite measure subsets $\cS\subset \cH^2$, for a.e.\ $h \in \cS$, $\Ret(\cS, h)$ is infinite. 

We begin by considering very specific subsets $\cS$. For $g \in \G$ and $m,r\ge 0$, let
\begin{align*}
    Z_g&=\{(h_1,h_2)\in \cH^2:~ h_1(e)<0, h_2(g)<C_q\}, \\
    Z'_g&=\{(h_1,h_2)\in \cH^2:~ h_1(e)<-C-|g|, h_2(g)<-C-2|g|-C_q\}, \\
     Z'_{g,m}&=\{h\in Z'_g:~ |\Ret(Z_g, h)|\ge m\}, \\
    Z'_{g,m,r} &=\{h\in Z'_g:~ |\Ret(Z_g, h) \cap B(r)|\ge m\}. 
\end{align*}
Because $\cH^2$ is the union of $f\cdot Z_g$ (over all $f,g \in \G$), to prove double-recurrence, it suffices to show that $\Ret(Z_g,h)$ is infinite for $\mu^2$-a.e.\ $h\in Z_g$. The next lemma proves the weaker statement that this holds for $\nu^2$-a.e.\ $h\in Z'_g$ (where $\nu \in \partial \Meas(\cH)$ is arbitrary). Afterwards, we use $\Add$-quasi-invariance of $\mu$ to amplify this statement to $\mu^2$-a.e.\ $h\in Z_g$ and thereby obtain double-recurrence.

\begin{lem}\label{L:half-way}
Let $\nu \in \partial\Meas(\cH)$ and $g\in \G$. 
Then for $\nu^2$-a.e.\ $h=(h_1,h_2)\in Z'_g$, $\Ret(Z_g, h)$ is infinite.
\end{lem}

\begin{proof}
It suffices to prove that for every $m>0$ 
\begin{align}\label{E:Z'-g}
 \lim_{r\to\infty}   \nu^2(Z'_g \setminus Z'_{g,m,r})=0.
\end{align}
To see this, assume \eqref{E:Z'-g}. Then for all $m>0$, for $\nu^2$-a.e.\ $h \in Z'_g$ there is an $r$ (depending on $m$ and $h$) such that $h\in Z'_{g,m,r} \subset Z'_{g,m}$. Thus $Z'_g \subset \bigcap_{m=1}^\infty  Z'_{g,m}$ (up to measure zero) which implies $\nu^2$-a.e.\ $h\in Z'_g$ has 
infinite return times to $Z_g$: $|\Ret(Z_g, h)|=\infty$.

For the remainder of the proof, let $\nu \in \partial\Meas(\cH)$. Thus there exists a sequence $(\mu_{n_i})_{i=1}^\infty$ with $n_i\to\infty$ such that $\mu_{n_i}$ converges almost-weakly to $\nu$ and each $\mu_{n_i}$ is of the form $\mu_{n_i}=|B(n_i)|^{-1}\sum_{x\in\G} \delta_{d_x-n_i}$.

Both sets $Z'_g$ and $Z'_{g,m,r}$ are closed and open in $\cH$ and both are contained $Z_g = \cH_{<0} \times g\cH_{<C_q}$. Because the restriction of $\mu_{n_i}$ to $\cH_{\le t}$ converges weakly to the restriction of $\nu$ to $\cH_{\le t}$ for any $t$, we must also have that the restriction of $\mu^2_{n_i}$ to $Z'_g$ converges weakly to the restriction of $\nu^2$ to $Z'_g$.   So the Portmanteau Theorem implies for any $m,r>0$
\begin{align}
 \lim_{i\to\infty}  \mu^2_{n_i}(Z'_g \setminus Z'_{g,m,r})= \nu^2(Z'_g \setminus Z'_{g,m,r}).
\end{align}
This equality is one of the main reasons for using the almost-weak topology instead of the vague topology. 

It now suffices to prove 
$$\lim_{r\to\infty} \liminf_{i\to\infty}  \mu^2_{n_i}(Z'_g \setminus Z'_{g,m,r}) = 0.$$

By definition of $\mu_{n}$, 
\begin{align*}
\mu^2_n(  Z'_g \setminus Z'_{g,m,r}) &= \frac{\#\{(x,y) \in B(e,n)\times B(g,n):~(d_x-n,d_y-n) \in Z'_g \setminus Z'_{g,m,r}\}}{|B(n)|^2}.
\end{align*}
This is because if $(d_x-n,d_y-n) \in Z'_g$ then $(x,y) \in B(e,n)\times B(g,n)$:
For the first coordinate it is clear that $x$ must be in $B(n)$ to ensure that $d_x(e)-n = |x| - n<-C-|g|$. 
For the second coordinate, if $d_y(g)-n<-C-2|g|-C_q<0$ then $d(y,g)<n$.

By definition, $(d_x-n,d_y-n) \in Z'_g$ if and only if 
\begin{align*}
    |x| &\le n - C-|g|,\quad 
    d(y,g)  \le n - C - C_q- 2|g|.
    \end{align*}
On the other hand, if $(d_x-n,d_y-n)\in Z'_g$ and $f \in \Ret(Z_g, (d_x-n,d_y-n))$ then $f(d_x-n)(e)\le 0$ and $f(d_y-n)(g)\le C_q$. Equivalently, 
\begin{align*}
d(x,f^{-1})\le n, \quad d(y,f^{-1}g)\le n+C_q.
\end{align*}
This is equivalent to
$$f^{-1} \in B(x,n) \cap B(y,n+C_q)g^{-1}.$$
Letting $N=n-C-|g|$ and $M=n-C-2|g|-C_q$ it suffices to show
\begin{align*}
\lim_{r\to\infty} \liminf_{n\to\infty}  \frac{\#\{(x,y) \in B(e,N)\times B(g,M):~|B(r) \cap B(x,n) \cap B(y,n+C_q)g^{-1}| < m\}}{|B(n)|^2} =0.
\end{align*}
The overlapping neighborhoods property is equivalent to: for every $m>0$
\begin{align*}
\lim_{r\to\infty} \liminf_{n\to\infty}  \frac{\#\{(x,y) \in B(n)^2:~|B(r) \cap B(x,n+C) \cap B(y,n+C)| < m\}}{|B(n)|^2} =0.
\end{align*}
By replacing $n$ with $n-C-|g|$ we see that this is equivalent to:
\begin{align*}
\lim_{r\to\infty} \liminf_{n\to\infty}  \frac{\#\{(x,y) \in B(n-C-|g|)^2:~|B(r) \cap B(x,n-|g|) \cap B(y,n-|g|)| < m\}}{|B(n-C-|g|)|^2} =0.
\end{align*}
We claim that
$$\{(x,y) \in B(e,N)\times B(g,M):~|B(r) \cap B(x,n) \cap B(y,n+C_q)g^{-1}| < m\}$$
is contained in
$$\{(x,y) \in B(n-C-|g|)^2:~|B(r) \cap B(x,n-|g|) \cap B(y,n-|g|)| < m\}.$$
To see this, suppose $(x,y)$ is in the first set. It is immediate that $x\in B(n-C-|g|)$. By the quasi-triangle inequality, $|y| \leq d(y,g)+d(g,e)+C_q\leq n-C-|g|$, so $y\in B(n-C-|g|)$ too. To finish, it suffices to show that
$$B(r) \cap B(x,n-|g|) \cap B(y,n-|g|) \subset B(r) \cap B(x,n) \cap B(y,n+C_q)g^{-1}.$$
This holds because $B(x,n-|g|) \subset B(x,n)$ and $B(y,n-|g|) \subset  B(y,n+C_q)g^{-1}$. To justify the latter inclusion, let $z \in B(y,n-|g|)$. This means $d(y,z)\le n-|g|$. By the quasi-triangle inequality and left-invariance,
$$d(y,zg) \le d(y,z) + d(z,zg)+C_q = d(y,z) + |g|+C_q \le n+C_q.$$
Thus $zg \in B(y,n+C_q)$ which implies $z \in B(y,n+C_q)g^{-1}$ as required.

It follows that 
\begin{align*}
&\lim_{r\to\infty} \liminf_{n\to\infty}  \frac{\#\{(x,y) \in B(e,N)\times B(g,M):~|B(r) \cap B(x,n) \cap B(y,n+C_q)g^{-1}| < m\}}{|B(n)|^2} \\
&\leq \lim_{r\to\infty} \liminf_{n\to\infty} \frac{\#\{(x,y) \in B(n-C-|g|)^2:~|B(r) \cap B(x,n-|g|) \cap B(y,n-|g|)| < m\}}{|B(n-C-|g|)|^2}.
\end{align*}
The latter equals zero by the overlapping neighborhoods property. This completes the proof.
\end{proof}

    


\begin{proof}[Proof of Theorem \ref{T:metric-group2}]


 Because $\mu=\int \nu ~d\zeta(\nu)$ is a convex integral of measures in $\partial \Meas(\mu)$, the conclusions of Lemma \ref{L:half-way} hold for $\mu$ in place of $\nu$. This means: for all $g\in \G$, for $\mu^2$-a.e.\ $h=(h_1,h_2)\in Z'_g$, $\Ret(Z_g, h)$ is infinite. By
Corollary \ref{C:return-time}, $Z'_g \subset \mathsf{Con}(\cH^2)$  up to a set of $\mu^2$-measure zero.

Let $\mathsf{Expand}^{\times 2}:\cH^2 \to \cH^2$ be the map
$$\mathsf{Expand}^{\times 2}(h_1,h_2) = (\mathsf{Expand}(h_1),\mathsf{Expand}(h_2)).$$
Also, for $n\in \N$, let $(\mathsf{Expand}^{\times 2})^n$ be the composition of $\mathsf{Expand}^{\times 2}$ with itself $n$ times.

Because $\mu$ is $\Expand$-quasi-invariant, by Remark \ref{R:trivial}, the map $\mathsf{Expand}$ is a finite measure extension of $\G \cc (\cH,\mu)$ of itself.  It follows from Lemma \ref{L:finitemeasureextension} that the preimage  of $\mathsf{Con}(\cH^2)$ under $\mathsf{Expand}^{\times 2}$ is equal to $\mathsf{Con}(\cH^2)$ (modulo a set of measure zero). 


Observe that $(\mathsf{Expand}^{\times 2})^{2|g|+C+2C_q}(\cH_{\leq 0} \times g\cH_{\leq C_q}) \subset Z'_g$. Since $Z'_g \subset \mathsf{Con}(\cH^2)$, this implies
$$\cH_{\leq 0} \times g\cH_{\leq C_q}\subset \mathsf{Con}(\cH^2).$$
However, the set 
$$\bigcup_{g\in \G} \cH_{\leq 0} \times g\cH_{\leq 0}$$
is a complete section for the action of $\G$ on $\cH^2$ modulo $\mu^2$ by Lemma \ref{L:support}.  It follows from Theorem \ref{T:Kaim-Hopf} that $\G \cc (\cH^2,\mu\times \mu)$ is infinitely conservative. Therefore $\G \cc (\cH,\mu)$ is doubly recurrent. 
\end{proof}

\section{Product groups}\label{S:Product-groups}

Let $\G_1,\G_2$ be finitely generated non-amenable groups and $\G=\G_1\times \G_2$. We will find conditions on $\G_1,\G_2$ under which $\G$ has fixed price 1. The main result is: 


\begin{thm}\label{T:main3}
For $i=1,2$, let $d_i$ be a left-invariant proper integer-valued quasi-metric on a countable group $\G_i$ and let $\eps>0$. Assume each $(\G_i,d_i)$ is $\eps$-approximately sub-additive (as in Definition \ref{D:qm-properties}). Let $\G=\G_1\times \G_2$. Let $d$ be the $\ell^1$ quasi-metric on $\G$:
$$d(x,y) = d_1(x_1,y_1)+d_2(x_2,y_2)$$
for $x=(x_1,x_2)\in \G_1\times \G_2$ and $y=(y_1,y_2)\in \G_1\times \G_2$. Assume for $i=1,2$
\begin{align}\label{E:balanced}
    \lim_{n\to\infty}  \frac{\#B(\G_i,n)}{\#B(\G,n)} = 0
\end{align}
where $B(\G,n),B(\G_i,n) $ is the ball of radius $n$ in $\G$, $\G_i$ respectively (centered at the identity say). Then $\G$ has fixed price 1.
\end{thm}

\subsection{Double recurrence}

\begin{prop}\label{P:balanced}
Assume the hypotheses of Theorem \ref{T:main3}. Then $(\G,d)$ has the overlapping neighborhoods property of Definition \ref{D:onp}. 
\end{prop}

\begin{proof}
We will write $B(n)$, $B(\G_i,n)$ to mean the ball of radius $n$ centered at the identity in $\G$, $\G_i$ respectively. 

Let $(x,y) \in B(n)^2$. For example, $x=(x_1,x_2)\in B(n) \subset \G=\G_1\times \G_2$. We will show that if $|x_1|$ and $|y_2|$ are sufficiently large then the balls $B(x,n+C)$ and $B(y,n+C)$ have a large overlap when $C=2\eps+C_q$, where $C_q$ is the constant in the quasi-triangle inequality. This will use $\eps$-approximate sub-additivity.

Because $d_1$ is $\eps$-approximately sub-additive, for each integer $t$ with $0\le t\le |x_1|=d_1(x_1,e)$, there exists an element $\xi_1(t) \in SS(\G_1,t,\eps)$ such that $\xi_1(t)^{-1}x_1 \in SS(\G_1,|x_1|-t,\eps)$. This means
\begin{align}
    t-\eps& \le |\xi_1(t)|\le t+\eps \label{E:xi1} \\
    |x_1|-t-\eps& \le |\xi_1(t)^{-1}x_1|=d_1(\xi_1(t),x_1)  \le |x_1|-t+\eps.\label{E:xi2} 
\end{align}
We might think of $\xi_1$ as providing something like a path from the identity to $x_1$ even though there is no requirement that $\xi_1(t)$ is close to $\xi_1(t+1)$. 

Similarly, for each integer $t$ with $0\le t \le |y_2|$ there is a group element $\zeta_2(t)\in \G_2$ with
\begin{align}
    t-\eps& \le |\zeta_2(t)|\le t+\eps \label{E:zeta1}\\
    |y_2|-t-\eps& \le |\zeta_2(t)^{-1}y_2|=d_2(\zeta_1(t),y_2)  \le |y_2|-t+\eps.\label{E:zeta2}
\end{align}


Let $T=\min(|x_1|,|y_2|)$. Define a map $\rho:\{0,\ldots, T\} \to \G$ by
$$\rho(t) = (\xi_1(t),\zeta_2(t)).$$
We will show that if $t$ is sufficiently small then $\rho(t)$ lies in the overlap of $B(x,n+C)$ with $B(y,n+C)$. To prove this, we bound $d(x,\rho(t))$ as follows:
\begin{align*}
  d(x,\rho(t))  &= d_1(x_1, \xi_1(t)) + d_2(x_2, \zeta_2(t))  \\
  &\le |x_1|-t+\eps + d_2(x_2, e) + d_2(e, \zeta_2(t)) + C_q \\
  &\le |x_1|-t+\eps + |x_2| + t + \eps + C_q = |x|  + 2\eps + C_q.
\end{align*}
The first line is by definition of $d$, the second comes from \eqref{E:xi2} and the quasi-triangle inequality, the last comes from \eqref{E:zeta1}.

  Similarly, $d(y,\rho(t)) \le  |y| + 2\eps+C_q$ and $d(e,\rho(t)) \le 2t+2\eps$.
  
Fix $r>0$. It follows that
$$\rho(t) \in B(r)\cap B(x,n+2\eps+C_q)\cap B(y,n+2\eps+C_q) $$
for all $t \in \{0,\ldots, \min(T,r/2-\eps)\}$. Since $T=\min(|x_1|,|y_2|)$,  this implies
\begin{align*}
    |B(r)\cap B(x,n+2\eps+C_q)\cap B(y,n+2\eps+C_q) | &\ge \min(r/2, |x_1|, |y_2|) - \eps.
\end{align*}
Fix $m\in \N$. Suppose $C=2\eps+C_q$ and $r>2m+2\eps$. If $|B(r) \cap B(x,n+C) \cap B(y,n+C)| < m$ then the previous inequality implies 
$$m>\min(r/2, |x_1|, |y_2|) - \eps\ge \min(m, |x_1|-\eps, |y_2|-\eps).$$
So we must have either $|x_1|<m+\eps$ or $|y_2|<m+\eps$. Thus
   \begin{align*}
& \frac{\#\{(x,y) \in B(n)^2:~|B(r) \cap B(x,n+C) \cap B(y,n+C)| < m\}}{|B(n)|^2}\\
&\le  \frac{|B(\G_1,m+\eps)|\cdot |B(\G_2,n)|\cdot |B(n)| + |B(n)| \cdot |B(\G_1,n)| \cdot|B(\G_2,m+\eps)| }{|B(n)|^2}.
\end{align*}

Because $B(\G,n)$ contains the products $B(\G_1,n)\times B(\G_2,0)$ and $B(\G_1,0)\times B(\G_2,n)$, it follows that $|B(n)|\ge |B(\G_i,n)|$ for $i=1,2$. Equation \ref{E:balanced} now implies
\begin{align*}
    \lim_{r\to\infty} \liminf_{n\to\infty}  \frac{\#\{(x,y) \in B(n)^2:~|B(r) \cap B(x,n+C) \cap B(y,n+C)| < m\}}{|B(n)|^2} =0.
\end{align*}
This is the overlapping neighborhoods property.
\end{proof}

\subsection{Cost}

For this subsection, we assume the hypotheses of Theorem \ref{T:main3}. In order to prove that $\G$ has fixed price 1, we will invoke Theorem \ref{T:general}. For that purpose, we need to construct an imp action of $\G$ which is limit-amenable, doubly-recurrent and has normalized cost 1. By Proposition \ref{P:regular}, there exists a $\G$-invariant measure $\mu$ on $\cH$ satisfying certain conditions including that the action $\G \cc (\cH,\mu)$ is limit-amenable. By Proposition \ref{P:balanced} and Theorem \ref{T:metric-group2}, the action $\G \cc (\cH,\mu)$ is also doubly-recurrent. So this looks like a promising candidate.

Unfortunately, we do not know how to prove that this action has normalized cost 1 (unless $\G$ is exact, in which case the action is amenable by Theorem \ref{T:exact} and therefore has normalized cost 1). We will instead show a certain finite-measure-preserving extension of it has normalized cost 1 and is still limit-amenable. This is sufficient because double recurrence lifts to finite-measure-preserving extensions by Lemma \ref{L:finitemeasureextension}. 

The extension will be an action of the form $\G \cc (\Cocycle (\G)\times \cH(\G), \tmu)$ where $\Cocycle(\G)$ is a space of cocycles as defined next. A map $c:\G\times \G \to \Z^2$ is a {\bf cocycle} if
$$c(g,h)+c(h,k) = c(g,k)$$
for all $g,h,k\in \G$. It is {\bf 1-Lipschitz} if 
$$\|c(g,h)\| \le d(g,h)$$
for all $g,h \in \G$ where the norm on $\Z^2$ is defined by
$$\|(n,m)\|=|n|+|m|.$$
Let $\Cocycle= \Cocycle(\G)$ be the space of all $1$-Lipschitz cocycles $c:\G\times \G \to \Z^2$ with the topology of pointwise convergence on finite subsets. We will write $\Cocycle(\G)$ and $\Cocycle$ interchangeably when $\G$ is clear from context.

\begin{lem}
    $\Cocycle(\G)$ is compact.
\end{lem}
\begin{proof}
This follows from the Ascoli-Arzela Theorem. Alternatively, we see that $\Cocycle(\G)$ is naturally identified with a closed subset of a product of intervals of the form $[-d(g,h),d(g,h)]$ over all $(g,h)\in \G^2$. By Tychonoff's Theorem, the latter space is compact.
\end{proof}
Let $\G$ act on $\Cocycle(\G)$ by
$$(gc)(x,y) = c(g^{-1}x, g^{-1}y).$$
We let $\G$ act on $\Cocycle(\G)\times \cH(\G)$ by $g\cdot (c,h) = (g c, g h)$.

Actually, we will only need a subset of $\Cocycle(\G)\times \cH(\G)$. To define this subset, let $\ssum: \Z^2 \to \Z$ be given by $\ssum(n,m)=n+m$.

\begin{defn}\label{D:extendedhorofunctions}
    Let $\tcH_n$ be the set of all pairs $(c,h) \in \Cocycle(\G)\times \cH(\G)$ satisfying 
$$\ssum(c(g,e))+n=h(g)$$
for all $g\in \G$. Also let $\tcH=\cup_{n \in \Z} \tcH_n$. 
\end{defn}

Observe that $\tcH$ is $\G$-invariant and closed in $\Cocycle(\G)\times \cH(\G)$. In fact, if $f\in\G$ then $f\tcH_n =\tcH_{n+c(e,f^{-1})}$. In Lemma \ref{L:existence-1} below we construct a $\G$-invariant measure on $\tcH$ with nice properties. Before doing so, we present a general measure theory result which will be needed.

\begin{prop}\label{P:general10}
Let $X$ be a locally compact Polish space and $K$ be a compact metric space. Let $\Radon(X)$ and $\Radon(K\times X)$ be the space of Radon measures on $X$ and on $K\times X$ in the vague topology respectively. If $\pi:K\times X \to X$ is the projection map then the push-forward $\pi_*:\Radon( K\times X)\to \Radon(X)$ is proper. That is, if $\cM \subset \Radon(\G,X)$ is compact then $\pi_*^{-1}(\cM)$ is also compact.
\end{prop}

\begin{proof}
Let $\cM \subset \Radon(X)$ be compact. Let $(\tilde{\mu}_n)_{n=1}^\infty \subset \pi_*^{-1}(\cM)$. It suffices to prove this sequence has a subsequential limit. 

Let $\mu_n=\pi_*(\tilde{\mu}_n)$ be the push-forward measure. Since $\cM$ is compact, after passing to a subsequence if necessary, we may assume $\mu_n$ converges to a measure $\mu \in \cM$  as $n\to\infty$ in the vague topology. 

Because the space $C_c(K\times X)$ of compactly supported continuous functions is separable in the uniform topology, it suffices, by a diagonalization argument, to prove that for every $f\in C_c(K \times X)$ there exists a subsequence $(n_i)_{i=1}^\infty$ such that the limit of $\tilde{\mu}_{n_i}(f)$ exists. 

Because $f$ has compact support there is a constant $M$ such that $|f(k,x)|\le M$ for all $(k,x)\in K\times X$. Additionally, there exists a compact subspace $Y \subset X$ such that $f(k,x)=0$ for all $(k,x) \notin K \times Y$. 

Because $\mu$ is Radon, $\mu(Y)<\infty$. By the unbounded Portmanteau Theorem \ref{T:portmanteau}, we have $\limsup_{n\to\infty} \mu_n(Y) \le \mu(Y)<\infty$. Since  $\tilde{\mu}_n(K\times Y) = \mu_n(Y)$, this implies that
$$\sup_{n}\tilde{\mu}_n(K\times Y)  =L$$
is finite. Since the support of $f$ is contained in $K\times Y$, it follows that $\int f~d\tilde{\mu}_n \in [-ML,ML]$ for all $n$.  Since this interval is compact, there exists a subsequence $(n_i)_{i=1}^\infty$ such that the limit of $\tilde{\mu}_{n_i}(f)$ exists as required.
\end{proof}

Let $\pi:\Cocycle(\G)\times \cH(\G) \to \cH(\G)$, be the projection $\pi(c,h)=h$. Recall that $\Expand:\cH \to \cH$ is the map $\Expand(h)=h-1$. Define $\widetilde{\Expand}:\tcH \to \tcH$ by $\widetilde{\Expand}(c,h)=(c,h-1)$. So $\pi \circ \widetilde{\Expand}=\Expand \circ \pi$. Moreover, the maps $\pi, \widetilde{\Expand}$ and $\Expand$ are all $\G$-equivariant.

Recall from \S \ref{S:metric-measures} that $\mu_{0} = \sum_{x\in \G} \d_{d_x}$ is a measure on $\cH(\G)$ and $\mu_{n} = \frac{\Add^n_{*}\mu_{0}}{|B(n)|}$. By Lemma \ref{L:support}, $\partial\Meas(\cH)$ is the set of all measures $\mu$ on $\cH$ such that there exists a sequence $(n_i)_{i=1}^\infty$ tending to infinity with $\mu = \lim_{i\to\infty} \mu_{n_i}$ in the almost-weak topology.

\begin{lem}\label{L:existence-1}
For every measure $\mu \in \partial\Meas(\cH)$ there exists a $\G$-invariant measure $\tmu$ on $\tcH$ which projects to $\mu$ under the map $\pi:\Cocycle \times \cH \to \cH$, where $\pi(c,h)=h$. Moreover, $\tmu$ can be chosen so that the action $\G \cc (\tcH,\tmu)$ is limit-amenable.
\end{lem}

\begin{proof}
    Recall that for $x\in \G$, $d_x:\G \to \Z$ is the function $d_x(y)=d(x,y)$. For a given $x=(x_1,x_2) \in \G$, let $c_x:\G\to \R^2$ be the cocycle
$$c_x( (y_1,y_2), (z_1,z_2)) = (d_1(x_1,y_1)-d_1(x_1,z_1), d_2(x_2,y_2)-d_2(x_2,z_2)).$$
Define a measure $\tilde{\mu}_0$ on $\tcH(\G)$ by 
$$\tilde{\mu}_0 = \sum_{x\in \G} \d_{c_x,d_x}.$$
It is immediate $\pi_*\tmu_0=\mu_0$. Moreover, the action $\G \cc (\tcH,\tmu_0)$ is regular by construction. 

For $n\in \N$, define a measure $\tilde{\mu}_n$ on $\tcH(\G)$ by 
$$\tilde{\mu}_n = \frac{\widetilde{\Expand}^n_*\tilde{\mu}_0}{|B(n)|}.$$ 
Because $\pi$ and $\widetilde{\Expand}$ are $\G$-equivariant, it follows that $\tmu_n$ is $\G$-invariant, $\pi_*\tmu_n=\mu_n$ and the action $\G \cc (\tcH,\tmu_n)$ is regular (in the sense that it is measurably conjugate to the left-regular action of $\G$ on itself).

Now suppose $\mu \in \partial\Meas(\cH)$. By Lemma \ref{L:support}, $\mu$ is the almost-weak limit of measures $\mu_{n_i}$ for some divergent sequence $(n_i)_{i=1}^\infty$. 

By Lemma \ref{L:compact}, $\overline{\Meas(\cH)}$ is compact. By Proposition \ref{P:general10}, $\pi^{-1}_*(\overline{\Meas(\cH)})$ is compact in the vague topology. So, after passing to a sub-sequence if necessary, we may assume that $\tmu_{n_i}$ converges to a measure $\tmu$ in the vague topology. Because the $\G$-action is continuous and each $\tmu_{n_i}$ is $\G$-invariant, $\tmu$ is $\G$-invariant. Because $\pi$ is continuous, it follows that $\pi_*\tmu=\mu$. Because each action $\G \cc (\tcH, \tmu_{n_i})$ is regular, the action $\G \cc (\tcH, \tmu)$ is limit-amenable. 
\end{proof}

\begin{prop}\label{P:product-cost}
Let $\tmu$ be a $\G$-invariant measure on $\tcH$ satisfying:
\begin{enumerate}
    \item $\tmu(\tcH_{\le 0})=1$, 
    \item for $\tmu$-a.e.\ $(c,h)$ there exists $g \in \G$ with $h(g)\le 0$.
\end{enumerate}
Then the normalized cost of the action $\G \cc (\tcH(\G),\tmu)$ is 1.
\end{prop}

\begin{proof}
Because of the Ergodic Decomposition Theorem \cite{MR1784210}, without loss of generality, we may assume $\tmu$ is ergodic with respect to the $\G$-action. Since $\tcH=\cup_k \tcH_k$, there exists an integer $k$ such that $\tmu(\tcH_k)>0$. After replacing $\tmu$ with $\frac{\widetilde{\Expand}^{-k}_*\tmu}{\tmu(\tcH_{\le -k})}$ if necessary we may assume $\tmu(\tcH_0)>0$. Since the action is ergodic, this means that $\tcH_0$ is a complete section for the action.

Let $\cR^\G$ be the orbit-equivalence relation 
$$\cR^\G=\{(\xi, g\xi):~\xi \in \tcH\}.$$
Let $\cR_0$ be the restriction of $\cR^\G$ to $\tcH_0$:
$$\cR_0=\cR^\G\cap (\tcH_0\times \tcH_0).$$
Because $\tcH_0$ is a complete section for the $\G$-action, it suffices to compute the cost of $\cR_0$ with respect to $\tmu_0$ (where $\tmu_0$ is the restriction of $\tmu$ to $\tcH_0$).

We will show that there is a normal sub-equivalence relation $\cK \subset \cR_0$ which splits as a direct product. Then we can apply one of Gaboriau's theorems to prove $\cK$ has cost $\tmu(\tcH_0)$. This is the main step towards proving $\cR_0$ also has cost $\tmu(\tcH_0)$.

There is a canonical cocycle $\theta:\cR_0 \to \Z^2$ given by 
$$\theta( (c,h), (gc,gh) ) =c(e,g^{-1})=-c(g,e).$$
Let $\cK \le \cR_0$ be the kernel of this cocycle:
$$\cK=\{(\xi, \xi'):~\theta(\xi,\xi')=(0,0)\}.$$
The first step is to show that $\cK$ has cost $\tmu(\tcH_0)$ with respect to $\tmu_0$. 

To analyze $\theta$ and $\cK$, we need to introduce more notation. For $i=1,2$, let $\cH(\G_i)$ be the space of horofunctions for the metric group $(\G_i,d_i)$. As usual, let $\cH_0(\G_i)$ be the subspace of horofunctions $h \in \cH(\G_i)$ with $h(e)=0$. Define $\Phi:\cH_0(\G_1)\times \cH_0(\G_2) \to \tcH_0$ by $\Phi(h_1,h_2)=(c,h)$ where 
\begin{align*}
    h(g_1,g_2) &= h_1(g_1) + h_2(g_2) \\
    c( (f_1,f_2), (g_1,g_2)) &= (h_1(g_1)-h_1(f_1), h_2(g_2) - h_2(f_2)).
\end{align*}
Also define $\Psi:\tcH_0 \to \cH_0(\G_1)\times \cH_0(\G_2)$ by 
$$\Psi(c,h)=(h_1,h_2)$$
where $h_1,h_2$ are defined by $c(e,g)=(h_1(g_1),h_2(g_2))$. Using the definition of $\tcH$ (Definition \ref{D:extendedhorofunctions}), it can be seen that $\Phi$ and $\Psi$ are inverses of each other. In particular, they are both homeomorphisms. Moreover, they are $\G$-equivariant in the following sense: if $g=(g_1,g_2) \in \G$, $(h_1,h_2)\in \cH_0(\G_1)\times \cH_0(\G_2)$ and $g_ih_i \in \cH_0(\G_i)$ for $i=1,2$, then $\Phi((g_1h_1,g_2h_2))=g\Phi(h_1,h_2) \in \tcH_0$. A similar statement holds for $\Psi$.

For $i=1,2$, let $\cS^{\G_i}$ be the orbit-equivalence relation of $\G_i$ on $\cH(\G_i)$:
$$\cS^{\G_i}=\{(\xi, g\xi):~\xi \in \cH(\G_i), g\in \G_i\}.$$
Let $\cS_i$ be the restriction of $\cS^{\G_i}$ to $\cH_0(\G_i)$:
$$\cS_i=\cS^{\G_i}\cap (\cH_0(\G_i)\times \cH_0(\G_i)).$$
Then $\Phi$ induces a bijection from $\cS_1\times \cS_2$ to $\cK$. This is because of the $\G$-equivariance mentioned earlier. It now follows from Gaboriau's Theorem \cite[Theorem 24.9]{MR2095154}, that $\cK$ has cost $\tmu_0(\tcH_0)$ with respect to $\tmu_0$. 

Next we will show that $\cR_0$ has cost $\tmu_0(\tcH_0)$ by constructing a graphing which witnesses this cost (up to an error, which will be controlled). First, we need to better understand how the $\cR_0$-classes are partitioned into $\cK$-classes.

The range of the cocycle $\theta$ can be simplified in the following sense. Suppose $\Phi(h_1,h_2)=(c,h)$. Then
$$\theta((c,h),g\cdot (c,h))=c(e,g^{-1})=(h_1(g_1^{-1}), h_2(g_2^{-1})).$$
However, since we are implicitly assuming $g\cdot (c,h) \in \tcH_0$, it follows that 
$$h(g^{-1})=h_1(g_1^{-1})+h_2(g_2^{-1})=0.$$
Therefore, the image of $\theta$ is contained in the anti-diagonal subgroup $\Delta=\{(-n,n):~n \in \Z\}$.

Let $\delta>0$ and let $\{\cB_n\}_{n\in \Z}$ be a sequence of pairwise disjoint Borel subsets of $\tcH_0$ such that 
\begin{enumerate}
    \item for a.e.\ $\xi \in \tcH_0$ and every $n\in \Z$ there exists $\xi'\in \cB_n$ with $(\xi,\xi')\in \cK$ and 
    \item if $\cB=\cup_{n\in \Z} \cB_n$ then $\tmu(\cB)<\delta$.
\end{enumerate}
The first property is equivalent to saying that each $\cB_n$ is a complete section for $\cK$. 

For each $n$, let $\cB'_n$ be the set of all $\xi\in \cB_n$ such that there exists $\xi'$ with $\theta(\xi,\xi')=(-n,n)$ (in particular $(\xi,\xi')\in \cR_0$). This is a Borel set because $\theta$ is continuous. Let $\phi_n:\cB'_n \to \tcH_0$ be a Borel map such that $\theta(\xi,\phi_n(\xi))=(-n,n)$ (in particular, the graph of $\phi_n$ is contained in $\cR_0$).

Let $\cG_\cK$ be a graphing of $\cK$ with cost $<\tmu(\tcH_0)+\delta$. Let $\cG$ be the union of $\cG_\cK$ with $\{(\xi, \phi_n(\xi)):~\xi \in \cB'_n, n\in\Z\}$. 

We claim that $\cG$ is a graphing of $\cR_0$. To see this, let $(\xi,\xi')\in \cR_0$ and suppose $\theta(\xi,\xi')=(-n,n)$. Because $\cB$ is a complete section (mod $\mu$) for a.e.\ such $\xi$, there exists $\zeta \in \cB_n$ with $(\xi,\zeta)\in \cK$. Because $\theta(\xi,\xi')=(-n,n)$ and $(\xi,\zeta)\in \cK$, it follows that $\theta(\zeta,\xi')=(-n,n)$. So $\zeta \in \cB'_n$. So $(\zeta,\phi_n(\zeta))\in \cG$. Also note $(\phi_n(\zeta), \xi')\in \cK$ by the cocycle condition.

We now see that there is a path in $\cG$ from $\xi$ to $\xi'$: namely the path obtained by concatenating a path in $\cG_{\cK}$ from $\xi$ to $\zeta$ with the edge $(\zeta,\phi_n(\zeta))$ together with a path from $\phi_n(\zeta)$ to $\xi'$ in $\cG_\cK$. Because $(\xi,\xi')\in \cR_0$ is arbitrary, this shows $\cG$ is a graphing.

Since the cost of $\cG$ is at most the cost of $\cG_\cK$ plus $\tmu(\cB)$, $\cost_{\tmu_0}(\cG) \le \tmu(\tcH_0)+2\delta$. Since $\delta$ is arbitrary, it follows that the cost of $\cR_0$ is $\tmu(\tcH_0)$ with respect to $\tmu_0$. By definition, this implies that the normalized cost of $\G \cc (\tcH,\tmu)$ is 1.
\end{proof}

\subsection{Proof of Theorem \ref{T:main3}}

\begin{proof}[Proof of Theorem \ref{T:main3}]

By Proposition \ref{P:balanced}, $(\G,d)$ has the overlapping neighborhoods property. Let $\mu$ be a $\G$-invariant measure satisfying the conclusions of Proposition \ref{P:regular}. By Theorem \ref{T:metric-group2},  $\G \cc (\cH,\mu)$ is doubly-recurrent. By Lemma \ref{L:existence-1}, there exists a $\G$-invariant measure $\tmu$ on $\tcH$ which projects to $\mu$. Moreover, the action $\G \cc (\tcH,\tmu)$ is limit-amenable. By Lemma \ref{L:finitemeasureextension}, the action $\G \cc (\tcH,\tmu)$ is doubly recurrent. By Proposition \ref{P:product-cost}, the normalized cost of $\G \cc (\tcH,\tmu)$ is 1. By Theorem \ref{T:general}, $\G$ has fixed price 1.
\end{proof}

\subsection{Comparable growth}

For $i=1,2$, we assume $\G_i$ is a countable group and $d_i$ is a left-invariant, proper, integer-valued quasi-metric on $d_i$ which is $\eps$-approximately subadditive. The latter means that
$$SS(\G_i,n,\eps)SS(\G_i,m,\eps) \supset S(\G_i,n+m)$$
where $S(\G_i,n), SS(\G_i,n,\eps)$ are the radius $n$ sphere and the union of the spheres with radius $r \in [n-\eps,n+\eps]$ respectively. Because $d_i$ is integer-valued we will assume $\eps$ is also integer-valued.

\begin{defn}
We  say $(\G_1,d_1)$ and $(\G_2,d_2)$ have {\bf roughly comparable growth rates} if there are functions $f_i:\N \to [0,\infty)$ for $i=1,2$ such that for all $n\in \N$ and $i=1,2$
\begin{align*}
    f_i(n) \#SS(\G_i,n,\eps)&\le \#SS(\G_{2-i},n,\eps) \\
    \sum_{n=1}^\infty f_i(2n\eps+m) &=\infty
\end{align*}
for every $m\in \{0,1,\ldots\}$. 

\end{defn}



\begin{prop}\label{P:roughly-comparable}
Suppose that $(\G_1,d_1)$ and $(\G_2,d_2)$ have roughly comparable growth rates.
Then equation \eqref{E:balanced} is satisfied. 
\end{prop}

\begin{proof}
Observe that
\begin{align*}
    \frac{\#B(\G_i,n)}{\#B(\G,n)} = \sum_{r=0}^n \frac{\#S(\G_i,r)}{\#SS(\G,r,2\eps)}\cdot  \frac{\#SS(\G,r,2\eps)}{\#B(\G,n)}.
\end{align*}
Since 
\begin{align*}
  \sum_{r=0}^n \frac{\#SS(\G,r,2\eps)}{\#B(\G,n)}\le 4\eps,
\end{align*}
and of course $\frac{\#SS(\G,r,2\eps)}{\#B(\G,n)} \to 0$ as $n\to\infty$ with $r$ held fixed, it suffices to prove:
\begin{align}\label{E:rough}
   \lim_{r\to\infty} \frac{\#S(\G_i,r)}{\#SS(\G,r,2\eps)} = 0.
\end{align}
For simplicity, let us assume $r$ is divisible by $2\eps$. We estimate $\#\SS(\G,r,2\eps)$ as follows. Observe that $\SS(\G,r,2\eps)$ contains the direct product $SS(\G_1,r-2k\eps,\eps)\times SS(\G_2,2k\eps,\eps)$ for all $0\le k \le r/2\eps$. Moreover these products are pairwise disjoint. Therefore 
\begin{align*}
   \# \SS(\G,r,2\eps) &\ge \sum_{k=0}^{r/2\eps} \#SS(\G_1,r-2k\eps,\eps)\#SS(\G_2,2k\eps,\eps)
\end{align*}
Because $(\G_1,d_1)$ and $(\G_2,d_2)$ are roughly comparable, $\#SS(\G_2,2k\eps,\eps) \ge f_1(2\k\eps)\#SS(\G_1,2k\eps,\eps)$. Therefore,
\begin{align*}
   \# \SS(\G,r,2\eps) &\ge \sum_{k=0}^{r/2\eps} \#SS(\G_1,r-2k\eps,\eps)\#SS(\G_1,2k\eps,\eps)f_1(2k\eps).
\end{align*}
By $\eps$-subadditivity, the product $SS(\G_1,r-2k\eps,\eps)\cdot SS(\G_1,2k\eps,\eps)$ contains the sphere $S(\G_1,r)$. Therefore, 
\begin{align*}
   \# \SS(\G,r,2\eps) &\ge \sum_{k=0}^{r/2\eps} S(\G_1,r)f_1(2k\eps).
\end{align*}
Thus
$$\frac{\#S(\G_1,r)}{\#SS(\G,r,2\eps)} \le \frac{1}{\sum_{k=0}^{r/2\eps} f_1(2k\eps).} $$
Because $\sum_{k=0}^\infty f_1(2k\eps)=\infty$, this implies the limit \eqref{E:rough} with one modification: we assumed $i=1$ and $r$ is divisible by $2\eps$. The general case, when $r$ is congruent to $m$ mod $2\eps$ for some fixed number $m$ and $i\in \{1,2\}$ is similar and left to the reader.
\end{proof}

\begin{cor}
    For any countable group $\G$, $\G\times \G$ has fixed price 1.
\end{cor}

\begin{proof}
If $\G$ is amenable then $\G \times \G$ is amenable and it follows from the Ornstein-Weiss Theorem \cite{OW80} that all amenable groups have fixed price 1. This is also in \cite[Corollary 31.2]{MR2095154}. So assume $\G$ is non-amenable.    If $\G$ is finitely generated, then we choose $d_1=d_2$ to be a word metric on $\G$. By Proposition \ref{P:roughly-comparable} equation \eqref{E:balanced} is satisfied.  By Theorem \ref{T:main3}, $\G\times \G$ has fixed price 1.

If $\G$ is non-amenable but not finitely generated, then there exist non-amenable finitely generated subgroups $\G_1 \le \G_2 \le \cdots $ such that $\G = \cup_{i=1}^\infty \G_i$. By the previous paragraph, each $\G_i\times \G_i$ has fixed price 1. Because $\G \times \G = \cup_{i=1}^\infty \G_i \times \G_i$, it follows that $\G$ also has fixed price 1. This fact follows from a theorem of Gaboriau that appears in \cite[Proposition 32.1 (ii)]{MR2095154}.
\end{proof}

\begin{cor}\label{C:poly-exp}
Suppose for $i=1,2$, $(\G_i,d_i)$ are countable groups equipped with left-invariant proper metrics satisfying the $\eps$-approximate sub-additivity condition and  the growth condition
$$C^{-1} n^{\d_i} e^{\a_i n} \le |B(\G_i,d_i,n)| \le C n^{\d_i} e^{\a_i n}$$
for some constants $\d_i, C, \a_i>0$. Such groups are said to have {\bf exact polynomial-exponential growth}. Suppose as well that $|\d_1-\d_2|\le 1$. Then $\G_1\times \G_2$ has fixed price 1.
\end{cor}

\begin{proof}
Define a quasi-metric $d'_i$ on $\G_i$ by $d'_i(x,y)=\lceil\alpha_i\cdot d_i(x,y)\rceil$ where $\lceil x\rceil$ is the smallest integer greater than or equal to $x$. By Proposition \ref{P:roughly-comparable}, the rescaled metric groups $(\G_i,d'_i)$ satisfy  equation \eqref{E:balanced}.  By Theorem \ref{T:main3}, $\G_1\times \G_2$ has fixed price 1.
\end{proof}

\begin{remark}
    According to \cite{fujiwara2025hardylittlewoodmaximaloperatorspaces} and references therein, the following groups have exact polynomial-exponential growth 
        \begin{enumerate}
        \item $(\G,d)$ where $\G$ is a lattice ins a connected semi-simple Lie group $G$ with finite center and $d$ is induced from a $G$-invariant Riemannian metric arising from the Killing form;
        \item hyperbolic groups with respect to word metrics \cite{MR1214072};
        \item right-angled Artin groups with word metrics induced by standard generating sets;
        \item geometrically finite discrete subgroups of $\operatorname{Isom}(\H^n)$ with respect to word metrics induced by suitable finite generating sets;
        \item Coxeter groups of exponential growth with respect to standard generating sets;
        \item braid groups of exponential growth with respect to standard generating sets;
        \item Artin groups of extra-large type with respect to standard generating sets.
    \end{enumerate}
\end{remark}




\appendix

\section{Recurrence}\label{S:recurrence}

In this subsection, we recall the Hopf decomposition of \cite{MR2731695}. Throughout, we let $G \cc (X,\mu)$ be an action by measure-class preserving transformations.

By the Ergodic Decomposition Theorem \cite{MR1784210}, there exist a standard measure space, denoted $(Z,\zeta)$ and measurable maps $\pi:X \to Z$, $\nu:Z \to \Prob(X)$ such that 
\begin{enumerate}
\item $\pi$ is $\G$-invariant mod $\mu$;
\item $\nu_z(\pi^{-1}(z))=1$ for a.e.\ $z\in Z$;
    \item $\mu = \int \nu_z~d\zeta(z)$;
    \item if $\phi:X \to Y$ is any measurable map to a standard Borel space $Y$ which is $\G$-invariant mod $\mu$, then there exists a measurable map $\tphi:Z \to Y$ such that $\phi(x) = \tphi(\pi(x))$ for a.e.\ $x$.
\end{enumerate}
We say that tuple $(Z,\zeta, \pi, \nu)$ comprises the \textbf{ergodic decomposition of the action $\G \cc (X,\mu)$}.

\begin{defn}\label{D:Hopf}
   The \textbf{continual part} of the action $\G \cc (X,\mu)$ is 
   $$\mathsf{Cont}(X)=\{x\in X:~ \nu_{\pi(x)} \textrm{ is non-atomic}\}.$$
The \textbf{discontinual part} of the action, denoted $\mathsf{Discont}(X) \subset X$ is the complement and consists of all atomic orbits ergodic components. Also let 
$$\mathsf{Discont}_{\textrm{cofinite}}(X)=\{ x\in \mathsf{Discont}(X):~ \Stab_\G(x) \textrm{ is finite }\}$$
   where   $$\Stab_\G(x)=\{g\in \G:~ gx=x\}$$
is the stabilizer of $x$. Thus we can write $X$ as a disjoint union
$$X = \mathsf{Cont}(X) \sqcup (\mathsf{Discont}(X)\setminus  \mathsf{Discont}_{\textrm{cofinite}}(X))\sqcup \mathsf{Discont}_{\textrm{cofinite}}(X).$$
Each of these parts is $\G$-invariant.
\end{defn}

\begin{defn}
Given $Y \subset X$ and $x\in X$, let 
 $$\Ret(Y,x)=\{g\in \G:~ gx \in Y\}$$
 be the return-time set. The set $Y$ is said to be \textbf{recurrent} if for a.e.\ $y\in Y$ there exists a non-identity element $g\in \G$ with $gy \in Y$ (i.e. the return-time set $\Ret(Y,y)\ne \{e\}$). The set $Y$ is said to be  \textbf{infinitely recurrent} if for a.e.\ $y \in Y$, the return-time set $\Ret(Y,y)$ is infinite. The action $\G \cc (X,\mu)$ 
 is \textbf{infinitely conservative} if every measurable subset of $X$ with positive measure is infinitely recurrent.  
 We will say the action is {\bf doubly-recurrent (DR)} if the diagonal action $\G \cc (X\times X,\mu\times\mu)$ is infinitely conservative.
\end{defn}

\begin{defn}
    Let $\mathsf{Con}(X)=X  \setminus \mathsf{Discont}_{\textrm{cofinite}}(X)$ and $\mathsf{Dis}(X) = \mathsf{Discont}_{\textrm{cofinite}}(X)$.
\end{defn}

\begin{thm}\label{T:Kaim-Hopf}
The restriction of the action to the set $  \mathsf{Con}(X)$ is infinitely conservative (with respect to the measure $\mu$). On the other hand, if $\G \cc (X,\mu)$ is an imp and $E \subset \mathsf{Dis}(X)$ has finite positive measure, then for a.e.\ $x\in X$, $\Ret(E,x)$ is finite. 
\end{thm}

\begin{proof}
    The first statement is proven in \cite[Propositions 7 and 8]{MR2731695}. To prove the second statement, let $E \subset \mathsf{Discont}_{\textrm{cofinite}}(X)$ have finite positive measure. By \cite[Lemma 4]{MR2731695}, there exists a measurable map $$\phi:\mathsf{Discont}_{\textrm{cofinite}}(X) \to \mathsf{Discont}_{\textrm{cofinite}}(X)$$ such that for every $x$, $\phi(x)$ is in the $\G$-orbit of $x$. Moreover, $\phi(x)=\phi(gx)$ for every $x$ and every $g\in \G$. Let $W$ be the image of $\phi$. Then $\mathsf{Discont}_{\textrm{cofinite}}(X)$ is the disjoint union of $gW$. 

    Because $\mu$ is $\G$-invariant,
    $$\mu(E) = \int_W |\phi^{-1}(x)\cap E|~d\mu(x).$$
Since $\mu(E)$ is finite and $\phi^{-1}(x)=\G x$, this implies $\mu$-a.e.\ $x$ is such that $|\G x \cap E|<\infty$. But $|\G x\cap E| = |\Ret(E,x)|$. So $|\Ret(E,x)|<\infty$.
\end{proof}

\begin{cor}\label{C:return-time}
     Let $\G \cc (X,\mu)$ be an infinite-measure-preserving action. Let $Y \subset X$ have finite positive measure. For $y\in Y$, let 
 $$\Ret(Y,y)=\{g\in \G:~ gy \in Y\}$$
 be the return-time set and 
 $$Y_\infty = \{y \in Y:~ \Ret(Y,y) \textrm{ is unbounded} \}.$$
 Then $Y_\infty \subset \mathsf{Con}(X)$ (ignoring a set of measure zero).
\end{cor}

\begin{defn}
    Let $\G \cc (X,\mu)$ be an ergodic imp  action. By Theorem \ref{T:Kaim-Hopf}, $X^2$ is the disjoint union of $\G$-invariant measurable sets $\mathsf{Con}(X^2)$ and $\mathsf{Dis}(X^2)$ and the restriction of $\G$ to $\mathsf{Con}(X^2)$ is infinitely conservative. We will say the action $\G \cc (X,\mu)$ is {\bf partially doubly recurrent} (PDR) if for a.e.\ $x,y \in X$ there exist $x=x_1,x_2,\ldots, x_n=y$ with $(x_i,x_{i+1})\in \mathsf{Con}(X^2)$ for all $i$. In other words, the equivalence relation generated by $\mathsf{Cont}(X^2)$ is all of $X$ (up to a set of measure zero). 
\end{defn}

\begin{lem}\label{L:conserv-erg-dec}
Let $\G \cc (X,\mu)$ be measure-preserving. Then this action is infinitely conservative (doubly-recurrent, partially doubly recurrent) if and only if a.e.\ ergodic component is infinitely conservative (doubly-recurrent, partially doubly recurrent).
\end{lem}

\begin{proof}
This is a direct consequence of Theorem \ref{T:Kaim-Hopf}.
\end{proof}

\begin{lem}\label{L:finitemeasureextension}
Suppose $\G \cc (X_1,\mu_1)$ is a finite measure extension of $\G \cc (X_2,\mu_2)$ (and both are imp actions). 
\begin{enumerate}
\item $\G \cc (X_1,\mu_1)$ is infinitely conservative if and only if $\G \cc (X_2,\mu_2)$ is infinitely conservative.
\item $\G \cc (X_1,\mu_1)$ is doubly recurrent if and only if $\G \cc (X_2,\mu_2)$ is doubly recurrent.
\item $\G \cc (X_1,\mu_1)$ is partially doubly recurrent if and only if $\G \cc (X_2,\mu_2)$ is partially doubly recurrent.
\end{enumerate}
\end{lem}

\begin{proof}
(1): If $\G \cc (X_1,\mu_1)$ is conservative, then by lifting finite measure sets from $X_2$ up to $X_1$, we see that $\G \cc (X_2,\mu_2)$ is also conservative. 

Suppose $\G \cc (X_2,\mu_2)$ is conservative. We will show $\G\cc (X_1,\mu_1)$ is also conservative. Let $Z \subset X_1$ be a set with positive finite measure. We need to show $\Ret(Z,z)$ is unbounded for a.e.\ $z\in Z$. 

Because $\phi:X_1 \to X_2$ is a finite-measure-extension, there exists a partition $Z=\sqcup_i Z_i$ of $Z$ into sets and there exist subsets $Y_i \subset X_2$ such that $\mu_2(Y_i)<\infty$ and $Z_i \subset Y_i$ for all $i$. So without loss of generality, we may assume there exists a finite measure set $Y \subset X_2$ such that $Z \subset \phi^{-1}(Y)$.

Because $\G \cc (X_2,\mu_2)$ is conservative, the return time set $\Ret(Y,y)$ is unbounded for a.e.\ $y\in Y$. Therefore, the return time set $\Ret(\phi^{-1}(Y),y)$ is unbounded for a.e.\ $y \in \phi^{-1}(Y)$. By Corollary \ref{C:return-time}, the return time set $\Ret(Z,z)$ is unbounded for a.e.\ $z\in Z$. This finishes the proof of item (1).

(2): The factor $(X_1\times X_1, \mu_1\times \mu_1) \to (X_2\times X_2,\mu_2\times \mu_2)$ is a finite measure extension. So item (1) implies item (2). 

(3): As in item (1), it is straightforward to check that if $\G \cc (X,\mu)$ is PDR then the factor $\G \cc (Y,\nu)$ is also PDR. So assume that the actor $\G \cc (Y,\nu)$ is PDR. 

In general, if $T \subset Y \times Y$, then we let $T^n$ be the set of all $(x,y)\in Y\times Y$ such that there exist $x=x_1,x_2,\ldots, x_n=y$ such that $(x_i,x_{i+1})\in T$ for all $i$.

Let $S \subset Y\times Y$ be a set with positive measure. Let $S_\infty$ be the set of all $(x,y)\in S$ such that the return-time set $\Ret(S,(x,y))$ is infinite. Because the action is PDR, $\cup_n S_\infty^n$ has full measure in $S$.

Let $\tilde{S}_\infty = \phi^{-1}(S_\infty)$. As in item (1), $\tilde{S}_\infty \subset \mathsf{Con}(X^2)$. Moreover, by induction on $n$, note that $\phi^{-1}(S^n_\infty) = \tilde{S}^n_\infty$. Therefore, $\cup_n \tilde{S}^n_\infty$ has full measure in $\phi^{-1}(S)$. Because $S$ is arbitrary, this implies the action $\G \cc (X,\mu)$ is partially doubly recurrent. 

\end{proof}

\section{Measured equivalence relations}\label{S:mer}

Let $(X,\mu)$ be a standard Borel  space and $\cR \subset X \times X$ a Borel equivalence relation. For $x\in X$, let $ [x]_\cR=\{y\in X:~(x,y) \in \cR\}$ be its equivalence class.  We say $\cR$ is 
\begin{itemize}
    \item {\bf discrete} or {\bf countable} if every equivalence class is at most countable;
    \item {\bf aperiodic} if every equivalence class is infinite;
    \item {\bf finite} if every equivalence class is finite;
    \item {\bf hyperfinite} if there exist finite Borel equivalence relations $\cR_1 \subset \cR_2\subset \cdots$ with $\cR=\cup_i \cR_i$.
\end{itemize}


The next theorem provides several equivalent formulations for when $\cR$ preserves the measure $\mu$.
\begin{thm}\label{T:mp}
    Let $\cR$ be a discrete equivalence relation on a standard measure space $(X,\mu)$. The following are equivalent:
    \begin{enumerate}
        \item (Full group) Let $[\cR]$ be the group of all Borel automorphisms $\phi:X \to X$ such that $(x,\phi(x))\in \cR $ for every $x$. This is the {\bf full group} of $\cR$. $\cR$ preserves $\mu$ in the sense that $\phi_*\mu=\mu$ for every $\phi \in [\cR]$.
        \item (Mass transport principle) For every non-negative Borel map $F:\cR \to \R$, 
$$\int \sum_{y\in \cR} F(x,y)~d\mu(x) = \int \sum_{y\in \cR} F(y,x)~d\mu(x).$$
\item (Group action) There exists a countable group $\G$ with a measure-preserving action $\G \cc (X,\mu)$ such that if $\cR_\G=\{(x,gx):~x\in X, g\in \G\}$ is the orbit-equivalence relation then $\cR_\G=\cR$ mod $\mu$.
\item (Unimodularity) Define Borel measures $\mu_L, \mu_R$ on $\cR$ by 
\begin{align*}
    \mu_L(E) &=  \int \#\{y:~(x,y) \in E\} ~d\mu(x)\\
   \mu_R(E) &=  \int \#\{y:~(y,x) \in E\} ~d\mu(x).
\end{align*}
Then $\mu_L=\mu_R$.
    \end{enumerate}
\end{thm}

\begin{remark}
     Item (3) uses the Feldman-Moore Theorem \cite{MR0578656}. The rest are exercises.
\end{remark}

We say that $(X,\mu,\cR)$ is pmp, imp, or mp if the conditions above are satisfied and $\mu$ is a probability, infinite or arbitrary measure respectively.

In the case where $\cR$ is aperiodic, we get the following result, due to Slaman-Steel.

\begin{thm}[Marker Lemma]\label{marker}
    Let $\cR$ be an aperiodic discrete Borel equivalence relation on $X$. Then there exists a \textbf{vanishing sequence of markers} for $\cR$, i.e. there is a sequence $\{S_n\} \subset X$ of Borel sets such that 
    \begin{enumerate} \itemsep0em
        \item $S_0 \supseteq S_1 \supseteq S_2 \supseteq \cdots$,
        \item $\cap_n S_n = \emptyset$, and
        \item each $S_n$ meets every equivalence class of $\cR$. That is, each $S_n$ is a \textbf{complete section} for $\cR$. 
    \end{enumerate}
\end{thm}

In the case where $\mu$ is a standard probability measure we are able to use the Marker Lemma to find complete sections with arbitrarily small measure.
Note that for infinite measures, a vanishing sequence of markers does not necessarily correspond to a finite measure complete section. For example, if we have an aperiodic equivalence relation on $\R$ equipped with Lebesgue measure, the sequence $S_n = [n, \infty)$ is a vanishing sequence of markers but each $S_n$ has infinite measure. 

However, we can always find a finite measure section for an aperiodic equivalence relation $\cR$ as long as $(X, \mu)$ is $\sigma$-finite and a.e.\ ergodic component is non-atomic.

\begin{thm}\label{T:section}
    Let $(X,\mu)$ be a $\sigma$-finite infinite non-atomic measure space. Suppose $\cR \subset X \times X$ is a countable aperiodic $\mu$-quasi-invariant equivalence relation. Suppose a.e.\ ergodic component of $(\cR,\mu)$ is purely non-atomic. Then for every $\eps>0$ there exists a Borel section $S$ for the equivalence relation $\cR$ with $\mu(S) < \eps$.%
\end{thm}

\begin{proof}
    
    Let $\pi: (X, \mu) \to (Y,\nu)$ be the factor map of $(X,\mu)$ to the space $(Y, \nu)$ of ergodic components for the action of $\G$. Then for $\nu$-a.e.\ $y\in Y$, the action of $\G$ on  $(X_y, \cB_y,  \mu_y)$ is ergodic and $\mu_y$-invariant, where $X_y = \pi^{-1}(y)$, $\cB_y = \cB \cap X_y$, and $\mu_y$ the associated measure to $y$ for the disintegration of $\mu$ over $\nu$. That is, for each $y \in Y$, there exists a measure $\mu_y$ such that $\mu = \int_Y \mu_y d\nu(y)$ and the action of $\G$ on $(\pi^{-1}(y), \mu_y)$ is ergodic. Additionally, by assumption, $\mu_y$ is non-atomic for a.e.\ $y$.

    
    Let $f: Y \to [0, \infty)$ be a Borel function with $f(y) > 0$ for $\nu$-a.e.~$y$ and $\int f d\nu < \eps$, i.e.\ an everywhere positive $\rL^1$ function on $Y$. Such a function exists because $(Y,\nu)$ is $\sigma$-finite.

Define
$$S=\{x\in [0,\infty):~ \mu_{\pi(x)}([0,x])\le f(\pi(x))\}.$$
Because $f\circ \pi$ is Borel and the map $x\mapsto\mu_{\pi(x)}([0,x])$ is Borel, it follows that $S$ is Borel. Moreover, for $y\in Y$, $\mu_y(S)=f(y)$. This is because $S \cap \pi^{-1}(y)$ is an interval of the form $[0,t]$ where $t$ is the largest number with $\mu_{y}([0,t])\le f(y)$. So
$$\mu(S) = \int \mu_{y}(S)~d\nu(y) = \int f(y)~\dee\nu(y)<\eps.$$
Lastly, we observe that $S$ meets every equivalence class because it meets a.e.\ ergodic component $\pi^{-1}(y)$ in a positive measure set since $f$ is positive a.e.
\end{proof}




\begin{remark}
    Here we use the quasi-invariant version of the ergodic decomposition theorem, which can be found in \cite{MR1784210}. We can apply this result since any infinite measure is equivalent to some probability measure. That is, if the action of $\G \cc X$ is $\mu$-quasi-invariant, it will be $\nu$-quasi-invariant for any probability measure $\nu$ equivalent to $\mu$.
\end{remark}

\begin{remark}
    The requirement that a.e.\ ergodic component of $(\cR, \mu)$ is purely non-atomic is necessary.
    Let $\G$ be a countable group and consider the action $\G \cc \G \times \R$ by $g(h,x)=(gh,x)$. This action preserves the measure $c_\G \times \mathsf{Leb}$ and the orbit-equivalence relation is aperiodic.
    However, the ergodic components are the fibers $\G\times\{x\}$ for each $x\in \R$ equipped with counting measure $c_\G$.
    We can see that every complete section has infinite measure. This is because every complete section contains a subset of the form $\{(\pi(x),x):~x\in \R\}$ where $\pi:\R \to \G$ is a measurable and such a subset has infinite measure.
\end{remark}

\section{Spaces of measures}\label{S:spaces-of-measures}

The purpose of this section is to review convergence of measures; especially vague convergence in the context of locally compact spaces, and prove a version of the Portmanteau Theorem we will frequently use.

Let $(X,d)$ be a Polish space and $\Radon(X)$ be the set of Radon measures on $X$. Let
\begin{align*}
C(X)&=\{f: X \to \C:~ f \textrm{ is continuous}\},\\
\|f\| &= \sup_{x\in X} |f(x)|,\\
C_B(X)&=\{f\in C(X):~\|f\|<\infty \},\\
    C_c(X) &= \{f\in C(X):X \to \C:~f \textrm{ has compact support}\},\\
    C_0(X) &= \{f:X \to \C:~f \textrm{ vanishes at infinity}\}.
\end{align*}
The last condition means: for all $\eps>0$ there exists compact $K \subset X$ such that $|f(x)|<\eps$ for all $x\in X\setminus K$.

We also let $C^+(X), C_B^+(X),C_c^+(X),C_0^+(X)$ denote the non-negative functions in $C(X), \\ C_B(X),C_c(X),C_0(X)$ respectively. Observe $C_c(X) \subset C_0(X) \subset C_B(X) \subset C(X)$. Additionally, $C_B(X)$ is a Banach algebra with the norm $f\mapsto \|f\|$. The subspace $C_0(X)$ is closed in $C_B(X)$ and is therefore a Banach space itself. Note that the space $C_c(X)$ is not closed unless $X$ is compact. If $X$ is lcsc then $C_c(X)$ is norm dense in $C_0(X)$ by Urysohn's Lemma.

Let $(\mu_n)_{n=1}^\infty, \mu_\infty$ be Radon measures on $X$. We say $(\mu_n)_n$ converges to $\mu_\infty$
\begin{itemize}
    \item {\bf vaguely} if $\lim_{n\to\infty} \int f ~d\mu_n = \int f ~d\mu_\infty$ for all $f\in C_c^+(X)$;
      \item {\bf weak*} if  $\lim_{n\to\infty} \int f ~d\mu_n = \int f ~d\mu_\infty$ for all $f\in C_0^+(X)$;
    \item  {\bf weakly} if $\lim_{n\to\infty} \int f ~d\mu_n = \int f ~d\mu_\infty$ for all $f\in C_B^+(X)$.
\end{itemize}

\begin{remark}
We use non-negative test functions in the above definitions to ensure that the integral $\int f~d\mu_\infty\in [0,\infty]$ is well-defined in case $\mu_\infty$ is an infinite measure and $f$ has non-compact support.
\end{remark}

A subset $A \in \Si$ is a $\mu$-continuity set if $\mu(\partial A) = 0$, where $\partial A = \overline{A}\cap \overline{X \setminus A}$ is the topological boundary of $A$. In other words, $\partial A$ is the set of limits of sequences of points in $A$ which are also limits of sequences of points outside of $A$.

The Portmanteau theorem is well-known for probability measures, but our work also focuses on infinite measures. The following locally compact version is a specialization of a more general result \cite{MR2271177}.

\begin{thm}[Locally Compact Portmanteau theorem]\label{T:portmanteau}
    Let $Y$ be a locally compact second countable space with Borel $\s$-algebra $\Si$.  Let $(\mu_n)^\infty_{n=1}$, $\mu_\infty$ be measures on $Y$. Then the following are equivalent
        \begin{enumerate}\itemsep0em
\item $\mu_n$ converges vaguely to $\mu_\infty$ as $n\to\infty$;
\item $\lim_{n\to\infty} \mu_n(A)=\mu_\infty(A)$ for every relatively compact $\mu_\infty$-continuity set $A\subset Y$;
        \item 
        \begin{enumerate}
            \item 
        $\limsup_n \mu_n(F) \leq \mu_\infty(F)$ for all relatively compact closed sets $F \subset Y$, and
        \item $\liminf_n \mu_n(O) \geq \mu_\infty(O)$ for all relatively compact open $O \subset Y$. 
        \end{enumerate}
    \end{enumerate}
    \end{thm}

\begin{proof}
This is a direct consequence of the unbounded Portmanteau Theorem proven in \cite[Theorem 2.1]{MR2271177} applied to the 1-point compactification of $Y$. That is, we let $X=Y\cup \{x_0\}$ be the 1-point compactification of $Y$. We let $d$ be an arbitrary metric on $X$ inducing its topology. Such a metric exists because $Y$ is second countable.

Item (1) above is equivalent to item (iv) of \cite[Theorem 2.1]{MR2271177}, item (2) above is equivalent to item (ii) of \cite[Theorem 2.1]{MR2271177}, item (3a) is equivalent to item (vi-a) of \cite[Theorem 2.1]{MR2271177}, item (3b) is equivalent to item (vi-b) of \cite[Theorem 2.1]{MR2271177}. By taking , we see that items (3) and (4) are equivalent to each other.
\end{proof}

There is also a Prokhorov-type theorem for vague compactness from \cite{MR3642325}. Let $X$ be a separable and complete metric space, and let $\cS$ denote the class of measurable subsets, $\hat\cS$ the space of bounded subsets.
Let $\cM_X$ denote the space of locally finite measures on $X$ and $\hat{\cM_X}$ the space of bounded measures on $X$. Let $\cK$ denote the space of compact subsets of $X$. 
\begin{thm}
    The vague topology on $\cM_X$ is Polish with Borel $\sigma$-field $\cB_{\cM_X}$. Furthermore, a set $A \subset \cM_X$ is vaguely relatively compact iff \begin{enumerate}
        \item $\sup_{\mu\in A} \mu B < \infty, \quad B \in \hat\cS$,
        \item $\inf_{K\in\cK} \sup_{\mu \in A} \mu(B/K) = 0, \quad B \in \hat\cS$.
    \end{enumerate}
\end{thm}
\noindent In particular, this recovers Prokhorov's theorem in the case that (1) and (2) hold with $B = S$.

\bibliography{biblio}
\bibliographystyle{alpha}

\end{document}